\documentclass{amsart}
\usepackage{amssymb}
\usepackage{braket}
\usepackage{mathrsfs}
\usepackage{amsrefs}
\usepackage{ifthen}
\usepackage{graphicx}
\usepackage{tikz}
\usetikzlibrary{patterns}
\usepackage{comment} 
\usepackage{mleftright}
\usepackage{cleveref}

\theoremstyle{plain}
\newtheorem{theo}{Theorem}[section]
\crefname{theo}{Theorem}{Theorems}
\Crefname{theo}{Theorem}{Theorems}
\newtheorem{prop}[theo]{Proposition}
\crefname{prop}{Proposition}{Propositions}
\Crefname{prop}{Proposition}{Propositions}
\newtheorem{lem}[theo]{Lemma}
\crefname{lem}{Lemma}{Lemmas}
\Crefname{lem}{Lemma}{Lemmas}
\newtheorem{cor}[theo]{Corollary}
\crefname{cor}{Corollary}{Corollaries}
\Crefname{cor}{Corollary}{Corollaries}

\crefname{claim}{Claim}{Claims}
\Crefname{claim}{Claim}{Claims}

\crefname{property}{Property}{Properties}
\Crefname{property}{Property}{Properties}

\crefname{problem}{Problem}{Problems}
\Crefname{problem}{Problem}{Problems}

\theoremstyle{definition}
\newtheorem{defi}[theo]{Definition}
\crefname{defi}{Definition}{Definitions}
\Crefname{defi}{Definition}{Definitions}

\crefname{notation}{Notation}{Notations}
\Crefname{notation}{Notation}{Notations}
\newtheorem {convention}[theo]{Convention}
\crefname{convention}{Convention}{Conventions}
\Crefname{convention}{Convention}{Conventions}
\newtheorem{cond}[theo]{Condition}
\crefname{cond}{Condition}{Conditions}
\Crefname{cond}{Condition}{Conditions}

\crefname{assum}{Assumption}{Assumptions}
\Crefname{assum}{Assumption}{Assumptions}

\theoremstyle{remark}
\newtheorem{rem}[theo]{Remark}
\crefname{rem}{Remark}{Remarks}
\Crefname{rem}{Remark}{Remarks}
\newtheorem{ex}[theo]{Example}
\crefname{ex}{Example}{Examples}
\Crefname{ex}{Example}{Examples}

\crefname{section}{Section}{Sections}
\Crefname{section}{Section}{Sections}
\crefname{subsection}{Subsection}{Subsections}
\Crefname{subsection}{Subsection}{Subsections}
\crefname{figure}{Figure}{Figures}
\Crefname{figure}{Figure}{Figures}

\newtheorem*{acknowledgement}{Acknowledgement}

\newcommand{\Z}{\mathbb{Z}}
\newcommand{\R}{\mathbb{R}}
\newcommand{\C}{\mathbb{C}}
\newcommand{\CP}{\mathbb{CP}}
\newcommand{\spc}{\mathrm{spin}^c}
\newcommand{\Area}{\mathrm{Area}}
\newcommand{\Hyp}{\mathrm{Hyp}}

\newcommand{\sign}{\mathop{\mathrm{sign}}\nolimits}

\newcommand{\calA}{{\mathcal A}}
\newcommand{\calB}{{\mathcal B}}
\newcommand{\calC}{{\mathcal C}}

\newcommand{\calG}{{\mathcal G}}
\newcommand{\calH}{{\mathcal H}}
\newcommand{\calK}{{\mathcal K}}
\newcommand{\calL}{{\mathcal L}}
\newcommand{\calM}{{\mathcal M}}

\newcommand{\calW}{{\mathcal W}}

\newcommand{\scrH}{{\mathscr H}}
\newcommand{\scrK}{{\mathscr K}}

\newcommand{\scrS}{{\mathscr S}}

\newcommand{\fraks}{\mathfrak{s}}
\newcommand{\frakt}{\mathfrak{t}}

\newcommand{\Met}{\mathrm{Met}}

\newcommand{\circPi}{\mathring{\Pi}}

\newcommand{\Si}{\Sigma}
\newcommand{\si}{\sigma}
\newcommand{\simin}[1]{\si_{\rm min}(#1)}
\newcommand{\simax}[1]{\si_{\rm max}(#1)}
\newcommand{\siSi}{\si \setminus \{\Sigma\}}
\newcommand{\Cy}{\mathop{\mathrm{Cy}}\nolimits}
\newcommand{\barCy}{\overline{\mathrm{Cy}}}
\newcommand{\Cystd}{\mathop{\mathrm{Cy}}_{\rm std}\nolimits}
\newcommand{\barCystd}{\overline{\mathrm{Cy}}_{\rm std}}
\newcommand{\Bd}{\mathop{\mathrm{Bd}}\nolimits}
\newcommand{\extdel}{{\rm ext}\partial}
\newcommand{\frakextdel}{{\mathfrak e}{\mathfrak x}{\mathfrak t}\partial}
\newcommand{\vp}{\varphi}
\newcommand{\inc}{\hookrightarrow}

\newcommand{\lb}{\left<}
\newcommand{\rb}{\right>}

\newcommand{\del}{\partial}
\newcommand{\codel}{\delta}

\newcommand{\im}{\mathop{\mathrm{Im}}\nolimits}

\newcommand{\dom}{\mathop{\mathrm{dom}}\nolimits}

\newcommand{\Conv}{\mathop{\mathrm{Conv}}\nolimits}
\newcommand{\SWinv}{\mathop{\mathrm{SW}}\nolimits}
\newcommand{\SWcoh}{\mathop{{\mathbb S}{\mathbb W}}\nolimits}
\newcommand{\SWcoch}[1][{{}}]{\mathop{{\mathcal S}{\mathcal W}_#1}\nolimits}
\newcommand{\SW}{Seiberg--Witten }

\newcommand{\id}{\mathrm{id}}

\newcommand{\bvec}[1]{\boldsymbol{#1}}
\newcommand{\bep}{\boldsymbol{\epsilon}}
\newcommand{\bSi}{\boldsymbol{\Sigma}}
\newcommand{\bS}{\boldsymbol{S}}
\newcommand{\Signk}{\mathrm{Sign}^{k}}
\newcommand{\Euler}{\mathop{\chi}\nolimits}
\newcommand{\genus}{\mathop{g}\nolimits}
\newcommand{\Map}{\mathop{{\rm Map}}\nolimits}

\title[A cohomological Seiberg--Witten invariant]{A cohomological Seiberg--Witten invariant emerging from the adjunction inequality}

\author{Hokuto Konno}
\address{Graduate School of Mathematical Sciences, the University of Tokyo, 3-8-1 Komaba, Meguro, Tokyo 153-8914, Japan}
\email{konno@ms.u-tokyo.ac.jp}
\date{}

\begin{document}

\maketitle

\begin{abstract}{}
We construct an invariant of closed $\spc$ $4$-manifolds using families of \SW equations.
This invariant is formulated as a cohomology class on a certain abstract simplicial complex consisting of embedded surfaces of a $4$-manifold.
We also give examples of $4$-manifolds which admit positive scalar curvature metrics and for which this invariant does not vanish.
This non-vanishing result of our invariant provides a new class of adjunction-type genus constraints on configurations of embedded surfaces in a $4$-manifold whose \SW invariant vanishes.
\end{abstract}

\tableofcontents

\section{Introduction}

In this paper we construct an invariant of closed $\spc$ $4$-manifolds.
This invariant is defined using families of \SW equations and formulated as a cohomology class on a certain abstract simplicial complex corresponding to a given $\spc$ $4$-manifold.
This simplicial complex encodes information about mutual intersections of embedded surfaces violating the adjunction inequalities with respect to the given $\spc$ structure.
We also give examples of $4$-manifolds which admit positive scalar curvature metrics and for which the invariant does not vanish.
This non-vanishing result of our invariant provides a new class of adjunction-type genus constraints on configurations of embedded surfaces in a $4$-manifold whose \SW invariant vanishes, for example, $k\CP^{2}\#l(-\CP^{2})$ for some $k, l > 3$.

Starting with Witten~\cite{MR1306021}, the \SW invariant of smooth $\spc$ $4$-manifolds has occupied an important place in $4$-dimensional geometry and topology.
This invariant is defined, roughly speaking, by counting the points of the moduli space of the solutions to the \SW equations.
For a given $\spc$ structure $\fraks$, the \SW invariant is defined to be trivial if the formal dimension $d(\fraks)$, given by the formula $d(\fraks)= (c_1(\fraks)^2 - 2\Euler(X) -3\sign(X))/4$, of the moduli space of the \SW equations with respect to $\fraks$ is negative.
However, even in the case that $d(\fraks)<0$, one can fruitfully consider a counting argument for the parameterized moduli space of a suitable family of \SW equations.
(This type of ideas has been suggested in Donaldson~\cite{MR1171888,MR1339810}.)
Indeed, the formal dimension of the moduli space of a family of \SW equations parameterized by a $k$-dimensional space is $d(\fraks) + k$, which is non-negative for $k \geq -d(\fraks)$.
Based on this idea, Ruberman~\cite{MR1671187, MR1734421, MR1874146} has defined several invariants of diffeomorphisms on $4$-manifolds using moduli spaces parameterized by $1$-dimensional spaces and Li--Liu~\cite{MR1868921} have defined an invariant of families of $\spc$ $4$-manifolds.
The main aim of this paper is to construct an invariant of $\spc$ $4$-manifolds using families of \SW equations;  we show that this invariant is non-trivial for some $\spc$ $4$-manifold $(X, \fraks)$ with $d(\fraks) < 0$.
Note that we construct an invariant of a {\it single} $4$-manifold, while Li--Liu's is of a {\it family} of $4$-manifolds.

The invariant that we construct in this paper also has other roots in a classical problem in $4$-dimensional topology: the genus bound problem.
The \SW theory has been used to attack this problem.
We first recall the celebrated paper due to Kronheimer--Mrowka~\cite{MR1306022} on the Thom conjecture.
By arguments in the paper, for a given $\spc$ $4$-manifold $(X, \fraks)$, one can show that a strong lower bound on genus called the adjunction inequality (with respect to $\fraks$) holds for an embedded surface in $X$ if the \SW invariant of $(X, \fraks)$ does not vanish and the surface has non-negative self-intersection number.
However, for a surface embedded in a $\spc$ $4$-manifold whose \SW invariant vanishes, the adjunction inequality does not hold in general.
(For example, see Nouh~\cite{MR3159966} for the $4$-manifold $\CP^{2}\#\CP^{2}$.)
Although one cannot expect any systematic result on the adjunction inequality for a {\it single} embedded surface in such a $\spc$ $4$-manifold by this reason, one can do for {\it configurations} of embedded surfaces.
The first result in this direction is due to Strle~\cite{MR2064429}.
He has given certain constraints on configurations of embedded surfaces with positive self-intersection number.
More precisely, Strle has considered not only a single surface but also several embedded surfaces, and he has shown that if they are mutually disjoint, then the adjunction inequality holds for at least one of them under some conditions.
Inspired by this Strle's work, the author~\cite{Konno} has given constraints on configurations of embedded surfaces with zero self-intersection number.
We note that both of Strle's and the author's results are valid for $\spc$ $4$-manifolds whose \SW invariants vanish.
To obtain the constraints on configurations, the author used a higher-dimensional family of \SW equations on a $4$-manifold.
This proof suggests that one might expect some systematic way to relate families of \SW equations and adjunction-type genus constraints on configurations.

In this paper we introduce an abstract simplicial complex associated with a $\spc$ $4$-manifold $(X, \fraks)$ and define an invariant of $(X, \fraks)$ as a simplicial cohomology class on this simplicial complex.
The simplicial complex encodes information about mutual intersections of all embedded surfaces violating the adjunction inequalities with respect to $\fraks$ and having zero self-intersection number.
In order to define the simplicial complex, we start with the following ``ambient" simplicial complex $\bar{\calK}$, which Mikio Furuta introduced to the author.
Throughout this paper, we mean by a {\it surface} a smooth oriented closed connected $2$-dimensional manifold.

\begin{defi}
Let $X$ be a smooth oriented closed connected $4$-manifold.
We define an abstract simplicial complex $\bar{\calK} = \bar{\calK}(X)$ as follows: 
\begin{itemize}
\item The vertices of $\bar{\calK}$ are defined as smoothly embedded surfaces (i.e. smooth oriented closed $2$-dimensional submanifolds of $X$) with self-intersection number zero.
We denote by $V(\bar{\calK})$ the set of vertices of $\bar{\calK}$.
\item For $n \geq 1$, a collection of $(n+1)$ vertices $\Si_0, \ldots, \Si_n \in V(\bar{\calK})$ spans an $n$-simplex if and only if $\Si_0, \ldots, \Si_n$ are mutually disjoint.
\end{itemize}
We call $\bar{\calK}$ the {\it complex of surfaces} of $X$.
\end{defi}

The complex of surfaces is a $4$-dimensional analog of the complex of curves due to Harvey~\cite{MR624817} in $2$-dimensional topology.
Note that, in the definition of the complex of surfaces above, we do not consider isotopy classes of embeddings of surfaces as in the definition of the complex of curves.
(See \cref{rem: impossibility on isotopy classes} on this point.)

We topologize $\bar{\calK}(X)$ as a simplicial complex, namely with the weak topology.
In fact, $\bar{\calK}(X)$ is contractible for any $X$
(\cref{prop : any complex of surface is contractible}), and therefore one cannot use any homotopical information of $\bar{\calK}(X)$ to define a non-trivial invariant of $X$.
We can, however, find a suitable subcomplex of $\bar{\calK}$ which is not homotopically trivial in general and will be used as parameter space (up to homotopy) of families of \SW equations.
The subcomplex is defined considering one of special phenomena in $4$-dimensional topology, namely the adjunction inequality.
This subcomplex is the main ingredient to study higher-dimensional families of \SW equations in this paper.
We will denote by $\genus(\Si)$ the genus of a surface $\Si$.
Set
\[
\Euler^-(\Si) := \max\{-\Euler(\Si),0\},
\]
where $\Euler(\Si) = 2 - 2\genus(\Si)$ is the Euler characteristic of $\Si$.

\begin{defi}
Let $\fraks$ be a $\spc$ structure on $X$.
We define
$\calK = \calK(X, \fraks)$ 
as the full subcomplex of $\bar{\calK}(X)$ spanned by
\[
V(\calK) := 
\Set{ \Si \in V(\bar{\calK}) | \Euler^-(\Si) < |c_1(\fraks) \cdot [\Si]| }.
\]
Namely a collection of vertices $\Si_0, \ldots, \Si_n \in V(\calK)$ spans an $n$-simplex of $\calK$ if and only if it does in $\bar{\calK}$.
We call $\calK$ the {\it adjunction complex of surfaces} of $(X, \fraks)$.
\end{defi}

We note a similarity between the simplicial complex $\calK$ above and the Kakimizu complex~\cite{MR1177053} in $3$-dimensional topology and knot theory: in the definitions of both simplicial complexes one focuses one's attention on surfaces of low genus.

We now introduce our invariant, which will be defined in \cref{subsection: Definition of the invariant}, in the most basic setting.
Let $(X, \fraks)$ be a smooth oriented closed $\spc$ $4$-manifold equipped with a homology orientation.
Here a homology orientation means an orientation of the vector space $H^1(X;\R) \oplus H^+(X;\R)$, where $H^+(X;\R)$ is a maximal positive-definite subspace of $H^2(X;\R)$ with respect to the intersection form of $X$.
We denote by $b^+(X)$ the dimension of $H^+(X;\R)$.
Let $n$ be the integer given by $d(\fraks) = -(n+1)$.
Assume that $n \geq 0$ and that
$b^+(X) \geq n+3$.
Under these assumptions, we shall construct a cohomology class
\[
\SWcoh(X, \fraks) \in H^n(\calK(X, \fraks) ;\Z)
\]
using $(n+1)$-dimensional families of \SW equations.
We will show that this cohomology class $\SWcoh(X, \fraks)$ is an invariant of a $\spc$ $4$-manifold $(X, \fraks)$.
We call $\SWcoh(X, \fraks)$ the {\it cohomological \SW invariant associated with the adjunction complex of surfaces}.
(In fact, we can extend the definition of the cohomological \SW invariant to more general $\spc$ structures.
See \cref{subsection: Definition of the invariant for general spinc structures} on this point.)
Since $H^{\ast}(\calK(X, \fraks) ;\Z)$, to which $\SWcoh(X, \fraks)$ belongs, is also a diffeomorphism invariant of $(X, \fraks)$, the calculation of $\SWcoh(X, \fraks)$ seems hard.
Nevertheless, we can show that there are examples of $(X, \fraks)$ with $\SWcoh(X, \fraks) \neq 0$.
In particular, we obtain examples of $(X, \fraks)$ such that the homotopy type of $\calK(X, \fraks)$ is non-trivial.

A remarkable point is that
\[
\SWcoh(k\CP^{2} \# l(-\CP^{2}), \fraks) \neq 0
\]
holds for some $k, l> 3$ and some $\spc$ structure $\fraks$ on $k\CP^{2} \# l(-\CP^{2})$.
(See \cref{cor: non-vanishing on psc manifold}.)
Note that the usual \SW invariant vanishes for any $\spc$ structure on it and the Donaldson invariant of this manifold also vanishes.
Furthermore, the refinement of the \SW invariant called the Bauer--Furuta invariant~\cite{MR2025298} of this $4$-manifold also vanishes.
Indeed, this $4$-manifold admits a metric with positive scalar curvature, and the Bauer--Furuta invariant vanishes for $4$-manifolds admitting a positive scalar curvature metric as the \SW invariant does.
To the best of the author's knowledge, $\SWcoh(X, \fraks)$ is the first invariant which is defined using the \SW theory and which is non-trivial for some $4$-manifolds admitting positive scalar curvature metrics.

The non-vanishing of the cohomological invariant yields a concrete geometric result connected with the adjunction inequality.
We will prove in fact not only that $\SWcoh(X, \fraks) \neq 0$ but also that the evaluation $\lb \SWcoh(X, \fraks), \ \cdot\ \rb : H_{\ast}(\calK(X, \fraks) ;\Z) \to \Z$ is non-zero for some $(X, \fraks)$.
We therefore also obtain a non-trivial homology class in $H_{\ast}(\calK(X, \fraks) ;\Z)$;
we can concretely give such a class.
We will see that each non-trivial homology class in $H_{\ast}(\calK(X, \fraks) ;\Z)$ provides some adjunction-type genus constraints on configurations: the adjunction inequality for some embedded surface under a certain condition on geometric intersections with other embedded surfaces.
Such constraints on configurations will be  obtained for $4$-manifolds whose \SW and Bauer--Furuta invariants vanish as explained above.
This point is similar to Strle~\cite{MR2064429} and the author's result~\cite{Konno}.
On the other hand, we will also show that {\it one} non-trivial homology class in $H_{\ast}(\calK(X, \fraks) ;\Z)$ provides constraints on {\it infinitely many} configurations of embedded surfaces.
This is a phenomenon newly found through our use of the simplicial homology theory.

We show this non-vanishing result using the following two key ingredients:
the first one is a combination of wall-crossing and gluing technique originally due to Ruberman~\cite{MR1671187, MR1734421, MR1874146}, and the second is a description of higher-dimensional wall-crossing phenomena in terms of embedded surfaces given in \cite{Konno} by the author.
Combining these two tools, we can obtain a solution to the \SW equations used to show the adjunction inequality on a $4$-manifold whose \SW invariant vanishes.

We finish off this introduction with an outline of the contents of this paper.
\Cref{section: Family of Riemannian metrics,section: Construction of the cohomological invariant} are devoted to constructing the invariant $\SWcoh$.
In \cref{section: Family of Riemannian metrics} we construct certain families of Riemannian metrics on a given $4$-manifold obtained by stretching neighborhoods of embedded surfaces.
In \cref{section: Construction of the cohomological invariant} we consider the moduli spaces of the \SW equations parameterized by the families of Riemannian metrics constructed in \cref{section: Family of Riemannian metrics}, and define $\SWcoh$ by counting the points of the parameterized moduli spaces.
In \cref{section: Non-vanishing results for the cohomological invariant} we prove that the invariants $\SWcoh(X, \fraks)$ are non-trivial for some $\spc$ $4$-manifolds $(X, \fraks)$ as described above, and give a geometric application of this non-vanishing result.
In particular, in \cref{subsection: Application to bounds on the complexity of configurations of surfaces} we give bounds on the complexity of configurations of surfaces using this non-vanishing, along an idea of Strle~\cite{MR2064429}.

\begin{acknowledgement}
The author would like to express his deep gratitude to Mikio Furuta for helpful suggestions, in particular the definition of the complex of surfaces, and for continued encouragement during this work.
The author would also like to express his appreciation to Nobuhiro Nakamura for telling him Ruberman's work on invariants of diffeomorphisms using $1$-parameter family.
The author also wishes to thank Kouichi Yasui for comment on \cref{ex: wall-crossing collection of surfaces}.
A note in \cref{rem: comment on blowup formula} on the blowup formula is due to him.
The author also thanks Genki Sato for a discussion on the simplicial complex theory.
Finally, the author would like to express his deep gratitude to an anonymous referee for many valuable comments and suggestions.
The author was supported by JSPS KAKENHI Grant Numbers 16J05569, 17H06461, 19K23412, 21K13785, and the Program for Leading Graduate Schools, MEXT, Japan.
\end{acknowledgement}

\section{Construction of the invariant I: Family of Riemannian metrics}

\label{section: Family of Riemannian metrics}

In this section we construct certain families of Riemannian metrics on a given $4$-manifold obtained by stretching neighborhoods of embedded surfaces and study some properties of these families.
These families will provide foundations of the construction of our cohomological invariant.

\subsection{Notation on simplicial complexes}

\label{subsection: Notation on simplicial complexes}

We first summarize some notation and convention on simplicial complexes.
Let $K$ be a simplicial complex.

\begin{itemize}
\item For $n \geq 0$, we will write $\Delta^n$ for the standard $n$-simplex that is defined by
\[
\Delta^n = \Conv\Set{e_0, \ldots, e_n}.
\]
Here $\{e_0, \ldots, e_n \} = \{{}^t(1,0, \ldots, 0), \ldots, {}^t(0, \ldots,0, 1)\}$  is the standard basis of $\R^{n+1}$, and $\Conv\{A\}$ is the convex hull of a subset $A \subset \R^{n+1}$.
\item We will denote by $V(K)$ the set of vertices of $K$.
\item We will denote by $S(K)$ the set of simplices of $K$ and by $S_n(K)$ the set of $n$-simplices of $K$. 
\item If $v_0, \ldots, v_n \in V(K)$ span an $n$-simplex $\si \in S(K)$, we regard $\si$ as the set $\si = \{v_0, \ldots, v_n\}$.
If $v$ is a vertex of $\si$, we will write $v \in \si$.
\item For an $n$-simplex $\si = \{v_0, \ldots, v_n\}$, we will denote by $\lb \si \rb = \lb v_0, \ldots, v_n \rb$ an oriented simplex.
The simplex with the opposite orientation is denoted by $-\lb \si \rb$.
\item If $\tau$ is a face of $\si$, we will write $\tau \prec \si$.
\item We will denote by $|K| = \bigcup_{\si \in S(K)} |\si|$ the geometric realization of $K$.
\item For a simplex $\si \in S(K)$, we can define a simplicial complex $\scrK(\si)$ by
\[
S(\scrK(\si)) = \Set{ \tau | \tau \prec \si }.
\]
We will also simply write this simplicial complex $\si$ when no confusion can arise.
The geometric realization of $\scrK(\si)$ can be identified with $|\si| \subset |K|$.
\item We will denote by $\Bd(K)$ the barycentric subdivision of $K$.
An $n$-simplex of $\Bd(K)$ is given by a set $\{ \si_0, \ldots, \si_n \}$, where $\si_0 \precneqq \cdots \precneqq \si_n$ is an increasing sequence in $S(K)$.
We identify $|\Bd(K)|$ with $|K|$ as usual.
For a simplex  $s = \{ \si_0, \ldots, \si_n \} \in S(\Bd(K))$ with $\si_0 \precneqq \cdots \precneqq \si_n$, set $\simax{s} := \si_n$ and $\simin{s} := \si_0$.
If readers are not familiar with this combinatorial description of the barycentric subdivision, 
see \cite{MR755006} to connect this with a geometric description.
\item We will denote by
\[
C_\ast(K) = C_\ast(K ; \Z),\quad C^\ast(K) = C^\ast(K ; \Z)
\]
the simplicial chain complex and the cochain complex with integer coefficient respectively.
We will denote by
\[
\del : C_\ast(K) \to C_{\ast-1}(K),\quad \codel : C^\ast(K) \to C^{\ast+1}(K)
\]
the boundary operator and the coboundary operator respectively.
\end{itemize}

\subsection{Outline of the construction of the cohomological invariant}

\label{subsection: Outline of the construction of the cohomological invariant}

Before starting technical details of the construction of the cohomological invariant, in this subsection we shall give a sketch of the whole construction and explain technical issues which will be contended with in subsequent subsections.

Let $(X, \fraks)$ be a smooth oriented closed $\spc$ $4$-manifold with a homology orientation.
Let $n$ be an integer defined by $d(\fraks) = -(n+1)$, where $d(\fraks)$ denotes the formal dimension of the (unparameterized) moduli space of the \SW equations.
Under the assumptions that
$b^+(X) \geq n+3$ and that $n \geq 0$, using families of \SW equations, we shall construct a cohomology class depending only on $(X, \fraks)$, which we call the {\it cohomological Seiberg--Witten invariant}:
\[
\SWcoh(X, \fraks) \in H^n(\calK(X, \fraks) ;\Z).
\]

Our construction of $\SWcoh(X, \fraks)$ consists of the following three steps:
\begin{description}
\item[Step~1] Construct a cochain on $\calK = \calK(X, \fraks)$ using a non-topological data, such as a metric.
\item[Step~2] Show that the cochain is cocycle.
\item[Step~3] Show that the cohomology class of the cocycle is independent of the choice made in the first step.
\end{description}
Let us explain these three steps in more detail.
In this explanation, we assume that $n = 1$ for simplicity.
The first step is to construct a cochain 
\[
\SWcoch(X, \fraks, \calA) \in C^1(\calK),
\]
where $\calA$ is a certain additional data:
this cochain depends on not only $(X, \fraks)$ but also other various choices.
The first choice used to construct $\SWcoch(X, \fraks, \calA)$ is a Riemannian metric $g$ on $X$.
To make a $1$-cochain, we have to give an integer corresponding to  an oriented $1$-simplex $\lb \si \rb = \lb \Si_0, \Si_1 \rb$.
Since $\Si_0$ and $\Si_1$ are disjoint, we can stretch some neighborhoods of $\Si_0$ and $\Si_1$ independently from the metric $g$.
Thus we obtain a $2$-parameter family of metrics.
Strictly speaking, we have to modify $g$ to be cylindrical near the sphere bundles of the normal bundles of $\Si_0$ and $\Si_1$ to stretch their neighborhoods.
This will cause the main technical difficulty, but, for the moment, we describe the whole construction as if we could forget the problem.
Since $\Si_0$ and $\Si_1$ violate the adjunction inequalities, we can show that the moduli space of the \SW equations is empty for such a stretched metric by using a quantitative version of Kronheimer--Mrowka~\cite{MR1306022}'s argument.
We describe the $2$-parameter family of metrics in \cref{figure : 2-parameter family obtained by stretching}.
Here $\Si_i$-direction means the stretching parameter for $\Si_i$.
The coordinate $(R_0, R_1)$ means the metric obtained by stretching the neighborhood of $\Si_i$  by the length $R_i$ $(i = 0, 1)$ from the initial metric $g$.  
For sufficiently large $R>0$, any metric on the codimension $1$ face $F$ of the ``$2$-simplex" in \cref{figure : 2-parameter family obtained by stretching} is stretched for at least one of the neighborhood of $\Si_0$ and one of $\Si_1$.
Therefore the moduli space on the face $F$ is empty.
Let us perturb the \SW equations on the ``$2$-simplex" except for the face $A$, and consider the moduli space of the perturbed \SW equations parameterized by the $2$-simplex.
Then, on any face of codimension $\geq1$, the parameterized moduli space is empty since we have assumed that $d(\fraks) = -2$.
We can count the points of the parameterized moduli space on the $2$-simplex and obtain an integer $\SWinv(\lb \si \rb)$.
(Here we have to use the given homology orientation and the orientation of the simplex $\lb \si \rb$.)
Therefore we obtain the $1$-cochain
\[
\SWcoch(X, \fraks, \calA) : C_1(\calK) \to \Z
\]
defined by
\[
 \lb \si \rb \mapsto \SWinv(\lb \si \rb).
\]
\begin{figure}
\begin{center}
\begin{tikzpicture}
[xscale = 0.7, yscale = 0.7]
\draw [->] [thick](0, 0) -- (3, 0);
\draw [->] [thick](0, 0) -- (0, 3);
\draw [ultra thick] (2.5, 0) -- (0, 2.5);
\draw(0,3.5) node {$\Si_1$-direction};
\draw(4.6,0) node {$\Si_0$-direction};
\draw(-0.3,2.5) node {$R$};
\draw(2.5,-0.3) node {$R$};
\draw(-0.3,-0.3) node {$g$};
\draw(1.9,1.3) node {$A$};
\fill(0,0) circle (2.8pt);
\fill [gray, opacity=.3] (0, 0) -- (2.5,0) -- (0,2.5) -- (0, 0);
\end{tikzpicture}
\end{center}
\caption{$2$-parameter family obtained from stretchings}
\label{figure : 2-parameter family obtained by stretching}
\end{figure}

The second step is to show that the cochain $\SWcoch(X, \fraks, \calA)$ is cocycle.
To do this, let $\lb \si \rb = \lb \Si_0 , \Si_1, \Si_2 \rb$ be a $2$-simplex of $\calK$.
We will see that $\SWcoch(X, \fraks, \calA)(\del \lb \si \rb) = 0$.
Write 
\begin{align}
\lb \tau_0 \rb = \lb \Si_{1}, \Si_{2} \rb,\ \lb \tau_1 \rb = \lb \Si_{0}, \Si_{2} \rb,\ \lb \si_2 \rb = \lb \Si_{0}, \Si_{1} \rb,
\label{eq : 1-simplices come from boundary of a 2-simplex}
\end{align}
then we have
$\del \lb \si \rb = \sum_{i = 0}^{2} (-1)^i \lb \tau_i \rb$.
We describe the $3$-parameter family of metrics in \cref{figure : basic argument to show SWcoch is cocycle} obtained from the stretching neighborhoods of $\Si_0$, $\Si_1$, and $\Si_2$.
Let $\bvec{v}_0$, $\bvec{v}_1$, and $\bvec{v}_2$ denote the coordinates $(R, 0, 0)$, $(0, R, 0)$, and $(0, 0, R)$ in \cref{figure : basic argument to show SWcoch is cocycle} respectively.
The integer $\SWcoch(X, \fraks, \calA)(\del \lb \si \rb) = \sum_{i = 0}^{2} (-1)^i \SWinv(\lb \tau_i \rb)$ is the sum of the counting the points of the moduli space of  the perturbed equations on codimension 1 faces of the ``$3$-simplex "given by
\[
\{g, \bvec{v}_0, \bvec{v}_1\},\ \{g, \bvec{v}_1, \bvec{v}_2\},\ \{g, \bvec{v}_2, \bvec{v}_0\}
\]
in \cref{figure : basic argument to show SWcoch is cocycle}.
Here we denote the face spanned by $\{v, v', v''\}$ briefly by $\{v, v', v''\}$, and we identify $g$ with the coordinate $(0,0,0)$.
Any two of these faces meet along a $1$-simplex, and thus these faces form a $2$-parameter family of metrics homeomorphic to $D^2$.
This 2-parameter family can be continuously deformed to the face $\{\bvec{v}_0, \bvec{v}_1, \bvec{v}_2\}$ fixing the all $1$-simplices corresponding to stretched metrics: 
\[
\{\bvec{v}_0, \bvec{v}_1\},\ \{\bvec{v}_1, \bvec{v}_2\},\ \{\bvec{v}_2, \bvec{v}_0\}.
\]
Since we have assumed that $b^+(X) \geq 3 + 1=4$, in this deformation of $2$-parameter families along a $3$-parameter family, we may assume that there is no reducible.
Note that, on the face $\{\bvec{v}_0, \bvec{v}_1, \bvec{v}_2\}$, all metrics are stretched for at least one of $\Si_0$, $\Si_1$, and $\Si_2$.
Thus the parameterized moduli space on this face is empty.
Therefore we have $\SWcoch(X, \fraks, \calA)(\del \lb \si \rb) = 0$ using the argument by cobordisms.
\begin{figure}
\begin{center}
\begin{tikzpicture}
[xscale = 0.7, yscale = 0.7]
\draw [->] [thick](0, 0) -- (3, 0);
\draw [->] [thick](0, 0) -- (0, 3);
\draw [->] [thick](0, 0) -- (-1.5, -1.5);
\draw [ultra thick] (2.5, 0) -- (0, 2.5);
\draw [ultra thick] (-1.2, -1.2) -- (0, 2.5);
\draw [ultra thick] (2.5, 0) -- (-1.2, -1.2);
\draw(4.6,0) node {$\Si_0$-direction};
\draw(0,3.5) node {$\Si_1$-direction};
\draw(-1.4,-2) node {$\Si_2$-direction};
\draw(-0.3,2.5) node {$R$};
\draw(2.5,-0.3) node {$R$};
\draw(-1.5,-1) node {$R$};
\draw(-0.3,0) node {$g$};
\fill(0,0) circle (2.8pt);
\fill [gray, opacity=.2] (2.5,0) -- (0,2.5) -- (0, 0);
\fill [gray, opacity=.4] (0,0) -- (0,2.5) -- (-1.2, -1.2);
\fill [gray, opacity=.3] (2.5,0) -- (0,0) -- (-1.2, -1.2);
\end{tikzpicture}
\end{center}
\caption{Argument to show that $\SWcoch(X, \fraks, \calA)$ is a cocycle}
\label{figure : basic argument to show SWcoch is cocycle}
\end{figure}

The third step is to show that the all ambiguities arising from choices used to construct the cocycle $\SWcoch(X, \fraks, \calA)$ are absorbed into coboundaries.
From this, we can obtain the well-defined cohomology class $\SWcoh(X, \fraks) := [\SWcoch(X, \fraks, \calA)]$.
To do this, let us take two choices $\calA_{0}$ and $\calA_{1}$ (for example, metrics $g_0$ and $g_1$) and make two cocycles $\SWcoch[0] := \SWcoch(X, \fraks, \calA_{0})$ and $\SWcoch[1] := \SWcoch(X, \fraks, \calA_{1})$ using these two choices respectively.
By connecting $g_0$ with $g_1$ by a path in the space of metrics, we obtain the family described in \cref{figure : the definitions of SWcochzo}.
(Strictly speaking, to define the path between $R_0$ and $R_1$ in suitable sense, we need a quantitative version of an argument in Kronheimer--Mrowka~\cite{MR1306022}.)
We define a $0$-cochain $\SWcoch[{{0,1}}] \in C^0(\calK)$ by
\[
\lb \Si \rb \mapsto \SWinv(\lb \Si \rb),
\]
where $\SWinv(\lb \Si \rb) \in \Z$ is the counting on the square in \cref{figure : the definitions of SWcochzo} under suitable orientation.
Let us take a $1$-simplex $\lb \si \rb = \lb \Si_0, \Si_1 \rb$.
For this $1$-simplex, we obtain the family described in \cref{figure : basic argument to consider the ambiguities}.
Let $\bvec{v}_0^i$ and $\bvec{v}_1^i$ denote $(R_i, 0)$ and $(0, R_i)$ in \cref{figure : basic argument to consider the ambiguities} respectively.
The integer $\SWcoch[0](\lb \si \rb) - \SWcoch[1](\lb \si \rb)$ is the counting on the two triangles $\{g_0, \bvec{v}_0^0, \bvec{v}_1^0\}$ and $\{g_0, \bvec{v}_0^1, \bvec{v}_1^1\}$ under suitable signs.
We can see that the moduli space on the square $\{\bvec{v}_0^0, \bvec{v}_1^0, \bvec{v}_0^1, \bvec{v}_1^1\}$ is empty using a quantitative version of Kronheimer--Mrowka's argument again.
Hence we deduce that the total sum of the counting on two triangles $\{g_0, \bvec{v}_0^0, \bvec{v}_1^0\}$, $\{g_0, \bvec{v}_0^1, \bvec{v}_1^1\}$ and two squares $\{g_0, g_1, \bvec{v}_0^1, \bvec{v}_0^0\}$, $\{g_0, g_1, \bvec{v}_1^1, \bvec{v}_1^0\}$ is zero using the argument by cobordisms.
Namely, we have $\SWcoch[0](\lb \si \rb) - \SWcoch[1](\lb \si \rb) = \pm  \delta\SWcoch[{{0,1}}](\lb \si \rb)$, since
\[
 \SWinv(\lb \Si_1 \rb) - \SWinv(\lb \Si_0 \rb)=  \SWcoch[{{0,1}}](\lb \Si_1 \rb - \lb \Si_0 \rb) = \SWcoch[{{0,1}}](\del \lb \si \rb) = \delta\SWcoch[{{0,1}}](\lb \si \rb).
\]
(Here $\pm$ means the suitable sign.)
Thus we obtain $\SWcoch[0] - \SWcoch[1] = \pm  \delta\SWcoch[{{0,1}}]$.
\begin{figure}
\begin{center}
\begin{tikzpicture}
[xscale = 0.7, yscale = 0.7]
\draw [->] [thick](0, 0) -- (3, 0);
\draw [->] [thick](2, 2) -- (5, 2);
\draw [thick](0, 0) -- (2, 2);
\draw [ultra thick](2.5, 0) -- (4.5, 2);
\draw(4.6, -0.3) node {$\Si$-direction};
\draw(6.6, 2) node {$\Si$-direction};
\draw(2.5, -0.4) node {$R_0$};
\draw(4.5, 2.3) node {$R_1$};
\draw(-0.3, -0.3) node {$g_0$};
\draw(2.4, 2.4) node {$g_1$};
\fill(0,0) circle (2.8pt);
\fill(2,2) circle (2.8pt);
\fill [gray, opacity=.3] (0, 0) -- (2.5,0) -- (4.5, 2) -- (2, 2);
\end{tikzpicture}
\end{center}
\caption{The definition of $\SWcoch[{{0,1}}]$}
\label{figure : the definitions of SWcochzo}
\end{figure}

\begin{figure}
\begin{center}
\begin{tikzpicture}
[xscale = 0.7, yscale = 0.7]
\draw [->] [thick](0, 0) -- (3, 0);
\draw [->] [thick](0, 0) -- (0, 3);
\draw [ultra thick] (2.5, 0) -- (0, 2.5);
\draw [->] [thick](4.5, 2) -- (5, 2);
\draw [->] [thick](2, 4.5) -- (2, 5);
\draw [thick, dashed](4.5, 2) -- (2, 2);
\draw [thick, dashed](2, 4.5) -- (2, 2);
\draw [ultra thick] (4.5, 2) -- (2, 4.5);
\draw [thick, dashed](0, 0) -- (2, 2);
\draw [ultra thick](2.5, 0) -- (4.5, 2);
\draw [ultra thick](0, 2.5) -- (2, 4.5);
\draw(4.6, -0.3) node {$\Si_0$-direction};
\draw(-1.5, 3.5) node {$\Si_1$-direction};
\draw(6.6, 1.8) node {$\Si_0$-direction};
\draw(1.7, 5.4) node {$\Si_1$-direction};
\draw(-0.4, 2.5) node {$R_0$};
\draw(2.5, -0.4) node {$R_0$};
\draw(1.5, 4.6) node {$R_1$};
\draw(4.5, 1.5) node {$R_1$};
\draw(-0.3, -0.3) node {$g_0$};
\draw(2.4, 2.4) node {$g_1$};
\fill(0,0) circle (2.8pt);
\fill(2,2) circle (2.8pt);
\fill [gray, opacity=.3/2] (4.5, 2) -- (2.5,0) -- (0,2.5) -- (2, 4.5);
\fill [gray, opacity=.1] (0, 0) -- (2.5,0) -- (0,2.5) -- (0, 0);
\fill [gray, opacity=.3/2] (0, 0) -- (2,2) -- (2,4.5) -- (0, 2.5);
\fill [gray, opacity=.3/2] (0, 0) -- (2.5,0) -- (4.5,2) -- (2, 2);
\fill [gray, opacity=.3/2] (2, 2) -- (4.5,2) -- (2,4.5) -- (2, 2);
\end{tikzpicture}
\end{center}
\caption{Argument to absorb the ambiguities}
\label{figure : basic argument to consider the ambiguities}
\end{figure}

These three steps enable us to define the cohomology class $\SWcoh(X, \fraks)$.
However, there are, of course, a number of details to justify the above argument.
The main technical issue is caused since the cardinality of the set of vertices of $\calK$ is infinite.
We here explain this issue for the motivation of the subsequent subsections.
In the first step, we have to stretch some neighborhoods of $\Si_0$ and $\Si_1$ independently for a $1$-simplex $\lb \si \rb = \lb \Si_0, \Si_1 \rb$.
Therefore we have to introduce a metric obtained by gluing the initial metric $g$ with some cylindrical metrics.
Let $N(\si; \Si_i)$ denote a neighborhood of $\Si_i$ with $N(\si; \Si_0) \cap N(\si; \Si_1) = \emptyset$, and let $g(\si)$ denote a new metric made from $g$ and having a cylindrical metric near the boundary of $N(\si; \Si_i)$.
Then we can do the argument in the first step under a small modification.
However, to do the second step, we have to consider a $2$-simplex $\lb \tau \rb = \lb \Si_0 , \Si_1, \Si_2 \rb$ and stretch neighborhoods of $\Si_0$, $\Si_1$, and $\Si_2$ independently.
For $1$-simplices $\si_j$ $(j = 0, 1, 2)$ given in \eqref{eq : 1-simplices come from boundary of a 2-simplex}, we have already fixed the neighborhood $N(\si_j; \Si_i)$ for each $\Si_i \in \si_j$.
However, for the independent stretching for these three surfaces, we need disjoint neighborhoods of the surfaces, and we may have to take a smaller neighborhood $N(\tau; \Si_i)$ in $N(\si_j; \Si_i)$ and define the metric $g(\tau)$ using these new neighborhoods.
The stretching argument in the first step is done using $g(\si)$ as the initial metric, however the argument in the second step should be done using $g(\tau)$.
We have to mediate between these two arguments, and this (and its higher-dimendional version) will be the main part of the rest part of \cref{section: Family of Riemannian metrics} and \cref{section: Construction of the cohomological invariant}.

Note that, if we only consider a finite subcomplex $K$ of $\calK$, then we can choose a sufficiently small neighborhood $N(K; \Si)$ for each surface $\Si \in V(\calK)$ satisfying that $N(K; \Si) \cap N(K; \Si') = \emptyset$ if $\Si \cap \Si' = \emptyset$.
However, for a vertex $\Si \in V(\calK)$, one can make a sequence $\Si', \Si'', \ldots \in V(\calK)$ closing to $\Si$ obtained as parallel copies (in the sense of the following \cref{rem: parallel copy}) of $\Si$.
Therefore we cannot take a neighborhood $N(\calK; \Si)$ depending only on each $\Si \in V(\calK)$ satisfying that $N(\calK; \Si) \cap N(\calK; \Si') = \emptyset$ if $\Si \cap \Si' = \emptyset$.

\begin{rem}
For an embedded surface $\Si \subset X$ with self-intersection number zero, by pushing $\Si$ in the fiber direction of its normal bundle, we obtain an embedded surface whose homology class coincides with $\Si$'s and which has no geometric intersections with $\Si$.
We call such a surface a {\it parallel copy} of $\Si$.
\label{rem: parallel copy}
\end{rem}

\subsection{Construction of a family of Riemannian metrics}

\label{subsection: Construction of a family of Riemannian metrics}

In this subsection we construct a family of Riemannian metrics on a $4$-manifold by stretching neighborhoods of embedded surfaces.
(We consider the simplicial complex $\calK$ and construct a family of Riemannian metrics using $\calK$ in this section in order to construct a cohomology class on $\calK$ in \cref{section: Construction of the cohomological invariant}.
However, for any subcomplex of $\bar{\calK}$, one can also consider a similar construction of a family of Riemannian metrics using it.)

Let $(X, \fraks)$ be a smooth oriented closed $\spc$ $4$-manifold.
We denote by $\Met(X)$ the space of Riemannian metrics.
Henceforth in this subsection we fix a metric $g \in \Met(X)$.
We also fix functions $N(\cdot)$ and $a(\cdot)$ as follows.
First, for each $\Si \in V(\calK)$, we fix a tubular neighborhood $N(\Si)$ of $\Si$ equipped with
a diffeomorphism $f_{\Si} : N(\Si) \to D^2 \times \Si$ with $f_{\Si}|_{\Si} = \iota_{\Si}$, where $D^2$ is the unit disk in $\C$ and $\iota_{\Si} : \Si = \{0\} \times \Si \inc D^{2} \times \Si$ is the inclusion.
We will denote by $\tilde{f}_{\Si} : N(\Si) \setminus \Si \to (0, 1] \times S^1 \times \Si$ the diffeomorphism defined as the composition of the restriction  of $f_{\Si}$ and the diffeomorphism $D^2 \setminus \{0\} \times \Si \cong (0, 1] \times S^{1} \times \Si$ obtained from the canonical identification $D^2 \setminus \{0\} \cong (0, 1] \times S^{1}$.
We next take a function $a : S(\calK) \to (0, 1]$ satisfying $a(\Si) = a(\{\Si\}) = 1$ for each $\Si \in V(\calK)$ and enjoying the following two conditions (a) and (b) for any $\si \in S(\calK)$:
\begin{cond}
$\quad$
\begin{description}
\item[(a)] For any (strictly small) face $\tau \precneqq \si$, the inequalities
\[
0 < a(\si) \leq  a(\tau)
\]
hold.
\item[(b)] The equality $U(\si,\Si) \cap U(\si,\Si') = \emptyset$ holds for vertices $\Si, \Si' \in \si$ with $\Si \neq \Si'$.
Here $U(\si, \Si)$ is the subset of $N(\Si)$ defined as
\[
U(\si, \Si) \cong \left(0, 6a(\si)\right] \times S^1 \times \Si
\]
via the identification obtained from $\tilde{f}_{\Si}$.
\end{description}
\label{cond: condition for a()}
\end{cond}
The existence of such a function $a(\cdot)$ easily follows from induction on the dimension of $\si \in S(\calK)$.

We fix our convention on cut-off functions as follows.
Let us consider the concrete monotonically increasing function $\rho_0 : [0,1] \to [0,1]$ defined as
\[
\rho_0(t) := \frac{\bar{\rho}_0(t)}{\bar{\rho}_0(t) + \bar{\rho}_0(1-t)},\ \bar{\rho}_0 : [0,1] \to [0,1],\ \bar{\rho}_0(t) = e^{-1/t}\ {\rm if}\ t > 0,\ \bar{\rho}_0(0) = 0.
\]

\begin{convention}
Let $a, b, a', b'$ be real numbers with $a<a'<b'<b$.
In this paper, the {\it cut-off  function $\rho$ on $([a, b], [a', b'])$} is the following function:
The function $\rho$ coincides with the constant function $1$ on $[a', b']$, with the pull-back of $\rho_0(t)$ by the diffeomorphism
\[
[0,1] \cong [a, a']\ ; \ t \leftrightarrow (a'-a)t+a
\]
on $[a, a']$, and with the pull-back of $\rho_0(1-t)$ by the diffeomorphism 
\[
[0,1] \cong [b', b]\ ;\  t \leftrightarrow (b-b')t+b'
\]
on $[b', b]$ respectively.
We also define the {\it cut-off function on $([a, b], [a', b']) \times S^1 \times \Si$} for some $\Si \in V(\calK)$ as the pull-back of the cut-off function defined above on $([a, b], [a', b'])$ by the projection
 $[a,b] \times S^1 \times \Si \to [a,b]$.
\end{convention}

We next define a metric $G(\cdot\ ,\ \cdot)$ using the positive numbers $a(\cdot)$ and the fixed metric $g$.
Let $\si \in S(\calK)$ be a simplex and $\lambda \in \R$ be a real number with $0 < \lambda \leq a(\si)$.
We define the metric $G(\si, \lambda)$ by gluing the metric $g$ with the product metric on
\begin{align}
\bigsqcup_{\Si \in \si} \left[ \lambda, 6\lambda \right] \times S^1 \times \Si,
\label{eq: most large disjoint set}
\end{align}
where we use the cut-off function on $(\left[ \lambda, 6\lambda \right], \left[ 2\lambda, 5\lambda \right]) \times S^1 \times \Si$ for each $\Si \in \si$ via the identifications obtained from $\{\tilde{f}_{\Si}\}_{\Si \in \si}$ in order to glue them.
Here we equip $\Si$ with the metric $g_{\Si}$ of constant scalar curvature and of unit area, and equip $S^1$ with the metric $d\theta^2$ of unit length in \eqref{eq: most large disjoint set}.
Note that $\{\tilde{f}_{\Si}\}_{\Si \in \si}$ gives an identification between mutually disjoint $(\dim{\si}+1)$-subsets of $X$ and \eqref{eq: most large disjoint set} since $\lambda \leq a(\si)$.
Let us define
\begin{align*}
\bar{X}(\lambda, \Si) := \left[ 2\lambda, 5\lambda \right] \times S^1 \times \Si
\subset N(\Si) \subset X
\end{align*}
via the identification obtained from $\tilde{f}_{\Si}$.

We next define a ``stretched" metric.
Let $\si, \tau \in S(\calK)$ be simplices with $\tau \prec \si$,
$\lambda \in \R$ be a real number with $0 < \lambda \leq a(\si)$,
and  $\{r_{\Si}\}_{\Si \in \tau} \in [0, \infty)^{\tau} = \prod_{\Si \in \tau} [0,\infty)$ be a family of non-negative numbers.
We define $G(\si, \lambda, \{r_{\Si}\}_{\Si \in \tau} ) \in \Met(X)$ as follows.
For $\Si \in \tau$, the restricted metric $G(\si, \lambda)|_{\bar{X}(\lambda, \Si)}$ can be expressed as
\begin{align}
G(\si, \lambda)|_{\bar{X}(\lambda, \Si)} = dt^2 + d\theta^2 + g_{\Si},
\label{eq: initial metric on cylinder}
\end{align}
where $dt^2$ is the standard metric on the interval $\left[2\lambda, 5\lambda \right]$.
Let 
\[
\rho_{\lambda} : \left[ 2\lambda, 5\lambda \right] \to [0,1]
\]
be the cut-off function on $(\left[ 2\lambda, 5\lambda \right], \left[ 3\lambda, 4\lambda \right])$.
We modify the metric $G(\si, \lambda)$
by replacing \eqref{eq: initial metric on cylinder} with
\[
(r_{\Si} \cdot \rho_{\lambda}(t) + 1) dt^2 + d\theta^2 + g_{\Si}
\]
on $\bar{X}(\lambda, \Si)$ for every $\Si \in \tau$; we write
\[
G(\si, \lambda, \{r_{\Si}\}_{\Si \in \tau} )
\]
for the modified metric.
We set
\begin{align*}
X(\lambda, \Si) := \left[ 3\lambda, 4\lambda \right] \times S^1 \times \Si \subset N(\Si) \subset X
\end{align*}
via the identification obtained from $\tilde{f}_{\Si}$.
We define the restricted Riemannian manifold $\Cy(\lambda, \Si, r_{\Si})$ by
\begin{align}
\Cy(\lambda, \Si, r_{\Si}) &:= \left(X(\lambda, \Si), G(\si, \lambda, \{r_{\Si}\}_{\Si \in \tau}) \right) \label{eq: definition of cylindrical part in X}\\
&\subset (X, G(\si, \lambda, \{r_{\Si}\}_{\Si \in \tau})). \nonumber
\end{align}
(Note that although $\si$ and $\tau$ appear in the right-hand side of \eqref{eq: definition of cylindrical part in X}, $\Cy(\lambda, \Si, r_{\Si})$ is independent of $\si$ and $\tau$ with $\Si \in \tau \prec \si$.)
Then the Riemannian manifold $\Cy(\lambda, \Si, r_{\Si})$ is isometric to the cylinder with the standard metric
\begin{align*}
\Cystd(\lambda, \Si, r_{\Si})
:= \left( \left[0, \lambda (r_{\Si} + 1) \right] \times S^1 \times \Si,\ dt^2 + d\theta^2 + g_{\Si}\right).
\end{align*}
We also set
\[
\barCy(\lambda, \Si, r_{\Si}) := (\bar{X}(\lambda, \Si), G(\si, \lambda, \{r_{\Si}\}_{\Si \in \tau})).
\]
Then there exists a positive number $\Lambda(\lambda, r_{\Si}) > \lambda (r_{\Si} + 1)$ such that $\barCy(\lambda, \Si, r_{\Si})$ is isometric to the cylinder
\[
\barCystd(\lambda, \Si, r_{\Si})
:= \left( \left[0, \Lambda(\lambda, r_{\Si}) \right] \times S^1 \times \Si,\ dt^2 + d\theta^2 + g_{\Si}\right).
\]
The number $\Lambda(\lambda, r_{\Si})$ depends only on $\lambda$ and $r_{\Si}$ since we fix convention on cut-off functions.

Let $n \geq 0$.
Henceforth we fix an $n$-simplex $\si \in S_{n}(\calK)$ in this subsection.
We consider simplices of $\Bd(\si)$, where we regard $\si$ as the simplicial complex $\scrK(\si)$ explained in \cref{subsection: Notation on simplicial complexes}.
We shall construct a continuous map
\begin{align}
\phi_{\si} : \bigcup_{s \in S(\Bd(\si))} |s| \times [0, \infty)^{\simin{s}} \to \Met(X).
\label{eq: map for total stretching}
\end{align}
Here the domain of $\phi_{\si}$ is a subset of 
\[
\bigcup_{s \in S(\Bd(\si))} |s| \times [0, \infty)^{\si}
= |\Bd(\si)| \times [0, \infty)^{\si}
= |\si| \times [0, \infty)^{\si}.
\]
The family of Riemannian metrics \eqref{eq: map for total stretching} is what we aim to construct in this subsection.
To construct the map \eqref{eq: map for total stretching}, for each simplex $s \in S(\Bd(\si))$, we will construct 
\begin{align}
\phi_{\si, s} :  |s| \times [0, \infty)^{\simin{s}} \to \Met(X)
\label{eq: map for partial stretching}
\end{align}
so that the restriction of $\phi_{\si, s}$ to $\dom(\phi_{\si, s}) \cap \dom(\phi_{\si, s'})$ coincides with that of  $\phi_{\si, s'}$ for another simplex $s' \in S(\Bd(\si))$.
Here we denote by $\dom(\phi_{\si, s})$ the domain of $\phi_{\si, s}$.
Then we can obtain \eqref{eq: map for total stretching} by gluing $\phi_{\si, s}$'s  together.
Namely we define the map \eqref{eq: map for total stretching} by
\begin{align}
\phi_{\si} = \bigcup_{s \in S(\Bd(\si))} \phi_{\si, s} :  \bigcup_{s \in S(\Bd(\si))} |s| \times [0, \infty)^{\simin{s}} \to \Met(X).
\label{eq: map for total stretching, definition}
\end{align}
To construct the map \eqref{eq: map for partial stretching}, let us consider simplices of $\Bd(s)$, where we regard $s$ as the simplicial complex $\scrK(s)$ explained in \cref{subsection: Notation on simplicial complexes}.
For each simplex $\scrS \in S(\Bd(s))$, we will construct 
\begin{align}
\phi_{\si, s, \scrS} :  |\scrS| \times [0, \infty)^{\simin{s}} \to \Met(X)
\label{eq: map for partial partial stretching}
\end{align}
so that the restriction of $\phi_{\si, s, \scrS}$ to $\dom(\phi_{\si, s, \scrS}) \cap \dom(\phi_{\si, s, \scrS'})$ coincides with that of  $\phi_{\si, s, \scrS'}$ for another simplex $\scrS' \in S(\Bd(s))$.
We will define \eqref{eq: map for partial stretching} by gluing $\phi_{\si, s, \scrS}$'s  together.

The map \eqref{eq: map for partial partial stretching} is constructed as follows.
Let $\scrS$ be a simplex of $\Bd(s)$.
We can write $\scrS = \Set{s_{0}, \ldots, s_{l}}$ with $s_{0} \precneqq \cdots \precneqq s_{l} \prec s$ for some $l \geq 0$.
Recall that the geometric realization  $|\scrS| = \Conv\{s_{0}, \ldots, s_{l}\}$ is a subset of the real vector space $\R^{\{s_{0}, \ldots, s_{l}\}}$ generated by $s_{0}, \ldots, s_{l}$.
Let $\bvec{t} = \{t_i\}_{i = 0}^l$ be a collection of non-negative numbers with $\sum_{i=0}^l t_i = 1$.
We set
\[
\lambda(\scrS, \bvec{t}) := \sum_{j=0}^l t_j a(\simax{s_j}).
\]
Note that $\lambda(\scrS, \bvec{t}) \leq a(\simin{s_i})$ holds for any $i \in \{0, \ldots, l\}$.
Indeed, $a(\simax{s_j}) \leq a(\simin{s_i})$ holds for each $j$ since we have
\begin{align}
\simin{s_l} \prec \cdots \prec \simin{s_0} \prec \simax{s_0} \prec \cdots \prec \simax{s_l}.
\label{eq: prec relations}
\end{align}
The metric $G(\simin{s_i}, \lambda(\scrS, \bvec{t}))$ is therefore well-defined.
Since we also have $\simin{s} \prec \simin{s_i}$, we can define
the metric
\[
G(\simin{s_i}, \lambda(\scrS, \bvec{t}), \{r_{\Si}\}_{\Si \in \simin{s}})
\]
for each $\{r_{\Si}\}_{\Si \in \simin{s}} \in [0, \infty)^{\simin{s}}$.
We now define the value of $\phi_{\si, s, \scrS}$ at the pair $(p, \{r_{\Si}\})$ by
\begin{align}
\phi_{\si, s, \scrS}(p, \{r_{\Si}\})
:= \sum_{i=0}^l t_i \cdot G(\simin{s_i}, \lambda(\scrS, \bvec{t}), \{r_{\Si}\}_{\Si \in \simin{s}}),
\label{eq: map for partial partial stretching, description}
\end{align}
where 
\begin{align*}
p = \sum_{i=0}^l t_i s_i \in |\scrS|,\quad 
\{r_{\Si}\} = \{r_{\Si}\}_{\Si \in \simin{s}} \in [0, \infty)^{\simin{s}}, {\rm and}\quad 
\bvec{t} = \{t_i\}_{i = 0}^l.
\end{align*}
We note that the summation in the right-hand side of \eqref{eq: map for partial partial stretching, description} makes sense since $\Met(X)$ is convex.

It is straightforward to check that the restriction of $\phi_{\si, s, \scrS}$ to $\dom(\phi_{\si, s, \scrS}) \cap \dom(\phi_{\si, s, \scrS'})$ coincides with that of  $\phi_{\si, s, \scrS'}$ for another simplex $\scrS' \in S(\Bd(s))$.
We can therefore define the map $\phi_{\si, s}$ by gluing $\phi_{\si, s, \scrS}$'s  together.
Namely we define the map \eqref{eq: map for partial stretching} by
\[
\phi_{\si, s} = \bigcup_{\scrS \in S(\Bd(s))} \phi_{\si, s, \scrS} : \bigcup_{\scrS \in S(\Bd(s))} |\scrS| \times [0, \infty)^{\simin{s}} \to \Met(X).
\]
Here note that $\bigcup_{\scrS \in S(\Bd(s))} |\scrS| = |\Bd(s)| = |s|$.
Similarly, one can check that the restriction of $\phi_{\si, s}$ to $\dom(\phi_{\si, s}) \cap \dom(\phi_{\si, s'})$ coincides with that of  $\phi_{\si, s'}$ for another simplex $s' \in S(\Bd(\si))$.
We can therefore define the map $\phi_{\si}$ by \eqref{eq: map for total stretching, definition} as explained.

\begin{ex}
\label{ex: domain of phi_i and decomposition of the domain}

We here describe the domain of $\phi_{\si}$ in the case that $\si$ is a $1$-simplex of $\calK$.
Write $\si = \{ \Si_0, \Si_1\}$ and put $s_i = \{ \{\Si_i\}, \si \}$ $(i = 0,1)$.
The domain of $\phi_{\si}$ is described as the shadowed part in \cref{figure : domain of phi_si}.
The shadowed part is a subspace of $|\Bd(\si)| \times [0, \infty)^{\si} = |\si| \times [0, \infty)^{\si} \cong \Delta^1 \times [0, \infty)^2$ and obtained as a union 
\begin{align*}
& \bigcup_{i=0}^1 \left( |s_i| \times [0, \infty)^{\{\Si_i\}} \right) \cup \left(|\{\si\}| \times [0, \infty)^{\si} \right) \cup  \bigcup_{i=0}^1 \left( |\{\{\Si_i\}\}| \times [0, \infty)^{\{\Si_i\}}\right)\\
\cong & \bigcup_{i=0}^1 \left( \Delta^1 \times [0, \infty) \right)_i \cup  [0, \infty)^{2} \cup \bigcup_{i=0}^1 [0, \infty)_i,
\end{align*}
where $\left( \Delta^1 \times [0, \infty) \right)_i $ and $[0, \infty)_i$ are a copy of $\Delta^1 \times [0, \infty)$ and of $[0, \infty)$ respectively.
Note that the last part $|\{\{\Si_i\}\}| \times [0, \infty)^{\{\Si_i\}}$ is contained in the first part $|s_i| \times [0, \infty)^{\{\Si_i\}}$.
In \cref{figure : domain of phi_si}, the horizontal arrows $\rightarrow$ correspond to the ``stretching parameter" $r_{\Si_0}$ on $\Si_0$ and the vertical arrows $\uparrow$ correspond to $r_{\Si_1}$.
The map $\phi_{\si}$ is obtained by stretching a neighborhood of $\Si_i$ on the part $|s_i| \times [0, \infty)^{\{\Si_i\}}$, and by stretching neighborhoods of both of $\Si_0$ and $\Si_1$ on the part $|\{\si\}| \times [0, \infty)^{\si}$.
The decompositions of the domain of $\phi_{\si}$ obtained from the barycentric subdivision of $s_i$ are also described in \cref{figure : domain of phi_si}.

\begin{figure}
\begin{center}
\begin{tikzpicture}
[xscale = 0.7, yscale = 0.7]
\draw [->] [very thick](0, 0) -- (3.1, 0);
\draw [->] [very thick](0, 0) -- (0, 3.1);
\draw [->] [very thick](2, -2) -- (5.1, -2);
\draw [->] [very thick](2, -2) -- (2, 1.1);
\draw [->] [very thick, dashed](-2, 2) -- (1.1, 2);
\draw [->] [very thick](-2, 2) -- (-2, 5.1);
\draw [->] [very thick](1, -1) -- (1+3.1, -1);
\draw [->] [very thick](-1, 1) -- (-1, 1+3.1);
\draw [very thick](-2, 2) -- (2, -2);
\draw [ultra thick, dotted](-2, 4.5) -- (0, 2.5);
\draw [ultra thick, dotted](0, 2.5) -- (2.5, 2.5);
\draw [ultra thick, dotted](2.5, 2.5) -- (2.5, 0);
\draw [ultra thick, dotted](2.5, 0) -- (4.5, -2);
\draw(-0.7, -0.3) node {$|\{\si\}|$};
\draw(2.1, -2.5) node {$|\{\{\Si_0\}\}|$};
\draw(-3.2, 1.8) node {$|\{\{\Si_1\}\}|$};
\draw(0.6, -1.4) node {$|s_0|$};
\draw(-1.6, 0.7) node {$|s_1|$};
\draw(4.6, -2.4) node {$R$};
\draw(-2.4, 4.5) node {$R$};
\draw(5.6, -2) node {$r_{\Si_0}$};
\draw(2.1, 1.35) node {$r_{\Si_1}$};
\fill(0,0) circle (2.8pt);
\fill(2,-2) circle (2.8pt);
\fill(-2,2) circle (2.8pt);
\fill [gray, opacity=.3] (0,0) -- (0,3) -- (3,3) -- (3,0);
\fill [gray, opacity=.4] (0,0) -- (3,0) -- (5, -2) -- (2,-2) -- (0,0);
\fill [gray, opacity=.4] (0,0) -- (0,3) -- (-2, 5) -- (-2,2) -- (0,0);
\end{tikzpicture}
\end{center}
\caption{The domain of $\phi_{\si}$}
\label{figure : domain of phi_si}
\end{figure}

\end{ex}

We now note the following \lcnamecref{lem: cylindrical structure for any point on the family}, which will be used to prove a vanishing property of the \SW moduli space for a stretched metric.

\begin{lem}
\label{lem: cylindrical structure for any point on the family}
Let $\si$ be a simplex of $\calK$, $s$ be a simplex of $\Bd(\si)$, and $\scrS$ be a simplex of $\Bd(s)$.
Let us write $\scrS = \Set{s_{0}, \ldots, s_{l}}$ with $s_{0} \precneqq \cdots \precneqq s_{l} \prec s$, where $l = \dim \scrS$.
Let $\bvec{t} = \{t_i\}_{i = 0}^l$ be a collection of non-negative numbers with $\sum_{i=0}^l t_i = 1$.
Set $p = \sum_{i=0}^l t_i s_i \in |\scrS|$.
Then, for each $\{r_{\Si}\} = \{r_{\Si}\}_{\Si \in \simin{s}} \in [0, \infty)^{\simin{s}}$, the subspaces of $X$
\begin{align}
\Set{ \bar{X}(\lambda(\scrS, \bvec{t}), \Si)) | \Si \in \simin{s} }
\label{eq: disjoint family}
\end{align}
are mutually disjoint in X and we have
\begin{align}
\bigsqcup_{\Si \in \simin{s}} (\bar{X}(\lambda(\scrS, \bvec{t}), \Si), \phi_{\si}(p, \{r_{\Si}\})) = \bigsqcup_{\Si \in \simin{s}} \overline{\Cy}(\lambda(\scrS, \bvec{t}), \Si, r_{\Si}).
\label{eq: equation between disjoint family}
\end{align}
\end{lem}

\begin{proof}
Since we have $\simin{s} \prec \simin{s_l}$ and \eqref{eq: prec relations}, we obtain $\lambda(\scrS, \bvec{t}) \leq a(\simin{s})$.
The subspaces in \eqref{eq: disjoint family} are therefore mutually disjoint.
For each $\Si \in \simin{s}$, we have
\begin{align*}
\phi_{\si}(p, \{r_{\Si}\})|_{\bar{X}(\lambda(\scrS, \bvec{t}), \Si)}
&=  \sum_{i=0}^l t_i \cdot G(\simin{s_i}, \lambda(\scrS, \bvec{t}), \{r_{\Si}\})|_{\bar{X}(\lambda(\scrS, \bvec{t}), \Si)}\\
&=  \sum_{i=0}^l t_i \cdot G(\simin{s}, \lambda(\scrS, \bvec{t}), \{r_{\Si}\})|_{\bar{X}(\lambda(\scrS, \bvec{t}), \Si)}\\
&= G(\simin{s}, \lambda(\scrS, \bvec{t}), \{r_{\Si}\})|_{\bar{X}(\lambda(\scrS, \bvec{t}), \Si)}.
\end{align*}
This proves \eqref{eq: equation between disjoint family}.
\end{proof}

\subsection{Structure of the parameter space}

\label{subsection: Structure of the parameter space}

In this subsection we investigate the structure of a subspace of the domain of the map \eqref{eq: map for total stretching}.

Let $n \geq 0$ and $\si \in S_{n}(\calK)$ be an $n$-simplex.
For a simplex $s \in S(\Bd(\si))$ and a non-negative number $R \geq 0$, we set
\begin{align*}
\frakextdel(\si, s, R) &:= \Set{ \{r_{\Si}\}_{\Si \in \simin{s}} \in [0, R]^{\simin{s}} | \exists \Si \in \simin{s},\ r_{\Si} = R}\\
&\subset [0, \infty)^{\simin{s}}.
\end{align*}
(The notation $\frakextdel$ stands for ``the exterior half" of the boundary of $[0, R]^{\simin{s}}$.)
We shall consider the subspace
\[
\frakextdel(\si, R) := \bigcup_{s \in S(\Bd(\si))} |s| \times \frakextdel(\si, s, R)
\]
of the domain of the map \eqref{eq: map for total stretching}.
For example, in the case that $\si$ is a $1$-simplex, $\frakextdel(\si, R)$ is the union of dotted lines in \cref{figure : domain of phi_si}.
We also set
\[
\extdel(\si, R) := \Set{ \{x_{\Si}\}_{\Si \in \si} \in [0,R]^{\si} | \exists \Si \in \si,\ x_{\Si} = R }.
\]
The aim of this subsection is to prove the following \lcnamecref{prop: structure of the subspace of the domain of phi}:

\begin{prop}
\label{prop: structure of the subspace of the domain of phi}
There exists a homeomorphism
\begin{align}
\psi_{\si, R} : \frakextdel(\si, R) \to \extdel(\si, R).
\label{eq: homeomorphism to the standard simplex}
\end{align}
\end{prop}

Before proving \cref{prop: structure of the subspace of the domain of phi}, we decompose the domain and the range of the map \eqref{eq: homeomorphism to the standard simplex} into small pieces.
Let $\Si \in \si$ be a vertex and $\tau \prec \si$ be a face with $\Si \in \tau$.
We set 
\begin{align*}
\frakextdel(\si, \Si, \tau, R) := \Set{ \{r_{\Si'}\}_{\Si' \in \tau} \in [0, R]^{\tau} | r_{\Si} = R}.
\end{align*}
For a simplex $s \in S(\Bd(\si))$ satisfying $\tau = \simin{s}$, we have $\frakextdel(\si, \Si, \tau, R) \subset \frakextdel(\si, s, R)$.
Moreover, we also have the decomposition
\begin{align}
\frakextdel(\si, R)
&=  \bigcup_{\Si \in \si} \bigcup_{\Si \in \tau \prec \si} \bigcup_{\substack{ s \in S(\Bd(\si)), \\  \tau = \simin{s}}} |s| \times \frakextdel(\si, \Si, \tau, R) \label{eq: decomposition to 3 unions}\\
&=  \bigcup_{\Si \in \si} \bigcup_{\Si \in \tau \prec \si} \bigcup_{\substack{ s \in S(\Bd(\si)), \\  \tau = \simin{s}, \\ \dim{s} + \dim\tau = n}} |s| \times \frakextdel(\si, \Si, \tau, R).
 \label{eq: decomposition to 3 unions, more detailed}
\end{align}
Here the variables in the three unions in the right-hand side of \eqref{eq: decomposition to 3 unions} and \eqref{eq: decomposition to 3 unions, more detailed} are $\Si$, $\tau$, and $s$ respectively.
In \eqref{eq: decomposition to 3 unions, more detailed}, $\dim{s}$ and $\dim\tau$ denote the dimension of $s$ as a simplex of $\Bd(\si)$  and that of $\tau$ as a simplex of $\calK$ respectively.
The space $\extdel(\si, R)$ also decomposes as $\extdel(\si, R) = \bigcup_{\Si \in \si} \extdel(\si, \Si, R)$,
where we set
\[
\extdel(\si, \Si, R) := \Set{ \{x_{\Si'}\}_{\Si' \in \si} \in [0,R]^{\si} | x_{\Si} = R } \subset \extdel(\si, R).
\]

For each $\Si \in \si$, we will construct a homeomorphism
\begin{align}
\psi_{\si, \Si, R} :  \bigcup_{\Si \in \tau \prec \si} \bigcup_{\substack{ s \in S(\Bd(\si)), \\  \tau = \simin{s}, \\ \dim{s} + \dim\tau = n}} |s| \times \frakextdel(\si, \Si, \tau, R) \to \extdel(\si, \Si, R)
\label{eq: homeomorphism to the part of standard simplex}
\end{align}
and define $\psi_{\si, R}$ by gluing $\psi_{\si, \Si, R}$'s.
Note that we have a homeomorphism
\begin{align}
\bar{\psi}_{\si,\Si, R} : \extdel(\si, \Si, R) \to [0, R]^{\siSi}
\label{eq: the map for identification between disk and cube}
\end{align}
given by
\[
\{x_{\Si'}\}_{\Si' \in \si} \mapsto \{x_{\Si'}\}_{\Si' \in \siSi},
\]
and we consider $[0, R]^{\siSi}$ in order to construct \eqref{eq: homeomorphism to the part of standard simplex}.
For a face $\tau \prec \si$ with $\Si \in \tau$, we define a collection of boundary points
\[
\{b_{\Si'}(\tau)\} = \{b_{\Si'}(\tau)\}_{\Si' \in \siSi} \in \{0,R\}^{\siSi} \subset \del([0,R]^{\siSi})
\]
by
\[
\Si' \in \tau \setminus \{\Si\} \Rightarrow b_{\Si'}(\tau) = R, {\rm \ and}\quad \Si' \notin \tau \setminus \{\Si\} \Rightarrow b_{\Si'}(\tau) = 0.
\]
We set
\[
Q(\si, \Si, \tau, R)
:= \Set{ \{x_{\Si'}\}_{\Si' \in \siSi} \in [0,R]^{\siSi} | \left|x_{\Si'} - b_{\Si'}(\tau)\right| \leq R/2}.
\]
Then we obtain the decomposition
\[
[0, R]^{\siSi}
= \bigcup_{\Si \in \tau \prec \si} Q(\si, \Si, \tau, R).
\]
Indeed, for any point $\{x_{\Si'}\} \in [0, R]^{\siSi}$, let us define $\tau \prec \si$ by 
\[
\tau := \Set{\Si' \in \siSi | \left|x_{\Si'} - R\right| \leq R/2} \cup \{\Si\},
\]
then we have $\{x_{\Si'}\} \in Q(\si, \Si, \tau, R)$.

For $\Si \in \si$ and $\tau \prec \si$ with $\Si \in \tau$, we will construct a homeomorphism
\begin{align}
\psi_{\si, \Si, \tau, R} : \bigcup_{\substack{ s \in S(\Bd(\si)), \\  \tau = \simin{s}, \\ \dim{s} + \dim\tau = n}} |s| \times \frakextdel(\si, \Si, \tau, R) \to Q(\si, \Si, \tau, R)
\label{eq: the homeo map decomposed 1}
\end{align}
and define $\psi_{\si, \Si, R}$ by gluing $\psi_{\si, \Si, \tau, R}$'s.
For a simplex $s \in S(\Bd(\si))$ satisfying $\tau = \simin{s}$ and $\dim{s} + \dim\tau = n$, we define subsets $\triangle(\si, \Si, \tau, s, R)$, $\Box(\si, \Si, \tau, R)$, and $Q(\si, \Si, \tau, s, R)$ of $Q(\si, \Si, \tau, R)$ as follows.
Note that we have $\si = \simax{s}$ (i.e. $n = \dim\simax{s}$) since
\[
n \geq \dim\simax{s} \geq \dim{s} + \dim\simin{s} = \dim{s} + \dim\tau = n.
\]
Such a simplex $s$ therefore corresponds to a sequence
\[
\tau = \si_{0} \precneqq \cdots \precneqq \si_{k} = \si,
\]
where $k = n - \dim\tau$.
We define
\begin{align*}
\triangle(\si, \Si, \tau, s, R)
& := \Conv\Set{ \left\{\frac{1}{2}b_{\Si'}(\si_i) \right\}_{\Si' \in \si \setminus \tau} | 0 \leq i \leq k } \\
&\subset [0,R/2]^{\si \setminus \tau},\\
\Box(\si, \Si, \tau, R)
& := Q(\si, \Si, \tau, R) \cap [0,R]^{\tau \setminus \{\Si\}} =[R/2, R]^{\tau \setminus \{\Si\}} ,
\end{align*}
and
\begin{align*}
&Q(\si, \Si, \tau, s, R)\\
&:=\Set{
\{x_{\Si'}\}_{\Si' \in \siSi} \in Q(\si, \Si, \tau, R) |
\begin{matrix}
\{x_{\Si'}\}_{\Si' \in \si \setminus \tau} \in \triangle(\si, \Si, \tau, s, R),\\
\{x_{\Si'}\}_{\Si' \in \tau \setminus \{\Si\}} \in \Box(\si, \Si, \tau, R)
\end{matrix}
}\\
&\cong \triangle(\si, \Si, \tau, s, R) \times \Box(\si, \Si, \tau, R).
\end{align*}
Here, in the case that $\tau = \si$ and $\tau = \{\Si\}$, we identify $Q(\si, \Si, \tau, s, R)$ with $\Box(\si, \Si, \tau, R)$ and $\triangle(\si, \Si, \tau, s, R)$ respectively.

\begin{lem}
\label{lem: decomposition of Q(tau)}
The space $Q(\si, \Si, \tau, R)$ decomposes as
\begin{align}
Q(\si, \Si, \tau, R)
= \bigcup_{\substack{ s \in S(\Bd(\si)), \\  \tau = \simin{s}, \\ \dim{s} + \dim\tau = n}} Q(\si, \Si, \tau, s, R).
\label{eq: decomposition of Q(tau)}
\end{align}
\end{lem}

\begin{proof}
Set $k = n - \dim\tau$.
It is easy to check the case when $k \leq 1$, therefore we assume $k \geq 2$.
It suffices to show that the left-hand side of \eqref{eq: decomposition of Q(tau)} is included in the right-hand side.
Let $\{x_{\Si'}\}_{\Si' \in \siSi}$ be a point in $Q(\si, \Si, \tau, R)$.
We inductively take $\si_0, \ldots, \si_k$ as follows.
We first set $\si_0 := \tau$.
For $i$ with $0 \leq i \leq k$, assume that $\si_i$ is already given.
Take $\Si_{i} \in \si \setminus \si_i$ such that $x_{\Si_i}$ attains the maximum
\[
\max_{\Si' \in \si \setminus \si_{i}}{x_{\Si'}}.
\]
Using this $\Si_{i}$ we define $\si_{i+1} := \si_i \cup \{\Si_i\}$.
Through this procedure we eventually have $\si_0, \ldots, \si_k$; we set $s := \Set{\si_0, \ldots, \si_k}$.
We note that 
\begin{align*}
\{x_{\Si'}\}_{\Si' \in \si \setminus \tau} \in \triangle(\si, \Si, \tau, s, R)
\end{align*}
holds.
Indeed, we have
\begin{align*}
\triangle(\si, \Si, \tau, s, R) =  \bigcap_{0 \leq i \leq k-2} \Set{ \{x_{\Si'}'\}_{\Si' \in \si \setminus \tau} \in [0,R/2]^{\si \setminus \tau} | x_{\Si_i}' \geq x_{\Si_{i+1}}' }
\end{align*}
and $x_{\Si_0} \geq \cdots \geq x_{\Si_{k-1}}$ by the definition of $\Si_i$.
Hence it follows that $\{x_{\Si'}\}_{\Si' \in \siSi} \in Q(\si, \Si, \tau, s, R)$.
\end{proof}

Using the decomposition above, we now prove \cref{prop: structure of the subspace of the domain of phi}.

\begin{proof}[Proof of \cref{prop: structure of the subspace of the domain of phi}]
Let $\Si$ be a vertex of $\si$, $\tau$ be a face of $\si$ with $\Si \in \tau$, and $s$ be a simplex of $\Bd(\si)$ satisfying $\tau = \simin{s}$ and $\dim{s} + \dim\tau = n$.
We define a map
\begin{align}
\psi_{\si, \Si, \tau, s, R} : |s| \times \frakextdel(\si, \Si, \tau, R) \to Q(\si, \Si, \tau, s, R)
\label{eq: the map most decomposed}
\end{align}
as follows.
If we write $s = \Set{\si_0, \ldots, \si_k}$, we have $|s| = \Conv\Set{\si_0, \ldots, \si_k}$.
We define a map
\[
\psi_{\si, \Si, \tau, s, R}^{\triangle} : |s| \to \triangle(\si, \Si, \tau, s, R)
\]
by
\[
\sum_{i=0}^k t_i \si_i
\mapsto
\sum_{i=0}^k t_i \cdot \left\{\frac{1}{2}b_{\Si'}(\si_i) \right\}_{\Si' \in \si \setminus \tau}
\]
for $t_i \geq 0$ with $\sum_{i=0}^k t_i = 1$.
We next define a map
\[
\psi_{\si, \Si, \tau, s, R}^{\Box} : \frakextdel(\si, \Si, \tau, R) \to \Box(\si, \Si, \tau, R)
\]
by
\[
\{ r_{\Si'} \}_{\Si' \in \tau \setminus \{\Si\}}
\mapsto
\left\{ \frac{1}{2} \left(r_{\Si'} + 1 \right)\right\}_{\Si' \in \tau \setminus \{\Si\}}
\]
using the identification $\frakextdel(\si, \Si, \tau, R) \cong  [0, R]^{\tau \setminus \{\Si\}}$ given by
\[
\{r_{\Si'}\}_{\Si' \in \simin{s}} \mapsto \{r_{\Si'}\}_{\Si' \in \tau \setminus \{\Si\}}.
\]
We define the map \eqref{eq: the map most decomposed} as the product map of $\psi_{\si, \Si, \tau, s, R}^{\triangle}$ and $\psi_{\si, \Si, \tau, s, R}^{\Box}$.
The map $\psi_{\si, \Si, \tau, s, R}$ is obviously a homeomorphism.
We denote by $\psi_{\si, \Si, \tau, s, R}^{-1}$ the inverse map.

It is straightforward to check that the restriction of $\psi_{\si, \Si, \tau, s, R}$ to $\dom(\psi_{\si, \Si, \tau, s, R}) \cap \dom(\psi_{\si, \Si, \tau, s', R})$ coincides with that of  $\psi_{\si, \Si, \tau, s', R}$ for another simplex $s'$.
We can therefore define the map \eqref{eq: the homeo map decomposed 1} by gluing $\psi_{\si, \Si, \tau, s, R}$'s  together.
Namely we define the map \eqref{eq: the homeo map decomposed 1} as
\[
\bigcup_{\substack{ s \in S(\Bd(\si)), \\  \tau = \simin{s}, \\ \dim{s} + \dim\tau = n}} \psi_{\si, \Si, \tau, s, R} : \bigcup_{\substack{ s \in S(\Bd(\si)), \\  \tau = \simin{s}, \\ \dim{s} + \dim\tau = n}} |s| \times \frakextdel(\si, \Si, \tau, R) \to Q(\si, \Si, \tau, R).
\]
In a similar vein, we obtain a map from $\dom(\psi_{\si, \Si, R})$ to $[0, R]^{\siSi}$ by gluing $\psi_{\si, \Si, \tau, R}$'s  together, and we define the map \eqref{eq: homeomorphism to the part of standard simplex} by composing the inverse map of the homeomorphism \eqref{eq: the map for identification between disk and cube} with the glued map.
Similarly we define the map \eqref{eq: homeomorphism to the standard simplex} by gluing $\psi_{\si, \Si, R}$'s  together.
We obtain the inverse map of $\psi_{\si, R}$ by gluing $\psi_{\si, \Si, \tau, s, R}^{-1}$'s together and therefore $\psi_{\si, R}$ is a homeomorphism.
\end{proof}

\begin{ex}
Let $\si = \{ \Si_0, \Si_1\}$ be a $1$-simplex of $\calK$.
Then the possibilities of triples $(\Si, \tau, s)$ such that $\Si \in \tau \prec \si$, $\tau = \simin{s}$, and $\dim{s} + \dim\tau = n$ are as follows:
\begin{enumerate}
\item $\Si = \Si_0,\ \tau = \{\Si_0\},\ s = \Set{ \{\Si_0\}, \{\Si_0, \Si_1\} }$,
\item $\Si = \Si_0,\ \tau = \{\Si_0, \Si_1\},\ s = \Set{ \{\Si_0, \Si_1\} }$,
\item $\Si = \Si_1,\ \tau = \{\Si_1\},\ s = \Set{ \{\Si_1\}, \{\Si_0, \Si_1\} }$, and
\item $\Si = \Si_1,\ \tau = \{\Si_0, \Si_1\},\ s = \Set{ \{\Si_0, \Si_1\} }$.
\end{enumerate}
The decompositions discussed above of the domain and the range of $\psi_{\si, R}$ are described in \cref{figure : domain and range of psi_si}.
(See also \cref{ex: domain of phi_i and decomposition of the domain} and \cref{figure : domain of phi_si}.)
The (upper and lower case) Roman numerals in \cref{figure : domain and range of psi_si} correspond to Arabic numerals above.
In \cref{figure : domain and range of psi_si}, the parts (I), (II), (III), and (IV) correspond to (i), (ii), (iii), and (iv) through $\psi_{\si, R}$ respectively.

\begin{figure}
\begin{center}
\begin{tikzpicture}
[xscale = 0.7, yscale = 0.7]
\draw [->] [thick](0, 0) -- (3.1, 0);
\draw [->] [thick](0, 0) -- (0, 3.1);
\draw [->] [thick](2, -2) -- (5.1, -2);
\draw [->] [thick](2, -2) -- (2, 1.1);
\draw [->] [thick](-2, 2) -- (1.1, 2);
\draw [->] [thick](-2, 2) -- (-2, 5.1);
\draw [thick](-2, 2) -- (2, -2);
\draw [ultra thick](-2, 4.5) -- (0, 2.5);
\draw [ultra thick](0, 2.5) -- (2.5, 2.5);
\draw [ultra thick](2.5, 2.5) -- (2.5, 0);
\draw [ultra thick](2.5, 0) -- (4.5, -2);
\draw(-0.7,-0.3) node {$|\{\si\}|$};
\draw(2.1, -2.5) node {$|\{\{\Si_0\}\}|$};
\draw(-3.2, 1.8) node {$|\{\{\Si_1\}\}|$};
\draw(0.6,-1.4) node {$|s_0|$};
\draw(-1.6,0.7) node {$|s_1|$};
\draw(4.6, -2.4) node {$R$};
\draw(-2.4, 4.5) node {$R$};
\draw(4, -1) node {(I)};
\draw(3, 1) node {(I\hspace{-.1em}I)};
\draw(-0.9, 4.2) node {(I\hspace{-.1em}I\hspace{-.1em}I)};
\draw(1.4, 3) node {(I\hspace{-.1em}V)};
\fill(0,0) circle (2.8pt);
\fill(2,-2) circle (2.8pt);
\fill(-2,2) circle (2.8pt);
\draw(5.6, -2) node {$r_{\Si_0}$};
\draw(2.1, 1.35) node {$r_{\Si_1}$};
\draw [->] [thick](7+0, 0) -- (7+4.1, 0);
\draw [->] [thick](7+0, 0) -- (7+0, 4.1);
\draw [ultra thick](7+0, 3.5) -- (7+3.5, 3.5);
\draw [ultra thick](7+3.5, 3.5) -- (7+3.5, 0);
\draw(7+3.5, -0.5) node {$R$};
\draw(7-0.3, 3.5) node {$R$};
\draw(7+4.6, 0) node {$x_{\Si_0}$};
\draw(7+0, 4.5) node {$x_{\Si_1}$};
\fill(7+3.5*0.5,3.5) circle (2.8pt);
\fill(7+3.5,3.5*0.5) circle (2.8pt);
\draw(7+4, 1) node {(i)};
\draw(7+4, 2.5) node {(ii)};
\draw(7+1, 4) node {(iii)};
\draw(7+2.5, 4) node {(iv)};
\end{tikzpicture}
\end{center}
\caption{The domain and the range of $\psi_{\si}$ for $\si = \{ \Si_0, \Si_1\}$}
\label{figure : domain and range of psi_si}
\end{figure}

\end{ex}

\begin{ex}
Let $\si = \{ \Si_0, \Si_1, \Si_2\}$ be a $2$-simplex of $\calK$.
For $\Si = \Si_0$, the possibilities of pairs $(\tau, s)$ such that $\Si \in \tau \prec \si$, $\tau = \simin{s}$ and $\dim{s} + \dim\tau = n$ are as follows:
\begin{enumerate}
\item $\tau = \{\Si_0\},\ s = \Set{ \{\Si_0\}, \{\Si_0, \Si_1\}, \{\Si_0, \Si_1, \Si_2\} }$,
\item $\tau = \{\Si_0\},\ s = \Set{ \{\Si_0\}, \{\Si_0, \Si_2\}, \{\Si_0, \Si_1, \Si_2\} }$,
\item $\tau = \{\Si_0, \Si_1\},\ s = \Set{ \{\Si_0, \Si_1\}, \{\Si_0, \Si_1, \Si_2\} }$,
\item $\tau = \{\Si_0, \Si_2\},\ s = \Set{ \{\Si_0, \Si_2\}, \{\Si_0, \Si_1, \Si_2\} }$, and
\item $\tau = \{\Si_0, \Si_1, \Si_2\},\ s = \Set{ \{\Si_0, \Si_1, \Si_2\} }$.
\end{enumerate}
For each case that $\Si = \Si_1$ and $\Si = \Si_2$, we have similar five combinations of pairs $(\tau, s)$.
The decompositions corresponding to all of these combinations of $(\tau, s)$ of the range of $\psi_{\si, R}$ are described in \cref{figure : range of psi_si}.
The lower case Roman numerals in \cref{figure : range of psi_si} correspond to Arabic numerals above.

\begin{figure}
\begin{center}
\begin{tikzpicture}
[xscale = 0.7, yscale = 0.7]
\draw [->] [thick](3, 0) -- (4, 0);
\draw [->] [thick](0, 3) -- (0, 4);
\draw [->] [thick](-1.5, -1.5) -- (-2, -2);
\draw [very thick](0, 3) -- (3, 3);
\draw [very thick](3, 0) -- (3, 3);
\draw [very thick](-1.5, 1.5) -- (-1.5, -1.5);
\draw [very thick](-1.5, 1.5) -- (0, 3);
\draw [very thick](-1.5, 1.5) -- (1.5, 1.5);
\draw [very thick](3, 3) -- (1.5, 1.5);
\draw [very thick](-1.5, -1.5) -- (1.5, -1.5);
\draw [very thick](1.5, 1.5) -- (1.5, -1.5);
\draw [very thick](3, 0) -- (1.5, -1.5);
\draw [very thick](-1.5/2, 1.5+1.5/2) -- (3-1.5/2, 1.5+1.5/2);
\draw [very thick](1.5, 3) -- (0, 1.5);
\draw [very thick](3-1.5/2, -1.5+1.5/2) -- (3-1.5/2, 1.5+1.5/2);
\draw [very thick](0, -1.5) -- (0, 1.5);
\draw [very thick](-1.5, 0) -- (1.5, 0);
\draw [very thick](3, 1.5) -- (1.5, 0);
\draw [very thick](0, 0) -- (-1.5, -1.5);
\draw [very thick](0, 3) -- (1.5/2, 1.5+1.5/2);
\draw [very thick](3, 0) -- (3-1.5/2, 1.5/2);
\draw(-1.9,-2.3) node {$x_{\Si_0}$};
\draw(4.5,-0.1) node {$x_{\Si_1}$};
\draw(0.2,4.2) node {$x_{\Si_2}$};
\draw(-0.5,-1.1) node {\small{(i)}};
\draw(-1.1,-0.5) node {\small{(ii)}};
\draw(0.7,-0.7) node {\small{(iii)}};
\draw(-0.7,0.7) node {\small{(iv)}};
\draw(0.7,0.7) node {\small{(v)}};
\fill [gray, opacity=.3] (0, 3) -- (3, 3) -- (3-3/2, 3-3/2) -- (-3/2,3-3/2);
\fill [gray, opacity=.3] (3, 3) -- (3, 0) -- (3-3/2, -3/2) -- (3-3/2, 3-3/2);
\fill [gray, opacity=.2] (3/2, 3/2) -- (3/2, 3/2-3) -- (3/2-3, 3/2-3) -- (3/2-3, 3/2);
\end{tikzpicture}
\end{center}
\caption{The range of $\psi_{\si}$ for $\si = \{ \Si_0, \Si_1, \Si_2\}$}
\label{figure : range of psi_si}
\end{figure}

\end{ex}

\subsection{Vanishing of solutions to the \SW equations}

\label{Vanishing property of the \SW moduli space on a family}

A basic tool to construct the cohomological invariant is the vanishing property of the \SW moduli space for a metric obtained by stretching neighborhoods of embedded surfaces violating the adjunction inequalities.
This is originally due to Kronheimer--Mrowka's paper on the Thom conjecture~\cite{MR1306022}.
While a cylinder with infinite length was used in \cite{MR1306022}, we here need some quantitative estimate for the length of stretched cylinders.
Such an analytical argument has been given in the author's paper \cite{Konno}.
In this subsection we adjust it to the setting of this paper.

We use the following terminology and notation on the \SW equations.
Fix a $\spc$ structure $\fraks$ on a smooth oriented closed connected $4$-manifold $X$.
Let $h$ be a Riemannian metric on $X$.
We denote by $\Omega^+ = \Omega^+_h = \Gamma(\Lambda^+_h) $ the space of self-dual $2$-forms on $X$ with respect to $h$.
Let $S^\pm \to X$ and $L \to X$ denote the spinor bundles and the determinant line bundle of $\fraks$ respectively.
A $U(1)$-connection $A$ on $L$ gives rise to the Dirac operator $D_A : \Gamma(S^+) \to \Gamma(S^-)$.
For a $U(1)$-connection $A$, a positive spinor $\Phi \in \Gamma(S^+)$, and an imaginary self-dual $2$-form $\mu \in i\Omega^+$, we call the equations
\begin{align*}
\begin{cases}
\rho(F^+_A - \mu) = \sigma(\Phi, \Phi),\\
D_A \Phi = 0
\end{cases}
\end{align*}
the (perturbed) \SW equations with respect to $(h, \mu)$.
Here $\rho : \Lambda^+ \to \mathfrak{su}(S^+)$ is the map obtained from the Clifford multiplication, $F^+_A$ is the self-dual part of the curvature $F_A$ of $A$, and $\sigma(\cdot, \cdot)$ is the quadratic form given by $\sigma(\Phi, \Phi) = \Phi \otimes \Phi^\ast - |\Phi|^2 \id/2$.
For $A$ and $\Phi$, we call the equations
\begin{align*}
\begin{cases}
\rho(F^+_A) = \sigma(\Phi, \Phi),\\
D_A \Phi = 0
\end{cases}
\end{align*}
the unperturbed \SW equations with respect to $h$.

We here need the following analytical \cref{lem: cylinder,lem: L^2-a priori}.
Let $\|\cdot\|_{L^2(X, h)}$ denotes the $L^2$-norm on $X$ with respect to a given metric $h$.
For a surface $\Si$ and $R > 0$, we equip $\Si$ with the metric of constant scalar curvature and of unit area, $S^1$ with the metric of unit length, and $[0,R] \times S^1 \times \Si$ with the standard product metric as in \cref{subsection: Construction of a family of Riemannian metrics}.
The statement of \cref{lem: cylinder} appears in the proof of Lemma~4 in Kronheimer--Mrowka~\cite{MR1492131}.

\begin{lem}[Kronheimer--Mrowka~\cite{MR1492131}]
\label{lem: cylinder}
Let $h$ be a metric on $X$, $\Si$ be a surface embedded in $X$ with $[\Si]^2 = 0$, and $R$ be a positive number.
Suppose that $(X, h)$ contains a Rimannian submanifold $(X', h)$ (with boundary) which is isometric to $[0,R] \times S^1 \times \Si$.
Then, for any closed 2-form $\omega$ on $X$, the inequality
\begin{align}
| [\omega] \cdot [\Si] |^2 \leq \frac{\|\omega\|^2_{L^2(X', h)}}{R}
\label{eq: key estimate}
\end{align}
holds.
\end{lem}

\begin{proof}
By Fubini's theorem, we have
\begin{align*}
\|\omega\|^2_{L^2(X', h)}
& = \int_{[0,R] \times S^1\times {\Sigma}} |\omega|^2 d\mu_{[0,R] \times S^1\times {\Sigma}}\\
& = \int_{[0,R] \times S^1} \left(\int_{\{(t,\theta)\}\times{\Sigma}} |\omega|^2 d\mu_{\{(t,\theta)\}\times\Sigma}\right)d\mu_{[0,R]\times S^1},
\end{align*}
where $d\mu$ denotes the volume form.
Note that $\Sigma$ and $\{(t,\theta)\}\times{\Sigma}$ have the same homology class in $X$.
Therefore, by the Cauchy--Schwarz inequality, we obtain
\begin{align*}
\Area({\Sigma})\cdot\int_{\{(t,\theta)\}\times{\Sigma}}|\omega|^2d\mu_{\{(t,\theta)\}\times\Sigma}\geq\left|\int_{\{(t,\theta)\}\times{\Sigma}}\omega\right|^2=\left|\int_{{\Sigma}}\omega\right|^2.
\end{align*}
Therefore
\begin{align*}
\|\omega\|^2_{L^2(X', h)}
& \geq \frac{1}{\Area({\Sigma})}\int_{[0,R]\times S^1}\left|\int_{{\Sigma}}\omega\right|^2d\mu_{[0,R]\times S^1}\\
&= \frac{\Area([0,R]\times S^1)}{\Area({\Sigma})}\left|\int_{{\Sigma}}\omega\right|^2
\end{align*}
holds.
Since $\Area([0,R]\times S^1)=R$ and $\Area(\Si)=1$, this proves the lemma.
\end{proof}

For a cohomology class $c \in H^2(X; \R)$ and a metric $h$ on $X$, we denote by $\calH_h(c)$ the harmonic representative of $c$ with respect to $h$.
We denote  by $s(h) : X \to \R$ the scalar curvature of $h$.
Let us define $\kappa(h) : X \to \R$ by
\[
\kappa(h) := \max\{ -s(h), 0 \}.
\]
The following \cref{lem: L^2-a priori} is essentially proven
as Lemma~3 in Kronheimer--Mrowka~\cite{MR1492131}, while in \cite{MR1492131} the \SW equations on $3$-manifolds were considered.

\begin{lem}[Kronheimer--Mrowka~\cite{MR1492131}]
\label{lem: L^2-a priori}
If there exists a solution to the \SW equations with respect to a metric $h$, then the inequality
\begin{align}
\|\calH_h(c_1(\fraks))\|_{L^2(X, h)}^2
\leq \frac{\| \kappa(h) \|_{L^2(X, h)}^2}{(4\pi)^2} - \int_X c_1(\fraks)^2
\label{eq: char-estimate}
\end{align}
holds.
\end{lem}

\begin{proof}
Write $\kappa = \kappa(h), s=s(h)$ and $L^2=L^2(X,h)$.
Let $(A,\Phi)$ be a solution to the \SW equations with respect to $h$ and $\fraks$.
A standard argument combining the \SW equations with the Weitzenb\"{o}ck formula gives the equality
\begin{align*}
\int_{X}|\nabla_{A}\Phi|^{2}+\frac{1}{4}\int_{X}s|\Phi|^{2}+\frac{1}{4}\int_{X}|\Phi|^{4}=0,
\end{align*}
and hence an inequality
\begin{align*}
\int_{X}|\Phi|^{4} \leq -\int_{X}s|\Phi|^{2}.
\end{align*}
This combined with the Cauchy--Schwarz inequality implies that
\begin{align}
\int_{X}|\Phi|^{4} \leq \int_{X}\kappa^{2}.
\label{L^2}
\end{align}
%
Rewriting $F_{A}^{+}$ in terms of $\Phi$ using the \SW equations, we have
\begin{align}
\label{eq: FA SW}
|F_A^+| \leq \frac{|\sigma(\Phi,\Phi)|}{2} \leq \frac{|\Phi|^2}{2\sqrt{2}}
\end{align}
at each point.
Now an $L^2$-a priori estimate for curvatures
\begin{align}
\| F_{A}^{+} \|_{L^2} \leq \frac{\| \kappa \|_{L^2}}{2\sqrt{2}}
\label{L^2-estimate}
\end{align}
follows from \eqref{L^2} and \eqref{eq: FA SW}. 

Set $c= c_{1}(\fraks)$.
Take a representative $\omega \in \Omega^2(X, \R)$ of $c$.
We have 
\[
\|\omega\|_{L^2}^2 = \|\omega^+\|_{L^2}^2 + \|\omega^-\|_{L^2}^2
\]
and 
\begin{align*}
\int_X c^2
= \int_X \omega \wedge \omega
= \|\omega^+\|_{L^2}^2 - \|\omega^-\|_{L^2}^2,
\end{align*}
hence
\begin{align}
\|\omega\|_{L^2}^2 = 2\|\omega^+\|_{L^2}^2-\int_X c^2
\label{F_A}
\end{align}
holds.
In addition, we have
\begin{align}
\|\calH_h(c)\|_{L^2} \leq \|\omega\|_{L^2}
\label{c-F_A}
\end{align}
since the $h$-harmonic form minimizes $L^2_h$-norm in the forms having the same cohomology class.
Since we can take $\omega$ as $\omega = (-1/2\pi i)F_A$, from (\ref{L^2-estimate}), (\ref{F_A}) and (\ref{c-F_A}), we obtain
\begin{align*}
\|\calH_h(c)\|_{L^2}^{2}
\leq \|\omega\|_{L^2}^{2}
&= 2\|\omega^+\|_{L^2}^2-\int_X c^2\\
&= \frac{1}{2\pi^2}\|F_A^+\|_{L^2}^2-\int_X c^2\\
&\leq \frac{\| \kappa \|_{L^2}^2}{16\pi^2}-\int_X c^2.
\end{align*}
This proves the lemma.
\end{proof}

Henceforth in this subsection, we fix a metric $g \in \Met(X)$.
We also fix $N(\cdot)$, $a(\cdot)$ as in \cref{subsection: Construction of a family of Riemannian metrics}.
For a simplex $\si$ in $\calK$,
we define
\[
\lambda(\si) = \lambda(\si, a(\cdot)) := \min_{\substack{ s \in S(\Bd(\si)),\\ \scrS \in S(\Bd(s))}}\Set{ \lambda(\scrS, \bvec{t}) | \substack{\bvec{t} = \{t_j\}_{j=0}^{\dim{\scrS}},\\ \sum_{j=0}^{\dim \scrS} t_j = 1, \\ t_j \geq 0}}
\]
and
\[
C(\si) = C(\si, g, N(\cdot), a(\cdot)) := \max_{h \in \phi_{\si}(|\si|)} \frac{\| \kappa(h) \|_{L^2(X, h)}^2}{(4\pi)^2} - \int_X c_1(\fraks)^2,
\]
where we regard
\[
|\si| = \bigcup_{s \in S(\Bd(\si))} |s| \times \{0\}^{\simin{s}}.
\]
These real numbers $\lambda(\si, a(\cdot))$ and $C(\si, g, N(\cdot), a(\cdot))$ depend only on $\si$ and $a(\cdot)$, and on $\si$, $g, N(\cdot)$, and $a(\cdot)$ respectively
since we have fixed the $\spc$ structure $\fraks$ in this subsection.
Note that $\lambda(\si) > 0$ holds.
We also note that $|C(\si)|<\infty$ holds since $\phi_{\si}(|\si|)$ is compact and the map $h \mapsto \| \kappa(h) \|_{L^2(X, h)}^2$ is continuous.
Set
\[
\bar{R}(\si) = \bar{R}(\si, g, N(\cdot), a(\cdot)) := \max\Set{\frac{C(\si, g, N(\cdot), a(\cdot))}{\lambda(\si, a(\cdot))}, 0 }
\]
and
\[
R(\si) = R(\si, g, N(\cdot), a(\cdot)) := \max_{\tau \prec \si} \bar{R}(\tau, g, N(\cdot), a(\cdot)).
\]

\begin{prop}
\label{prop: vanishing on certain face}
Let $\si$ be a simplex of $\calK$.
Then, for any $R \geq \bar{R}(\si)$, there is no solution to the unperturbed \SW equations for any metric in $\phi_{\si}(\frakextdel(\si, R))$.
\end{prop}

\begin{proof}
We argue by contradiction.
Assume that for some non-negative number $R \geq \bar{R}(\si)$ and some metric $h \in \phi_{\si}(\frakextdel(\si, R))$, there exists a solution to the unperturbed \SW equations for $h$.
For this $h$, there exist 
\begin{align*}
s \in S(\Bd(\si)),\ l \geq 0,\ \scrS = \{ s_0, \cdots, s_l\} \in S(\Bd(s)),\\ 
\{r_{\Si}\}_{\Si \in \simin{s}} \in \frakextdel(\si, s, R),\ \bvec{t} = \{t_j\}_{j=0}^l
\end{align*}
such that $t_j \geq 0$, $\sum_{j=0}^l t_j=1$, and $h = \phi_{\si}(p, \{r_{\Si}\})$, where $p = \sum_{j=0}^l t_j s_j$.
From \cref{lem: cylindrical structure for any point on the family}, we have
\begin{align}
\bigsqcup_{\Si \in \simin{s}} (\bar{X}(\lambda(\scrS, \bvec{t}), \Si), \phi_{\si}(p, \{r_{\Si}\})) \cong \bigsqcup_{\Si \in \simin{s}} \barCystd(\lambda(\scrS, \bvec{t}), \Si, r_{\Si}).
\label{eq: isometry in proof}
\end{align}
Here $\cong$ means an isometry.
We set $r'_{\Si} := \Lambda(\lambda(\scrS, \bvec{t}), r_{\Si})$.
Since $\Euler^-(\Si)^2 + 1 \leq |c_1(\fraks) \cdot [\Si]|^2$ holds for each $\Si \in \simin{s}$, we have
\begin{align}
 \sum_{\Si \in \simin{s}} r'_{\Si} \cdot \Euler^-(\Si)^2 + \sum_{\Si \in \simin{s}} r'_{\Si}
\leq & \sum_{\Si \in \simin{s}} r'_{\Si} \cdot |c_1(\fraks) \cdot [\Si]|^2  \label{eq: 1st inequality in prop}\\
\leq & \sum_{\Si \in \simin{s}} \|\calH_h(c_1(\fraks))\|_{L^2(\bar{X}(\lambda(\scrS, \bvec{t}), \Si), h)}^2 \nonumber\\
\leq & \|\calH_h(c_1(\fraks))\|_{L^2(X, h)}^2 \nonumber\\
\leq & \frac{\| \kappa(h) \|_{L^2(X, h)}^2}{(4\pi)^2} - \int_X c_1(\fraks)^2,\nonumber
\end{align}
where the second and the last inequality follow from \cref{lem: cylinder} and \cref{lem: L^2-a priori} respectively.
Set $X' = X \setminus \bigsqcup_{\Si \in \simin{s}} \bar{X}(\lambda(\scrS, \bvec{t}), \Si)$ and $\{0\} = \{0\}^{\simin{s}}$.
We also have
\begin{align}
&\| \kappa(h) \|_{L^2(X, h)}^2 \label{eq: 2nd inequality in prop}\\
= & \sum_{\Si \in \simin{s}}\| \kappa(h) \|_{L^2(\bar{X}(\lambda(\scrS, \bvec{t}), \Si), h)}^2 + \| \kappa(h) \|_{L^2(X', h)}^2 \nonumber\\
= & \sum_{\Si \in \simin{s}} (4\pi)^2 r'_{\Si} \cdot \Euler^-(\Si)^2 +  \| \kappa(\phi_{\si}(p, \{0\})) \|_{L^2(X', \phi_{\si}(p, \{0\}))}^2 \nonumber\\
\leq &  \sum_{\Si \in \simin{s}} (4\pi)^2 r'_{\Si} \cdot \Euler^-(\Si)^2 +  \| \kappa(\phi_{\si}(p, \{0\})) \|_{L^2(X, \phi_{\si}(p, \{0\}))}^2. \nonumber
\end{align}
From \eqref{eq: 1st inequality in prop} and \eqref{eq: 2nd inequality in prop}, we deduce
\[
\sum_{\Si \in \simin{s}} r'_{\Si}
\leq \frac{\| \kappa(\phi_{\si}(p, \{0\}) \|_{L^2(X, (\phi_{\si}(p, \{0\}))}^2}{(4\pi)^2} - \int_X c_1(\fraks)^2
\leq C(\si).
\]
However, since we also have
\begin{align*}
\sum_{\Si \in \simin{s}} r'_{\Si}
\geq \max_{\Si \in \simin{s}} r_{\Si}'
&> \lambda(\scrS, \bvec{t})\left( \max_{\Si \in \simin{s}}r_{\Si} + 1\right)\\
&\geq \lambda(\si)(R+1)
> \lambda(\si)\bar{R}(\si),
\end{align*}
it follows that $\bar{R}(\si) < C(\si)/\lambda(\si)$.
This contradicts the definition of $\bar{R}(\si)$.
\end{proof}

\section{Construction of the invariant I\hspace{-.1em}I:  Perturbations and counting arguments}

\label{section: Construction of the cohomological invariant}

In this section we complete the construction of the cohomological invariant.
We consider the moduli spaces of the \SW equations parameterized by the families of Riemannian metrics given in \cref{section: Family of Riemannian metrics}; we construct a cohomology class by counting the points of the parameterized moduli spaces.

\subsection{Definition of the invariant in a simple case}

\label{subsection: Definition of the invariant}

Let $X$ be a smooth oriented closed connected $4$-manifold equipped with a homology orientation, $\fraks$ be a $\spc$ structure on $X$, and $n$ be the integer given by $d(\fraks) = -(n+1)$.
We assume that $n \geq 0$ and that
$b^+(X) \geq n+3$.
(These assumptions will be needed since we will use $(n+1)$-parameter families in order to define the cohomological invariant.)
We also assume that $V(\calK(X, \fraks)) \neq \emptyset$.
In this subsection we define the cohomological invariant
\[
\SWcoh(X, \fraks) \in H^{n}(\calK(X, \fraks);\Z)
\]
under this setting.
Although we will consider more general $\spc$ structures in \cref{subsection: Definition of the invariant for general spinc structures} using what is called $\mu$-maps, the case which we treat in this subsection is the most basic one; we can give a non-trivial example for this case in \cref{section: Non-vanishing results for the cohomological invariant}.

We fix a metric $g \in \Met(X)$ and data $N(\cdot)$, $a(\cdot)$ as explained in \cref{subsection: Construction of a family of Riemannian metrics} throughout this subsection.
We also fix certain kinds of perturbations $\vp_{\bullet}$ explained below to define a cochain on $\calK$.
While in \cref{Vanishing property of the \SW moduli space on a family} we have considered only the unperturbed \SW equations,
we now need the perturbed ones.
Fix an integer $l \geq 2$.
The space of perturbations $\Pi = \Pi(X, \fraks)$ is given by
\begin{align}
\Pi(X, \fraks) := \Set{ (h, \mu) \in \Met(X) \times i\Omega^2 | \mu \in i\Omega^+_h } = \bigsqcup_{h \in \Met(X)}i\Omega^+_h,
\label{eq def of the space of perturbations}
\end{align}
where $\Omega^+_h$ is the completion by the $L^{2}_{l-1}$-Sobolev norm defined by $h$ of the space of self-dual $2$-forms with respect to $h$.
(We here omit ``$L^{2}_{l-1}(\cdot)$'' from our notation for simplicity.)
The space $\Pi \to \Met(X)$ is a Hilbert bundle over $\Met(X)$.
Let us identify $\Met(X)$ with the zero-section of this bundle,
and denote by $\Pi_{h}$ the fiber $i\Omega^+_h$ of this bundle $\Pi$ at $h \in \Met(X)$.
Fix a smooth reference connection $A_0$ of the determinant line bundle.
We call the subset of perturbations $\calW = \calW(X, \fraks) \subset \Pi(X, \fraks)$ defined by
\begin{align}
\calW(X, \fraks) = \bigsqcup_{h \in \Met(X)}\calW(X, \fraks)_h, \quad
\calW(X, \fraks)_h := F_{A_0}^{+_h} + i\im{d^{+_h}} \subset i\Omega^+_h
\label{eq def of wall}
\end{align}
the {\it wall}.
We define $\circPi = \bigsqcup_{h \in \Met(X)} \circPi_h$ by
\[
\circPi_h := \Pi_h \setminus \calW_h.
 \]
The wall $\calW$ is of codimension-$b^+(X)$ in $\Pi$ since $\Omega^+_h/\im(d^{+_h}) \cong H^+(X;\R)$ for each $h$.
Recall that, for $(h, \mu) \in \Pi$, the perturbed \SW equations with respect to $(h, \mu)$ have a reducible solution if and only if $(h, \mu) \in \calW$.

For $k \in \{-1, 0, \ldots, n\}$, we now define maps
\begin{align}
\Set{ \vp_{\si} : [0, R(\si)]^{\si} \to \circPi }_{\si \in S_{k}(\calK)}
\label{eq: family of perturbations}
\end{align}
inductively on $k$ as follows.
Here we define $S_{-1}(\calK)$ as $\{\emptyset\}$; we regard $\vp_{\emptyset}$ as a map $\vp_{\emptyset} : \{{\rm pt}\} \to \circPi$.
In the case when $k = -1$, take a generic point $\ast$ in $\circPi$ so that the moduli space for the perturbed \SW equations with respect to $\ast$ is empty.
(Recall that the formal dimension $d(\fraks)$ is assumed to be negative.)
Then we can define $\vp_{\emptyset} : \{{\rm pt}\} \to \circPi$ by $\vp_{\emptyset}({\rm pt}) = \ast$.
Next we define $\vp_{\bullet}$ in the case that $k = 0$.
Let $\{\Si\}  \in S_0(\calK)$.
We here simply write it $\Si$.
Note that there is no solution to the unperturbed \SW equations for the metric $\phi_{\Si}(|\{\Si\}| \times \{R(\Si)\})$ from \cref{prop: vanishing on certain face}.
Take a generic path from $\ast$ to $\phi_{\Si}(|\{\Si\}| \times \{R(\Si)\})$ in $\circPi$ so that the parameterized moduli space with respect to this path is smooth.
(Now we regard $\Met(X)$ as a subset of $\Pi$.)
Here, if $d(\fraks)<1$, the smoothness of the parameterized moduli space means just that it is empty.
This construction gives a map
\[
\vp_{\Si} : [0, R(\Si)]^{\{\Si\}} = [0, R(\Si)] \to \circPi.
\]
Note that $\vp_{\Si}|_{\{0\}} = \vp_{\emptyset}$ obviously holds.
Next, for $k \in \{1, \ldots, n\}$ assume that $\vp_{\tau}$ is already constructed for any $l \leq k-1$ and any $\tau \in S_l(\calK)$ and that the restriction $\vp_{\tau}|_{\extdel(\tau, R(\tau))}$ coincides with the composition
\begin{align*}
\phi_{\tau} \circ \psi_{\tau, R(\tau)}^{-1} : \extdel(\tau, R(\tau)) \cong \frakextdel(\tau, R(\tau)) \to \circPi.
\end{align*}
Here $\psi_{\tau, R(\tau)}$ is the homeomorphism given in \cref{prop: structure of the subspace of the domain of phi}.
We note that the image of this map $\phi_{\tau} \circ \psi_{\tau, R(\tau)}^{-1}$ is a priori contained in $\Met(X)$, however in fact in $\circPi$.
Indeed, from \cref{prop: vanishing on certain face}, there is no solution to the unperturbed \SW equations for any metric in the image of the map $\phi_{\tau} \circ \psi_{\tau, R(\tau)}^{-1}$.
In particular, reducible solutions do not appear on the image, therefore it is contained in $\circPi$.
(Here we regard $\Met(X)$ as a subset of $\Pi$.)
We shall construct $\vp_{\bullet}$ for $k$-simplices under these assumptions.
Let $\si \in S_k(\calK)$.
Note that $[0, R(\si)]^{\si}$ decomposes as
\begin{align}
&[0, R(\si)]^{\si} \label{eq: decomposition of the cube to interior and boundary} \\
= & \del{\left([0, R(\si)]^{\si}\right)} \cup (0, R(\si))^{\si} \nonumber \\
= &
\extdel(\si, R(\si))
\cup
 \bigcup_{\substack{\tau \prec \si, \\ \dim\tau = k-1}} [0, R(\si)]^{\tau}
\cup
(0, R(\si))^{\si} \nonumber \\
= &
\extdel(\si, R(\si))
\cup
 \bigcup_{\substack{\tau \prec \si, \\ \dim\tau = k-1}} \left( [0, R(\tau)]^{\tau} \cup \bigsqcup_{R(\tau) \leq R \leq R(\si)} \extdel(\tau, R) \right)
\cup
(0, R(\si))^{\si}.  \nonumber
\end{align}
Here we regard $[0,R(\tau)]^{\tau}$ as a subset of $[0,R(\si)]^{\si}$ by inserting some factors $\{0\}$ in appropriate places.
We take a map $\vp_{\si}$ satisfying following three conditions:
\begin{cond}
\label{cond: restriction of perturbations}
$\quad$
\begin{enumerate}
\item
\label{item: extboundary}
On $\extdel(\si, R(\si))$, the map $\vp_{\si}$ coincides with the composition
\[
\phi_{\si} \circ \psi_{\si, R(\si)}^{-1} : \extdel(\si, R(\si)) \to \circPi.
\]

\item 
\label{item: interior boundary}
For any face $\tau \prec \si$ with $\dim \tau = k-1$,
\begin{enumerate}
\item on $[0, R(\tau)]^{\tau}$, the map $\vp_{\si}$ coincides with $\vp_{\tau} : [0, R(\tau)]^{\tau} \to \circPi$, and
\item on $\bigsqcup_{R(\tau) \leq R \leq R(\si)} \extdel(\tau, R)$, the map $\vp_{\si}$ coincides with the 
composition 
\[
\phi_{\tau} \circ \left( \bigsqcup_{R(\tau) \leq R \leq R(\si)} \psi_{\tau, R}^{-1} \right) : \bigsqcup_{R(\tau) \leq R \leq R(\si)} \extdel(\tau, R) \to \circPi.
\]
\label{eq: composition of disjoint maps}
\end{enumerate}

\item The restriction of $\vp_{\si}$ to $(0, R(\si))^{\si}$ is generic in the sense that the parameterized moduli space with respect to $\vp_{\si}$ is smooth over $(0, R(\si))^{\si}$.
\label{item: genericity in interior}
\end{enumerate}
\end{cond}

By the construction of the homeomorphism given in \cref{prop: structure of the subspace of the domain of phi}, the map in the case of \eqref{eq: composition of disjoint maps} above is continuous.
The maps in other cases are obviously continuous, and so is the whole map $\vp_{\si}$.
Continuity of $\vp_{\si}$ and the compactness of $[0, R(\si)]^{\si}$ imply that the parameterized moduli space, defined soon later, with respect to $\vp_{\si}$  is compact.

The procedure above can be continued to $k = n+1$ since the codimension of $\calW$ is $b^{+}$ and we have assumed that $b^{+} > n+2$; however, we stop it  at $k = n$.
(In the subsequent subsections, we will use cobordism arguments, and therefore consider $(n+2)$-parameter families.
Our hypothesis $b^{+} \geq n+3$ will be needed in those arguments.)
We eventually obtain perturbations \eqref{eq: family of perturbations} for each $k \in \{-1, 0, \ldots, n\}$.

\begin{lem}
\label{lem: vanishing on boundary}
Let $\si \in S_{k}(\calK)$ be a simplex of $\calK$ with $k \in \{0, \ldots, n\}$.
Then there is no solution to the \SW equations for any element in $\vp_{\si}\mleft( \del{\left([0, R(\si)]^{\si}\right)} \mright)$.
There is also no solution to the \SW equations for $\vp_{\emptyset}({\rm pt})$.
\end{lem}

\begin{proof}
The proof in the case that $k = -1$ is obvious since $\vp_{\emptyset}({\rm pt}) = \ast$ and we choose $\ast$ so that the moduli space with respect to $\ast$ is empty.
For $k \in \{0, \ldots, n\}$, the proof is by induction on $k$ as follows.
First take $\Si \in S_{0}(\calK)$.
Then we have $\vp_{\Si}\mleft( \del{\left([0, R(\Si)]^{\Si}\right)} \mright) = \{ \ast \} \cup \phi_{\Si}(|\{\Si\}| \times \{R(\Si)\})$.
 There is no solution for $\ast$ as above, and also for $\phi_{\Si}(|\{\Si\}| \times \{R(\Si)\})$ by \cref{prop: vanishing on certain face}.
Next, for $k \in \{1, \ldots, n\}$ assume that the statement holds for any $l \leq k-1$.
Take $\si \in S_{k}(\calK)$.
For elements in $\extdel(\si, R(\si))$ and $\bigsqcup_{R(\tau) \leq R \leq R(\si)} \extdel(\tau, R)$ for some $\tau \prec \si$ with $\dim{\tau} = k-1$, the non-existence of solutions follows from \cref{prop: vanishing on certain face}.
For elements in $\del([0, R(\tau)]^{\tau})$, the non-existence follows from the induction hypothesis.
For elements in $(0, R(\tau))^{\tau}$, the non-existence follows from that $\dim\tau<n$ and that the formal dimension for the unparametrized Seiberg--Witten equations is given by $d(\fraks)=-(n+1)$.
\end{proof}

Let $(h, \mu) \in \Pi$.
We denote by $\calC(h) = \calC(h, \fraks)$ the $L^{2}_{l}$-completion defined by $h$ of the space of pairs consisting of a $U(1)$-connection of the determinant line bundle and a positive spinor for $\fraks$.
The gauge group $\calG(h)$ is defined as the $L^{2}_{l+1}$-completion using $h$ of $\Map(X, U(1))$.
We define $\calB(h) = \calB(h, \fraks)$ by $\calB(h) := \calC(h)/\calG(h)$.
We will denote by $\calM(h, \mu) = \calM((h, \mu), \fraks) \subset \calB(h)$ the moduli space for $(h, \mu)$ defined as the quotient of the space of solutions to the \SW equations with respect to $(h, \mu)$ divided by $\calG(h)$ as usual.
Let $\vp : T \to \Pi$ be a map from a topological space $T$, which is called the {\it parameter space}.
We define the parameterized moduli space parameterized by $\vp$ on $T$ as
\[
\calM(\vp) = \calM(\vp, \fraks) := \bigsqcup_{t \in T} \calM(\vp(t), \fraks).
\]

Now let us consider the parameterized moduli space given by $\vp_{\bullet}$.
Let $\si \in S_{k}(\calK)$ a simplex of $\calK$ with $k \in \{0, \ldots, n\}$.
The parameterized moduli space $\calM(\vp_{\si})$ is compact because of the compactness argument for the usual \SW theory and  the compactness of the parameter space $[0, R(\si)]^{\si}$.
Since we avoid the wall, $\calM(\vp_{\si})$ does not contain a reducible solution.
We here note that
\[
\calM(\vp_{\si}, \fraks) = \bigsqcup_{t \in (0, R(\si))^{\si}} \calM(\vp_{\si}(t), \fraks)
\]
follows from \cref{lem: vanishing on boundary}.
Recall that $\vp_{\si}$ is generic on $(0, R(\si))^{\si}$ and that $d(\fraks) = -(n+1)$.
We therefore have $\calM(\vp_{\si}) = \emptyset$ if $k \leq n - 1$, and $\calM(\vp_{\si})$ is a $0$-dimensional manifold if $k = n$.

Let $\lb \si \rb$ an oriented $n$-simplex of $\calK$.
Then we have an orientation of $[0, R(\si)]^{\si}$.
Putting the given homology orientation and this orientation together, we obtain the orientation of $\calM(\vp_{\si})$.
(See, for example, Appendix A in Salamon's book~\cite{Salamon}.
We follow the orientation convention in this book.)
We now eventually have an oriented compact $0$-dimensional manifold $\calM(\vp_{\si})$ from $\lb \si \rb$.
We will denote by
\[
\SWinv(\vp_{\si}) = \SWinv(\vp_{\si}, \fraks) \in \Z
\]
the algebraic count $\#\calM(\vp_{\si})$.
Set $\calA = (g, N(\cdot), a(\cdot),  \vp_{\bullet})$, which
we call an {\it additional data}.
We define an $n$-cochain
\[
\SWcoch(X, \fraks, \calA) \in C^n(\calK)
\]
by
\[
C_n(\calK) \to \Z\ ;\ \lb \si \rb \mapsto \SWinv(\vp_{\si}).
\]

We will prove the following \lcnamecref{prop: coclosedness} in \cref{subsection: Arguments by cobordisms}:

\begin{prop}
\label{prop: coclosedness}
The cochain $\SWcoch(X, \fraks, \calA)$ constructed above is a cocycle.
\end{prop}

We now may write the definition of the cohomological invariant:

\begin{defi}
\label{defi: main definition}
Let $X$ be a smooth oriented closed connected $4$-manifold equipped with a homology orientation, $\fraks$ be a $\spc$ structure on $X$, and $n$ be the integer given by $d(\fraks) = -(n+1)$.
Assume that $n \geq 0$, $b^+(X) \geq n+3$, and $V(\calK(X, \fraks)) \neq \emptyset$.
We define a cohomology class $\SWcoh(X, \fraks) \in H^n(\calK(X, \fraks) ;\Z)$ by
\[
\SWcoh(X, \fraks) := [\SWcoch(X, \fraks, \calA)],
\]
where $\calA = (g, N(\cdot), a(\cdot),  \vp_{\bullet})$ is an additional data.
\end{defi}

In fact, $\SWcoh(X, \fraks)$ is independent of the choice of $\calA$:

\begin{theo}
\label{theo: well-definedness of the invariant}
The cohomology class $\SWcoh(X, \fraks) \in H^n(\calK(X, \fraks) ;\Z)$ is an invariant of a $spin^{c}$ $4$-manifold $(X, \fraks)$.
Namely the cohomology class $[\SWcoch(X, \fraks, \calA)]$ in \cref{defi: main definition} is independent of the choice of $\calA$. 
\end{theo}

We will prove \cref{theo: well-definedness of the invariant} in \cref{subsection: Arguments by cobordisms}.

\begin{rem}
To define the number $\SWinv(\vp_{\si}) \in \Z$ used to define $\SWcoch(X, \fraks, \calA)$, in fact we do not need to assume that $\vp_{\si}$ is generic on $(0, R(\si))^{\si}$ for each $n$-simplex $\si$.
To justify the counting argument in the case that $\vp_{\si}$ is not generic, one can use Ruan's virtual neighborhood technique~\cite{MR1635698}.
More precisely, one needs a family version of the technique (for example, see~\cite{Konno}) in our case.
An advantage of the use of the virtual neighborhood technique is that we can use a map $\vp_{\si}$ which is continuous (not necessarily smooth) on $(0, R(\si))^{\si}$ to define $\SWcoch(X, \fraks, \calA)$.
This is based on the same mechanism in the following linear situation:
to define the index of a family of Fredholm operators, we do not need to assume that the family {\it smoothly} depends on its parameter, but just need to assume it {\it continuously} depends on it.
\end{rem}

\subsection{Arguments by cobordisms}
\label{subsection: Arguments by cobordisms}

The purpose of this subsection is to prove \cref{prop: coclosedness,theo: well-definedness of the invariant}.
Both of them are shown using arguments by cobordisms.
We follow the all settings and the notation of \cref{subsection: Definition of the invariant}.

\begin{proof}[Proof of \cref{prop: coclosedness}]
As we noted in \cref{subsection: Definition of the invariant}, perturbations $\vp_{\bullet}$ can be constructed for $(n+1)$-simplices since $b^{+} \geq n+3$.
Let us fix $\vp_{\bullet}$ for $(n+1)$-simplices through the inductive procedure in \cref{subsection: Definition of the invariant}.

Let us take an oriented $(n+1)$-simplex $\lb \si \rb = \lb \Si_0 , \ldots, \Si_{n+1} \rb$ of $\calK$.
Write $\lb \tau_i \rb = \lb \Si_{0}, \ldots, \hat{\Si}_{i} ,\ldots,\Si_{n+1} \rb$, then we have
\begin{align}
\SWcoch(X, \fraks, \calA)(\del \lb \si \rb) = \sum_{i = 0}^{n+1} (-1)^i \SWinv(\vp_{\tau_i}).
\label{eq: from the definition of the boundary op}
\end{align}
Recall the decomposition
\[
 \del{\left([0, R(\si)]^{\si}\right)}
 = 
 \extdel(\si, R(\si))
\cup
 \bigcup_{i=0}^{n+1}
 	\left( [0, R(\tau_i)]^{\tau_i}
		\cup \bigsqcup_{R(\tau_i) \leq R \leq R(\si)} \extdel(\tau_i, R)
	\right).
\]
By \cref{prop: vanishing on certain face},
the parameterized moduli space $\calM(\vp_{\si})$ vanishes on $\del{\left([0, R(\si)]^{\si}\right)}$ except for the components $[0, R(\tau_i)]^{\tau_i}$ for $i = 0, \ldots, n+1$.
The moduli space also vanishes on $\del\left([0, R(\tau_i)]^{\tau_i}\right)$ by \cref{lem: vanishing on boundary}.
Hence the count of the point of the moduli space on $\del{\left([0, R(\si)]^{\si}\right)}$ coincides with that on $\bigsqcup_{i=0}^{n+1} (0, R(\si_{\tau_{i}}))^{\tau_{i}}$ for the perturbations $\vp_{\tau_{i}}$.
By taking the orientations into account, we therefore have
\begin{align}
\sum_{i = 0}^{n+1} (-1)^i \SWinv(\vp_{\tau_i}) = \#\calM(\vp_{\si}|_{\del{\left([0, R(\si)]^{\si}\right)}}).
\label{eq: counting on boundary}
\end{align}
Since the total moduli space $\calM(\vp_{\si})$ is an oriented compact $1$-dimensional manifold, the right-hand side of \eqref{eq: counting on boundary} is zero.
From this and \eqref{eq: from the definition of the boundary op}, we have $\SWcoch(X, \fraks, \calA)(\del \lb \si \rb) = 0$.
This proves the \lcnamecref{prop: coclosedness}.
\end{proof}

\begin{proof}[Proof of \cref{theo: well-definedness of the invariant}]
Take two additional data $\calA_{i} = (g_{i}, N_{i}(\cdot), a_{i}(\cdot), \vp_{\bullet, i})$ $(i = 0, 1)$;
let us write $\SWcoch(X, \fraks, \calA_{i})$ simply $\SWcoch[i]$.
We now connect these two data $\calA_{0}$ and $\calA_{1}$ by a $1$-parameter family.
Take a generic path
\begin{align}
\label{eq: path of metrics for cobordism argument}
\{g_{t}\}_{t \in [0,1]}
\end{align}
from $g_{0}$ and $g_{1}$ in $\Met(X) \cap \circPi$.
For each $\Si \in V(\calK)$, by the uniqueness of tubular neighborhoods, we may take a continuous family of tubular neighborhoods $N_{t}(\Si)$ of $\Si$.
Then we have a family of diffeomorphisms $f_{\Si, t} : N_{t}(\Si) \to D^2 \times \Si$  with $f_{\Si, t}|_{\Si} = \iota_{\Si}$; we also obtain a family $\tilde{f}_{\Si, t} : N_{t}(\Si) \setminus \Si \to (0, 1] \times S^1 \times \Si$ by the same way to define $\tilde{f}_{\Si}$ in \cref{subsection: Construction of a family of Riemannian metrics}.
We may also take a family of positive numbers $a_{t}(\cdot)$ satisfying the following conditions:
$a_{t}(\Si) = a_{t}(\{\Si\}) = 1$ holds for each $\Si \in V(\calK)$ and each $t \in [0, 1]$, and for each $\si \in S(\calK)$, the condition obtained by replacing $a(\cdot)$, $U(\cdot)$, $N(\cdot)$, and $\tilde{f}_{\Si}$ with $a_{t}(\cdot)$, $U_{t}(\cdot)$, $N_{t}(\cdot)$, and $\tilde{f}_{\Si, t}$ in \cref{cond: condition for a()} respectively holds.
We next take a path $\vp_{\bullet, t}$ as follows.
By repeating the construction of $\phi_{\bullet}$ in \cref{subsection: Construction of a family of Riemannian metrics} with the parameter $t$, we obtain $\phi_{\bullet, t}$.
Set $R_{t}(\si) = R(\si, g_{t}, N_{t}(\cdot), a_{t}(\cdot))$.
Note that $R_{t}(\si)$ continuously depends on $t$.
Then we may prove the statement obtained by replacing $R(\si)$ and $\phi_{\si}$ with $R_{t}(\si)$ and $\phi_{\si, t}$ in \cref{prop: vanishing on certain face} respectively by the quite same way in \cref{Vanishing property of the \SW moduli space on a family}.
We now fix a map
\[
\vp_{\si, \bullet} = \bigsqcup_{t \in [0,1]} \vp_{\si, t} : \bigsqcup_{t \in [0,1]} [0, R_{t}(\si)]^{\si} \to \circPi
\]
for each $\si \in S_{k}(\calK)$ with $k \in \{-1, 0, \ldots, n\}$ through the following inductive procedure.
For $i=0,1$, denote by $\ast_{i}$ the image of $\vp_{\emptyset, i}$.
Take a path $\ast_{t}$ in $\circPi$ from $\ast_{0}$ to $\ast_{1}$.
For each $\{\Si\} \in S_{0}(\calK)$, recall that $\vp_{\Si, i} : [0, R_{i}(\Si)] \to \circPi$ is given as a path from $\ast_{i}$ to $\phi_{\Si, i}(|\{\Si\}| \times \{R_{i}(\Si)\})$.
Let us take a map $\vp_{\Si, \bullet} : \bigsqcup_{t \in [0,1]} [0, R_{t}(\Si)] \to \circPi$ satisfying the following three conditions:
\begin{enumerate}
\item On $\bigsqcup_{t \in [0,1]} \{0\} \cong [0,1]$, the map $\vp_{\Si, \bullet}$ coincides with
\[
g_{\bullet} : [0,1] \to \Met(X) \cap \circPi\ ;\ t \mapsto g_{t}
\]
taken in \eqref{eq: path of metrics for cobordism argument}.
\item On $\bigsqcup_{t \in [0,1]} \{R_{t}(\Si)\} \cong [0,1]$, the map $\vp_{\Si, \bullet}$ coincides with
\[
\phi_{\Si, \bullet} : [0, 1] \to \Met(X) \cap \circPi\ ;\ t \mapsto \phi_{\Si, t}(|\{\Si\}| \times \{R_{t}(\Si)\}).
\]
\item For each $i=0,1$, on $[0, R_{i}(\Si)]$, the map $\vp_{\Si, \bullet}$ coincides with $\vp_{\Si, i}$.
\end{enumerate}
Next we consider $k$-simplices with $k \in \{1\, \ldots, n\}$.
Assume that $\vp_{\tau, \bullet}$ is already constructed for any $l \leq k-1$ and $\tau \in S_l(\calK)$, and that for each $t \in [0,1]$ the restriction $\vp_{\tau, t}|_{\extdel(\tau, R_{t}(\tau))}$ coincides with the composition
\begin{align*}
\phi_{\tau, t} \circ \psi_{\tau, R_{t}(\tau)}^{-1} : \extdel(\tau, R_{t}(\tau)) \cong \frakextdel(\tau, R_{t}(\tau)) \to \circPi.
\end{align*}
(We note that the homeomorphism $\psi_{\tau, R}$ for $R \geq 0$ given in \cref{prop: structure of the subspace of the domain of phi} is independent of any additional data.)
Let us take $\si \in S_{k}(\calK)$.
By replacing $R(\cdot)$ with $R_{t}(\cdot)$ in \eqref{eq: decomposition of the cube to interior and boundary} and taking the disjoint union with respect to $t \in [0,1]$, we have a decomposition of $\bigsqcup_{t \in [0,1]} [0, R_{t}(\si)]^{\si}$.
We take $\vp_{\si, \bullet}$ satisfying the two conditions:
the first one is the condition obtained by replacing $R(\cdot)$, $\vp_{\si}$, and $\phi_{\si}$ in \eqref{item: extboundary} and \eqref{item: interior boundary} in \cref{cond: restriction of perturbations} with $R_{t}(\cdot)$, $\vp_{\si, t}$, and $\phi_{\si,t}$ respectively for any $t \in [0,1]$, and the second is that obtained by replacing $(0,R(\si))^{\si}$ and $\vp_{\si}$ in \eqref{item: genericity in interior} in \cref{cond: restriction of perturbations} with $\bigsqcup_{t \in [0,1]} (0,R_{t}(\si))^{\si}$ and $\vp_{\si, \bullet}$ respectively.
This procedure can be continued to $k = n$ since we have assumed that $b^{+} > n+2$.

Using the $1$-parameter family connecting $\calA_{0}$ and $\calA_{1}$ obtained above, we define
\[
\SWcoch[{{0,1}}] \in C^{n-1}(\calK)
\]
by
\[
C_{n-1}(\calK) \to \Z\ ;\ \lb \tau \rb \mapsto (-1)^{n-1} \SWinv(\vp_{\tau, \bullet}),
\]
where $\SWinv(\vp_{\tau, \bullet}) = \#\calM(\vp_{\tau, \bullet})$.
Here the orientation of $\calM(\vp_{\tau, \bullet})$ is defined by the product orientation of the parameter space via the identification $\bigsqcup_{t \in [0,1]} [0,R_{t}(\tau)]^{\tau} \cong [0,R_{0}(\tau)]^{\tau} \times [0,1]$.

Now let us take an oriented $n$-simplex $\lb \si \rb = \lb \Si_0 , \ldots, \Si_{n} \rb$ and
write $\lb \tau_j \rb = \lb \Si_{0}, \ldots, \hat{\Si} ,\ldots,\Si_{n} \rb$.
Then we have
\begin{align*}
&(\SWcoch[1] - \SWcoch[0] -\codel \SWcoch[{{0,1}}])( \lb \si \rb )\\
= & \SWinv(\vp_{\si, 1}) - \SWinv(\vp_{\si, 0}) - (-1)^{n-1} \sum_{j=0}^{n} (-1)^{j} \SWinv(\vp_{\tau_{j}, \bullet})\\
= & \#\calM(\vp_{\si, \bullet}|_{\del( \bigsqcup_{t \in [0,1]} [0, R_{t}(\si)]^{\si})}) = 0,
\end{align*}
where the last equality follows from the fact that $\calM(\vp_{\si, \bullet})$ is an oriented compact $1$-dimensional manifold.
We therefore have $\SWcoch[1] - \SWcoch[0] = \codel \SWcoch[{{0,1}}]$.
This proves the \lcnamecref{theo: well-definedness of the invariant}.
\end{proof}

\begin{rem}
\label{rem: impossibility on isotopy classes}
For a closed $\spc$ $4$-manifold $(X, \fraks)$, in the same vein of the definition of the complex of curves due to Harvey~\cite{MR624817} in $2$-dimensional topology, one can define an abstract simplicial complex $\calK^{{\rm isot}} = \calK^{{\rm isot}}(X, \fraks)$: vertices are isotopy classes of embeddings of surfaces and  simplices are given as collections of such isotopy classes which can be represented by disjoint surfaces.
A natural question one can ask is whether we can construct a cohomology class which is similar to $\SWcoh(X, \fraks)$ on this simplicial complex $\calK^{{\rm isot}}$.
However, to construct such a cohomology class on $\calK^{{\rm isot}}$,  the way used to construct $\SWcoh(X, \fraks)$ on $\calK$ cannot work similarly.
An essential issue is that it is difficult to find  an effective analogy of the proof of \cref{theo: well-definedness of the invariant} for $\calK^{{\rm isot}}$.
Namely it seems hard to absorb new ambiguity arising from choices of representatives of isotopy classes into coboundaries.
The author does not know the answer to the question above at this stage.
\end{rem}

\subsection{The case when the moduli space is of higher-dimension}

\label{subsection: Definition of the invariant for general spinc structures}

In this subsection, using  what is called $\mu$-maps, we extend the definition of the cohomological \SW invariant also to the case when higher-dimensional moduli spaces appear.
Let $(X, \fraks)$ be a smooth oriented closed connected $\spc$ $4$-manifold equipped with a homology orientation.
We assume that $V(\calK(X, \fraks)) \neq \emptyset$.
Let $n \geq 0$ be a natural number such that 
$b^+(X) \geq n+3$ holds and $d(\fraks) + n +1$ is a non-negative even number.
For such a natural number, we define a cohomology class
\[
\SWcoh^{n}(X, \fraks) \in H^{n}(\calK(X, \fraks);\Z)
\]
in this subsection.
If there is no such a natural number $n$ for $\fraks$, we define the cohomological \SW invariant of $(X, \fraks)$ as $0 \in H^{\ast}(\calK(X, \fraks);\Z)$.
We denote by $m$ the non-negative number given by $2m = d(\fraks) + n +1$.
The setting in \cref{subsection: Definition of the invariant} is the case that $d(\fraks) + n + 1 = 0$ in this subsection;
in this case $\SWcoh^{n}(X, \fraks)$ coincides with $\SWcoh(X, \fraks)$.

For a metric $h \in \Met(X)$, let $\calB^{\ast}(h) = \calB^{\ast}(h, \fraks)$ be the subspace of $\calB(h)$ consisting of the image of irreducible configurations by the quotient map $\calC(h) \to \calB(h)$.
We obtain, as usual, a complex line bundle $\calL(h) = \calL(h, \fraks) \to \calB^{\ast}(h)$ equipped with a hermitian metric using the based gauge group.

Let us fix an additional data $\calA = (g, N(\cdot), a(\cdot),  \vp_{\bullet})$ as in \cref{subsection: Definition of the invariant}.
For each $\si \in S_{k}(\calK)$ with $k \in \{-1, 0, \ldots, n\}$, let $\vp_{\si}' : [0, R(\si)]^{\si} \to \Met(X)$ be the map given as the composition of $\vp_{\si} : [0, R(\si)]^{\si} \to \circPi$ and the projection $\circPi \to \Met(X)$.
We set
\[
\calB^{\ast}(\vp_{\si}) := \bigsqcup_{t \in [0, R(\si)]^{\si}}\calB^{\ast}(\vp_{\si}'(t))
\]
and define a complex line bundle $\calL(\vp_{\si}) \to \calB^{\ast}(\vp_{\si})$ by
\[
\calL(\vp_{\si}) := \bigsqcup_{t \in [0, R(\si)]^{\si}} \calL(\vp_{\si}'(t)).
\]
As in the construction of $\vp_{\bullet}$, we can take sections
\begin{align*}
\Set{ s_{\si,i} : \calB^{\ast}(\vp_{\si}) \to \calL(\vp_{\si}) }_{\si \in S_{k}(\calK), i \in \{1, \ldots, m\}}
\end{align*}
inductively on $k \in \{-1, 0, \ldots, n\}$ satisfying the following conditions:
\begin{itemize}
\item The restriction of $s_{\si,i}$ to $\calB^{\ast}(\vp_{\si})|_{(0, R(\si))^{\si}}$ is smooth for each $k \in \{-1, 0, \ldots, n\}$, $\si \in S_{k}(\calK)$, and $i \in \{1, \ldots, m\}$.
\item The intersection
\[
\calM(\vp_{\si}, s_{\si,1}, \ldots, s_{\si,m}) := \calM(\vp_{\si}) \cap s_{\si,1}^{-1}(0) \cap \cdots s_{\si,m}^{-1}(0)
\]
is a $0$-dimensional compact oriented smooth manifold for each $\si \in S_{n}(\calK)$.
\end{itemize}
The last condition is achieved by taking all sections $s_{\si,i}$ for $k \in \{-1, 0, \ldots, n\}$, $\si \in S_{k}(\calK)$, and $i \in \{1, \ldots, m\}$ to be transverse to the zero-section.
(Note that, for all $\si \in S_{n}(\calK)$, the ``cut parametrized moduli space'' $\calM(\vp_{\si}, s_{\si,1}, \ldots, s_{\si,m})$ is empty over $\del([0, R(\si)]^{\si})$ by reason of formal dimension.)
Set $\calA_{n} := (g, N(\cdot), a(\cdot),  \vp_{\bullet}, \{s_{\bullet, i}\}_{i})$, which we also call an {\it additional data}.
We define an $n$-cochain
\[
\SWcoch^{n}(X, \fraks, \calA_{n}) \in C^n(\calK)
\]
by
\[
C_n(\calK) \to \Z\ ;\ \lb \si \rb \mapsto \#\calM(\vp_{\si}, s_{\si,1}, \ldots, s_{\si,m}).
\]
We then obtain the following proposition by the same way to prove \cref{prop: coclosedness}:

\begin{prop}
\label{prop: coclosedness n}
The cochain $\SWcoch^{n}(X, \fraks, \calA_{n})$ constructed above is a cocycle.
\end{prop}

We now may write the definition of the generalized cohomological invariant:

\begin{defi}
\label{defi: main definition n}
Let $X$ be a smooth oriented closed connected $4$-manifold equipped with a homology orientation and $\fraks$ be a $\spc$ structure on $X$.
Assume that $V(\calK(X, \fraks)) \neq \emptyset$ and that there exists a natural number $n \geq 0$ such that 
$b^+(X) \geq n+3$ holds and $d(\fraks) + n +1$ is a non-negative even number.
We define a cohomology class $\SWcoh^{n}(X, \fraks) \in H^n(\calK(X, \fraks) ;\Z)$ by
\[
\SWcoh^{n}(X, \fraks) := [\SWcoch^{n}(X, \fraks, \calA_{n})],
\]
where $\calA_{n} = (g, N(\cdot), a(\cdot),  \vp_{\bullet}, \{s_{\bullet, i}\}_{i})$ is an additional data.
\end{defi}

Note that $\SWcoh^{n}(X, \fraks) = \SWcoh(X, \fraks)$ holds in the case that $d(\fraks) + n + 1 = 0$ by definition.
Since two sections $s_{\si,i}$ and $s_{\si,i}'$ can be connected by a path of sections, two additional data $\calA_{n}$ and $\calA_{n}'$ can be also connected by a path.
So, as in \cref{theo: well-definedness of the invariant}, we have the following theorem by a cobordism argument:

\begin{theo}
\label{theo: well-definedness of the invariant n}
The cohomology class $\SWcoh^{n}(X, \fraks) \in H^n(\calK(X, \fraks) ;\Z)$ is an invariant of a $spin^{c}$ $4$-manifold $(X, \fraks)$.
Namely the cohomology class $[\SWcoch^{n}(X, \fraks, \calA_{n})]$ in \cref{defi: main definition n} is independent of the choice of $\calA_{n}$. 
\end{theo}

\begin{rem}
Although we will give a $\spc$ $4$-manifold $(X, \fraks)$ such that $\SWcoh(X, \fraks) \neq 0$ in \cref{section: Non-vanishing results for the cohomological invariant}, it seems hard to give a non-trivial example for $\SWcoh^{n}(X, \fraks)$ with $d(\fraks) + n +1 \neq 0$ using the proof of the non-vanishing result for $\SWcoh(X, \fraks)$.
Indeed, we will construct a $\spc$ $4$-manifold $(X, \fraks)$ with $\SWcoh(X, \fraks) \neq 0$ as the connected sum of two $\spc$ $4$-manifolds $(M, \fraks_{0})$ and $(N, \frakt)$, where the \SW invariant of $(M, \fraks_{0})$ does not vanish.
If one tries to give an example of $(X, \fraks)$ with $\SWcoh^{n}(X, \fraks) \neq 0$ and $d(\fraks) + n +1 \neq 0$ via this way, one have to find $(M, \fraks_{0})$ such that the \SW invariant of it does not vanish and $d(\fraks_{0}) > 0$.
The simple type conjecture due to Witten~\cite{MR1306021} tells us that it is difficult (or might be impossible).
The author does not know whether there exists a non-trivial example for $\SWcoh^{n}(X, \fraks)$ with $d(\fraks) + n +1 \neq 0$ at this stage.
\end{rem}

\section{Non-vanishing results for the cohomological invariant}

\label{section: Non-vanishing results for the cohomological invariant}

In this section we prove that the invariant $\SWcoh(X, \fraks)$ constructed in \cref{subsection: Definition of the invariant} is non-trivial for some $\spc$ $4$-manifolds.
We also explain the non-vanishing result from a classical point of view: we will derive some constraints connected with the adjunction inequality on configurations of embedded surfaces from the non-vanishing result.
An interesting point of our cohomological formulation is that {\it one} non-vanishing result provides certain constraints on {\it infinitely many} configurations.
We use the following two key tools to prove the non-vanishing result:
the first one is a combination of wall-crossing and gluing technique originally due to Ruberman~\cite{MR1671187, MR1734421, MR1874146}, and the second is a description of higher-dimensional wall-crossing phenomena in terms of embedded surfaces given in \cite{Konno} by the author.
These tools are explained in \cref{subsection: Ruberman's combination of wall-crossing and gluing arguments} and \cref{subsection: description of higher-dimensional wall-crossing phenomena in terms of embedded surfaces} respectively.

\subsection{Combination of wall-crossing and gluing arguments}
\label{subsection: Ruberman's combination of wall-crossing and gluing arguments}

Ruberman has discussed a combination of wall-crossing and gluing arguments in~\cite{MR1671187, MR1734421, MR1874146}.
Ruberman considered wall-crossing over a $4$-manifold with $b^+=1$, and
Baraglia and the author generalized that argument for all $b^+ \geq 1$ in \cite{BK}.
Here we shall use the result in \cite{BK}.

Let $(N, \frakt)$ be a smooth oriented closed $\spc$ $4$-manifold equipped with a homology orientation.
Set $k := b^+(N)$ and 
assume that $k \geq 1$.
Let $\vp^{N} : D^k \to \Pi(N, \frakt)$ be a generic map satisfying $\vp^{N}(S^{k-1}) \subset \circPi(N, \frakt)$ from the $k$-dimensional disk to the space of perturbations.
Here ``generic" means that $\vp^{N}$ is transverse to the wall $\calW(N, \frakt)$, defined in \eqref{eq def of wall}.
Then we can define the ``intersection number" $\vp^{N} \cdot \calW(N, \frakt) \in \Z$ of $\vp^{N}$ and $\calW(N, \frakt)$.
This intersection number can be interpreted as the mapping degree of the map $\vp^{N}|_{S^{k-1}} : S^{k-1} \to \circPi(N, \frakt) \simeq S^{k-1}$. 
Here the given orientation of $H^+(N; \R)$ is used to determine the sign of the mapping degree.
The precise convention of the sign is given in Ruberman~\cite{MR1874146} (in the case that $b^+=1$) and we do not repeat it here.
(In fact, for our non-vanishing result in \cref{subsection: Non-vanishing theorem} we do not have to determine the sign.)

Let $(M, \fraks_0)$ be a smooth oriented closed $\spc$ $4$-manifold equipped with a homology orientation.
Suppose that $b^+(M) \geq 2$ and $d(\fraks_0) = 0$.
For natural numbers $k \geq 1$ and $l \geq k$, we consider the $4$-manifold $N$ defined by 
\[
N = k\CP^2 \# l(-\CP^2) = \#_{i=1}^k \CP^2_{i} \# (\#_{j=1}^l (-\CP^2)).
\]
Here $\CP^2_{i}$ and $-\CP^2$ are copies of $\CP^{2}$ and $-\CP^{2}$ respectively, where $-\CP^{2}$ is the projective plane equipped with the opposite orientation. 
Let $H_{i}$ and $E$ be a generator of $H^2(\CP^2_{i})$ and one of $H^2(-\CP^2)$ respectively.
The basis $H_1, \ldots, H_k$ of $H^{+}(N;\R)$ gives a homology orientation of $N$.
Let $\frakt$ be the $\spc$ structure on $N$ determined by 
\begin{align}
c_1(\frakt) = \sum_{i=1}^k H_{i} + \sum_{j=1}^l E.
\label{eq: first Chern class defining property}
\end{align}
We set 
\[
(X, \fraks) := (M \# N, \fraks_0\#\frakt).
\]
Then we have $d(\fraks) = -k$.
Let $(g^{M}, \mu) \in \Pi(M, \fraks_0)$ be a generic point and $\vp^{N} : D^k \to \Pi(N, \frakt)$ be a generic map satisfying $\vp^{N}(S^{k-1}) \subset \circPi(N, \frakt)$.
Then we can consider the intersection number $\vp^{N} \cdot \calW(N, \frakt)$ as in the previous paragraph.
Let $B_{M}^{4} \subset M$ and $B_{N}^{4} \subset N$ be small balls.
Assume that the metric $g^{M}$ and the metric-component of each point in the image of $\vp^{N}$ are cylindrical near the boundaries of $B_{M}^{4}$ and $B_{N}^{4}$ respectively, and the boundaries of $B_{M}^{4}$ and $B_{N}^{4}$ are mutually isometric.
We also assume that $\mu$ and the self-dual $2$-form-component of each point in the image of $\vp^{N}$ are supported on the complement of $B_{M}^{4}$ and $B_{N}^{4}$ respectively.
Then we may define $\vp : D^{k} \to \Pi(X, \fraks)$ by $\vp(x) := (g^{M}, \mu)\#\vp^{N}(x)$, where $\#$ is the connected sum consisting of $M \setminus B^{4}_{M}$, $N \setminus B^{4}_{N}$, and a cylinder with the standard product metric and of sufficiently large length.
Since $\vp$ is also a generic map and we have $d(\fraks) = -k$, the moduli space $\calM(\vp, \fraks)$ parameterized by $\vp$ on $D^k$ is a $0$-dimensional compact manifold.

The following proposition is discussed by Ruberman in the case that $b^+(N)=1$ and $\vp^{N} \cdot \calW(N, \frakt) = \pm1$ in Section~4 of \cite{MR1874146}.
A similar argument, spelled out by Baraglia and the author in \cite{BK}, works for general $b^+(N) \geq 1$ and $\vp^{N} \cdot \calW(N, \frakt)$.

\begin{prop}[\cite{MR1874146}, \cite{BK}]
\label{prop : wall-crossing+gluing}
Suppose that the length of the cylinder used to define the connected sum is sufficiently large.
Then
\begin{align}
\# \calM(\vp, \fraks) = \pm(\vp^{N} \cdot \calW(N, \frakt)) \cdot \SWinv(M, \fraks_0)
\label{eq : wall-crossing+gluing}
\end{align}
holds in $\Z$.
Here $\SWinv(M, \fraks_0)$ is the \SW invariant of $(M, \fraks_0)$.
\end{prop}

\begin{proof}
We may adapt Theorem~4.4 in \cite{BK} after a small modification.
To explain it, first we summarize a part of the statement of the theorem.
The setting needed to describe the part of the statement of the theorem is as follows.
Let $(M, \fraks_0)$ $(N, \frakt),$ and $(X, \fraks)$ be as above.
Let $B$ be an oriented closed manifold of dimension $k$.
Let $M \to E_{M} \to B$ and $N \to E_{N} \to B$ smooth fiber bundles equipped with fiberwise $\spc$ structures  whose restriction on fibers are isomorphic to $\fraks_{0}$ and $\frakt$ respectively.
Assume that $E_{M}$ and $E_{N}$ admit sections, and let us construct the fiberwise connected sum $X \to E_{X} \to B$ along these sections.
This bundle $E_{X}$ is equipped with the fiberwise $\spc$ structure obtained from that of $E_{M}, E_{N}$, and the restriction this fiberwise $\spc$ structure on a fiber is isomorphic to $\fraks$.
Let us fix a homology orientation of $X$.
Since we have $d(\fraks)=-\dim{B}$ and $b^{+}(X) \geq \dim{B} + 2$, we can define the families \SW invariant $FSW(E_{X},\fraks) \in \Z$ by counting the parametrized moduli space for $E_{X}$ over $B$.

Theorem~4.4 in \cite{BK} gives a way to compute $FSW(E_{X},\fraks)$ as follows.
Fix a fiberwise metric of $E_{N}$.
Then we obtain a bundle of harmonic self-dual $2$-forms, denoted by $H^{+}(N) \to H^{+}(E_{N}) \to B$.
Since the space of fiberwise metrics is contractible, the isomorphism class of $H^{+}(E_{N})$ is independent of the choice of fiberwise metric.
Theorem~4.4 asserts that we have
\begin{align*}
FSW(E_{X},\fraks) = \SWinv(M,\fraks_{0}) \cdot \left<e(H^{+}(E_{N})), [B]\right>.
\end{align*}
(Note that $e(H^{+}(E_{N}))$ is denoted by $e(H^{+}(N))$ in \cite{BK}.)
We shall explain how to modify the above gluing result to show \eqref{eq : wall-crossing+gluing}.

First, note that, although the parameter space of the families is a closed manifold $B$ in Section~4 in \cite{BK}, the proof of the gluing formula localizes only around the points in $B$ on which wall-crossing for $N$ happens.
This localization, given as Corollary~7.9 in \cite{BK}, can be understood easily as follows.
Let $\{\eta_{N}(b)\}_{b \in B}$ be a families perturbation consisting of a fiberwise imaginary self-dual $2$-form over $E_{N}$.
(This corresponds to $\varphi^{N}$ we discussed, but for a while we fit our notation to that of \cite{BK} for explanation.)
Since we assumed that $d(\frakt)<0$, if the families perturbation $\{\eta_{N}(b)\}_{b \in B}$ was taken to be generic, the moduli space on the fiber $(E_{N})_{b}$ over $b$ with respect to $\eta_{N}(b)$ is empty unless the \SW equations on $(E_{N})_{b}$ have a reducible solution for $\eta_{N}(b)$.
The gluing argument for the moduli space on $M$, which is possibly non-empty,  and the empty moduli space on $N$ yields just the empty moduli space on $X$.
Therefore non-trivial gluing occurs only around the points $b$'s in $B$ for which wall-crossing for $(E_{N})_{b}$ happens for $\eta_{N}(b)$.

Because of this localization, we may use $B=D^{k}$ as the parameter space of families provided that wall-crossing does not happen over $\del B$.
Let us describe it more precisely.
Assume that $E_{M}$ and $E_{N}$ are trivialized over $B=D^{k}$: $E_{M} = B \times M, E_{N} = B \times N$.
Let $\eta_{M}=\{\eta_{M}(b)\}_{b \in B}$ and $\eta_{N}=\{\eta_{N}(b)\}_{b \in B}$ be generic families perturbations over $E_{M}$ and $E_{N}$ respectively.
Assume that the \SW equations on $(N, \eta_{N}(b))$ do not have reducible solutions for all $b \in \del B$.
Then we obtain the following gluing formula:
\begin{align}
\label{eq: gluing formula in BK2}
FSW(E_{X},\fraks,\eta_{M},\eta_{N}) = \left<e(H^{+}(E_{N}), \eta_{N}), [B, \del B]\right> \cdot \SWinv(M,\fraks_{0}).
\end{align}
Here let us explain notations appearing in \eqref{eq: gluing formula in BK2}.
First, $FSW(E_{X},\fraks,\eta_{M},\eta_{N}) \in \Z$ denotes the number obtained by counting the parameterized moduli space for $E_{X}$ with respect to the fiberwise metric on $E_{X}$ obtained from $\eta_{M}$ and $\eta_{N}$.
(The precise construction of the fiberwise metric on $E_{X}$ is described in Section~7 in \cite{BK}.)
Second, $e(H^{+}(E_{N}), \eta_{N}) \in H^{k}(B, \del B ;\Z)$ is the relative Euler class with respect to $\eta_{N}$ defined as follows.
For each $b \in B$, denote by $g_{N}(b)$ the underlying metric of $\eta_{N}(b)$.
Set $g_{N} = \{g_{N}(b)\}_{b \in B}$.
For a metric $h$ on $N$, let $H^{+}(N,h)$ denote the space of imaginary self-dual harmonic $2$-forms of $N$ with respect to the metric $h$.
Then we can define a vector bundle $H^{+}(E_{N}, g_{N}) \to B$ with fiber $H^{+}(N,g_{N}(b))$.
Regard $\eta_{N}$ as a section $\eta_{N} : B \to \Pi(N) \times B$ of the trivial bundle $\Pi(N) \times B\to B$. 
Let
\[
p : \Pi(N) \to \bigsqcup_{h \in \Met(N)} H^{+}(N,h)
\]
be the fiberwise $L^{2}$-projection.
Let $s : B \to H^{+}(E_{N}, g_{N})$ be the section defined by
\[
s(b) = p(F_{A_{N}}^{+_{g_{N}(b)}}-\eta_{N}(b)),
\]
where $A_{N}$ is a fixed smooth reference connection on $N$.
One may see that $s(b) \neq 0$ for all $b \in \del B$ because of the assumption that the \SW equations on $(N, \eta_{N}(b))$ do not have reducible solutions for all $b \in \del B$.
Therefore we may define the relative Euler class $e(H^{+}(E_{N}), \eta_{N}) \in H^{k}(B, \del B ;\Z)$ with respect to $\eta_{N}$ as the pull-back of the Thom class of the bundle $H^{+}(E_{N}) \to B$ by $s$.

As in the proof of Proposition~7.2 of \cite{BK}, we may take $\eta_{N}$ so that $\eta_{N}$ meets the wall $\calW(N, \frakt)$ transversely in a neighborhood of the zero points of the section $s : B \to H^{+}(E_{N}, g_{N})$, and that $\eta_{N}$ does not meet the wall outside the neighborhood.
More precisely, $\eta_{N}$ can be taken to meet the wall once around each zero point $b$ of $s$, and that the sign of the intersection between $s$ and the zero section of $H^{+}(N)$ at $b$ corresponds to that of between $\eta_{N}$ and the wall around $b$.
Thus we have that $\left<e(H^{+}(E_{N}), \eta_{N}), [B, \del B]\right> =  \pm \eta_{N} \cdot \calW(N, \frakt)$.
Coming back to our current notation of this paper, $\eta_{N}$ denotes $\varphi^{N}$.
This and \eqref{eq: gluing formula in BK2} complete the proof.
\end{proof}

\subsection{Higher-dimensional wall-crossing and embedded surfaces}
\label{subsection: description of higher-dimensional wall-crossing phenomena in terms of embedded surfaces}

In~\cite{Konno} the author has given a description of higher-dimensional wall-crossing phenomena (i.e. wall-crossing for general $b^{+} \geq 1$)  in terms of embedded surfaces.
In this subsection we recall and rewrite a part of it as a convenient form for the setting of this paper.
We also introduce some concepts, which will be used to state the non-vanishing result in \cref{subsection: Non-vanishing theorem}.
For natural numbers $k \geq 1$ and $l \geq k$, set $N = k\CP^2 \# l(-\CP^2) = \#_{i=1}^k \CP^2_{i} \# (\#_{j=1}^l (-\CP^2))$ and let $\frakt$ be the $\spc$ structure on $N$ given in \cref{subsection: Ruberman's combination of wall-crossing and gluing arguments}.

Let $\Si_i^{+}, \Si_i^{-} \in V(\calK(N, \frakt))$ $(i = 1, \ldots, k)$ be vertices such that 
\begin{align}
i \neq i' \Rightarrow
\Si_i^+ \cap \Si_{i'}^+ = \Si_i^+ \cap \Si_{i'}^- = \Si_i^- \cap \Si_{i' }^- = \emptyset
\label{eq: disjoint condition}
\end{align}
holds.
We write the collection of $2k$-surfaces $\Si_i^{+}, \Si_i^{-}$ $(1 \leq i \leq k)$ as
\[
\bSi = \{\Si_i^{\epsilon}\}_{1 \leq i \leq k, \epsilon \in \{+,-\}}.
\]
Set
\[
\Signk := \{+, -\}^{\{1, \ldots, k\}}.
\]
For each $\bep = \{ \epsilon_{i} \}_{i=1}^{k} \in \Signk$,
let $\si_{\bep} \in S(\calK(N, \frakt))$ be the $(k-1)$-simplex defined by
\[
\si_{\bep} := \Set{ \Si^{\epsilon_{i}}_{i} | 1 \leq i \leq k}.
\]
We may also define a subcomplex $K(\bSi)$ of $\calK(N, \frakt)$ by giving the set of simplices as
\[
S(K(\bSi)) := \Set{ \Set{ \Si^{\epsilon_{i}}_{i} | i \in J } | \{\epsilon_{i}\}_{i \in J} \in \{+, -\}^{J},\ J \subset \{1, \ldots, k\}}.
\]
The set $\Set{ \si_{\bep} |  \bep \in \Signk }$ can be regarded as the set of $(k-1)$-simplices of $K(\bSi)$.
Note that the geometric realization $|K(\bSi)|$ is homeomorphic to $S^{k-1}$.

\begin{rem}
\label{rem: regarding a subcomplex on connected sum}
Let $(M, \fraks_{0})$ be an oriented closed $\spc$ $4$-manifold.
Let us consider the $\spc$ $4$-manifold $(X, \fraks) := (M \# N, \fraks_0\#\frakt)$.
Here the connected sum $M \# N$ is defined by gluing a ``punctured" $M$ and $N$ along a puncture belonging to the complement of all $\Si^{\epsilon}_{i}$'s in $N$.
Then each $\Si^{\epsilon}_{i}$ can be regarded as an embedded surface in $X$ and it also violates the adjunction inequality with respect to $\fraks$.
Therefore $K(\bSi)$ can be regarded as a subcomplex of $\calK(X, \fraks)$.
We will fix an orientation of each $\si_{\bep}$ as below.
Then we can consider the fundamental class of $K(\bSi)$, and thus obtain the homology class $[K(\bSi)] \in H_{k-1}(\calK(X, \fraks);\Z)$.
\end{rem}

Fix data $g^{N}$, $N(\cdot)$, and $a(\cdot)$ on the $4$-manifold $N$ as  in \cref{section: Family of Riemannian metrics}.
Then we can define $R^{N}(\si_{\bep}) = R(\si_{\bep}, g^{N}, N(\cdot), a(\cdot))$ as in \cref{Vanishing property of the \SW moduli space on a family}.
Set
\[
R^{N}_{\bSi} := \max_{\bep \in \Signk} R^{N}(\si_{\bep}).
\]
For a collection of positive numbers $\{R_{\tau}\}_{\tau \prec \si_{\bep}, \bep \in \Signk}$, we may take maps
\begin{align}
\{\vp^{N}_{\tau, R_{\tau}} : [0, R_{\tau}]^{\tau} \to \Pi(N, \frakt)\}_{\tau \prec \si_{\bep}, \bep \in \Signk}
\label{eq: maps using collection of pos numbers}
\end{align}
inductively on the dimension of $\tau$ by applying the same procedure in \cref{subsection: Definition of the invariant}.
Here we should note that the image of $\vp^{N}_{\si_{\bep}, R_{\si_{\bep}}}$ is not contained in $\circPi(N, \frakt)$ in general even if we take $\vp^{N}_{\si_{\bep}, R_{\si_{\bep}}}$ to be generic; $\vp^{N}_{\si_{\bep}, R_{\si_{\bep}}}$ generically intersects the wall at points.
However, the partial image $\vp^{N}_{\si_{\bep}, R_{\si_{\bep}}}(\extdel(\si_{\bep}, R_{\si_{\bep}}))$ satisfies that
\begin{align}
\vp^{N}_{\si_{\bep}, R_{\si_{\bep}}}(\extdel(\si_{\bep}, R_{\si_{\bep}})) \subset \Met(N) \cap \circPi(N, \frakt)
\label{eq: subset relation for partial image}
\end{align}
if $R_{\si_{\bep}} \geq R^{N}_{\bSi}$ as in \cref{subsection: Definition of the invariant}.
This is  based on the vanishing result, \cref{prop: vanishing on certain face}, and therefore we have a stronger property:
there is no solution to the unperturbed \SW equations for any metric in $\vp^{N}_{\si_{\bep}, R_{\si_{\bep}}}(\extdel(\si_{\bep}, R_{\si_{\bep}}))$.
Furthermore, we also have
\begin{align}
\vp^{N}_{\si_{\bep}, R_{\si_{\bep}}}(\del\left([0, R_{\si_{\bep}}]^{\si_{\bep}}\right)) \subset \circPi(N, \frakt)
\label{eq: subset relation for partial partial image}
\end{align}
for any positive number $R_{\si_{\bep}}$ with $R_{\si_{\bep}} \geq R^{N}_{\bSi}$.
Indeed, there is no solution to the \SW equations for any element in the left-hand side of \eqref{eq: subset relation for partial partial image}.
It follows from the same argument used to prove \cref{lem: vanishing on boundary}.
We therefore obtain the relation \eqref{eq: subset relation for partial partial image} from consideration of reducible solutions.

For $R>0$ we define a space $D^{k}_{\bSi, R}$ as the quotient
\begin{align}
D^{k}_{\bSi, R} := \left( \bigsqcup_{\bep \in \Signk} [0, R]^{\si_{\bep}} \right) / \sim,
\label{eq: disk obtained as a quotient}
\end{align}
where $\sim$ is the equivalence relation generated by the following identification:
for $\bep, \bep' \in \Signk$ and for a common face $\tau$ of $\si_{\bep}$ and $\si_{\bep'}$,
we identify the subset $[0, R]^{\tau} \subset [0, R]^{\si_{\bep}}$ and the subset $[0, R]^{\tau} \subset [0, R]^{\si_{\bep'}}$ by the identity map $[0, R]^{\tau} \to [0, R]^{\tau}$.

\begin{defi}
\label{defi: orientation of bSi}
We call a collection of orientations of $\si_{\bep}$'s an {\it orientation of} $\bSi$ if the orientations of $[0, R]^{\si_{\bep}}$'s  defined by that of $\si_{\bep}$'s induce an orientation of $D^{k}_{\bSi, R}$ for $R > 0$.
\end{defi}

Note that an orientation of $\bSi$ induces an orientation of $|K(\bSi)| \cong S^{k-1}$.
Let $\{R_{\tau}\}_{\tau \prec \si_{\bep}, \bep \in \Signk}$ be a collection of positive numbers and $R$ be a positive number such that $R = R_{\si_{\bep}}$ holds for any $\bep \in \Signk$.
Then we can take maps \eqref{eq: maps using collection of pos numbers}.
By the construction of $\vp_{\bullet}$, the collection of maps $\{\vp^{N}_{\si_{\bep}, R}\}_{\bep \in \Signk}$
gives rise to a map
\[
\vp^{N}_{\bSi, R} : D^{k}_{\bSi, R} \to \Pi(N, \frakt).
\]
The space $D^{k}_{\bSi, R}$ is homeomorphic to a $k$-dimensional disk.
The boundary $S^{k-1}_{\bSi, R} := \del D^{k}_{\bSi, R}$ can be expressed as
\[
S^{k-1}_{\bSi, R} = \left( \bigsqcup_{\bep \in \Signk} \extdel(\si_{\bep}, R) \right) / \sim.
\]
We note that $\vp^{N}_{\bSi, R}(S^{k-1}_{\bSi, R}) \subset \Met(N) \cap \circPi(N, \frakt)$ holds by \eqref{eq: subset relation for partial image} and moreover there is no solution to the unperturbed \SW equations for any metric in $\vp^{N}_{\bSi, R}(S^{k-1}_{\bSi, R})$.
We can therefore define the intersection number $\vp^{N}_{\bSi, R} \cdot \calW(N, \frakt)$ for this $\vp^{N}_{\bSi, R}$.

\begin{rem}
\label{rem: foundation of disjoint counting}
Let $\pi : \bigsqcup_{\bep \in \Signk} [0, R]^{\si_{\bep}} \to D^{k}_{\bSi, R}$ be the quotient map.
We note that the subspace
\begin{align}
\Set{ p \in \bigsqcup_{\bep \in \Signk} [0, R]^{\si_{\bep}} | \#\pi^{-1}(\pi(p))>1} 
\label{eq: identified area}
\end{align}
of $\bigsqcup_{\bep \in \Signk} [0, R]^{\si_{\bep}}$ is contained in
\begin{align}
\bigsqcup_{\bep \in \Signk} \del\left([0, R]^{\si_{\bep}}\right).
\label{eq: disjoint union of boundaries}
\end{align}
We also note that there is no solution to the \SW equations for any element in the image of \eqref{eq: disjoint union of boundaries} by $\vp^{N}_{\si_{\bep}, R}$'s if $R \geq R^{N}_{\bSi}$ as we mentioned to obtain \eqref{eq: subset relation for partial partial image}.
\end{rem}

We here introduce a term in order to describe wall-crossing phenomena in terms of embedded surfaces.
In general, for an oriented closed $4$-manifold $W$,
we often identify $H^{2}(W;\Z)$ with $H_{2}(W;\Z)$ by the Poincar\'{e} duality, and
we denote the evaluation $\left<a, b\right>$ for $a \in H^{2}(W;\Z), b \in H_{2}(W;\Z)$ by $a \cdot b$.

\begin{defi}
\label{defi: wall-crossing collection of surfaces}
For vertices $\Si_i^{+}, \Si_i^{-} \in V(\calK(N, \frakt))$ $(i = 1, \ldots, k)$ satisfying \eqref{eq: disjoint condition}, 
we call the collection $\bSi = \{\Si_i^{\epsilon}\}_{1 \leq i \leq k, \epsilon \in \{+,-\}}$ a {\it wall-crossing collection of surfaces in} $(N, \frakt)$  if the following two conditions are satisfied:
\begin{enumerate}
\item Let $i, i' \in \{1, \ldots, k\}$. 
If $i \neq i'$, then $H_{i'} \cdot [\Si_i^{\pm}] = 0$ holds.
\item For each $i \in \{1, \ldots, k\}$, the two integers
\[
(c_1(\frakt) \cdot  [\Si_i^{+}]) \cdot (H_i \cdot [\Si_i^{+}])
\]
and
\[
(c_1(\frakt) \cdot  [\Si_i^{-}]) \cdot (H_i \cdot [\Si_i^{-}])
\]
are non-zero and have different signs.
\end{enumerate}
Here $H_{i}$ is the one in \cref{subsection: Ruberman's combination of wall-crossing and gluing arguments}.
\end{defi}

Concrete examples of wall-crossing collections of surfaces will be given in \cref{subsection: Examples}.
The description of wall-crossing phenomena in term of embedded surfaces is given as the following \lcnamecref{prop: description of higher-dim. wall-crossing}.
An essential part of its proof has been given also in the proof of Lemma~3.3 of \cite{Konno}.

\begin{prop}[\cite{Konno}]
\label{prop: description of higher-dim. wall-crossing}
Let $\bSi = \{\Si_i^{\epsilon}\}_{1 \leq i \leq k, \epsilon \in \{+,-\}}$ be a wall-crossing collection of surfaces.
Assume that $\Si_i^+$ and $\Si_i^-$ intersect transversally for each $i$.
Let us take a collection of positive numbers $\{R_{\tau}\}_{\tau \prec \si_{\bep}, \bep \in \Signk}$ and a positive number $R$.
Suppose that $R = R_{\si_{\bep}}$ holds for any $\bep \in \Signk$ and that $R \geq R^{N}_{\bSi}$.
If we take $\{R_{\tau}\}_{\tau \prec \si_{\bep}, \bep \in \Signk}$ so that $R$ is large enough,
then $\vp^{N}_{\bSi, R} : D^{k}_{\bSi, R} \to \Pi(N, \frakt)$ satisfies that $\vp^{N}_{\bSi, R} \cdot \calW(N, \frakt) = \pm1$.
\end{prop}

\begin{proof}
Assuming that $\Si_i^+$ and $\Si_i^-$ intersect transversally for each $i$, we have the following alternative simple description of the family $\vp^{N}_{\bSi, R}$ by taking suitable additional data $g^{N}$, $N(\cdot)$, and $a(\cdot)$ in \cref{section: Family of Riemannian metrics}.
This construction is originally due to Fr{\o}yshov~\cite{MR2052970}.
The construction is based on $k$-parameter families of metrics obtained from simultaneous stretches of neighborhoods of disjoint $k$-surfaces.
Roughly, the desired whole family is given by patching all of the possible stretching families, corresponding to all combinations of disjoint $k$-surfaces belonging to $\bSi$.
Note that such a simultaneous stretching cannot be considered for surfaces intersecting each other, say $\Si_{i}^{+}$ and $\Si_{i}^{-}$.

To spell out this construction of Riemannian metrics precisely,
take a Riemannian metric $g^{N} \in \Met(N)$ such that a neighborhood of the sphere bundle of the normal bundle $\nu(\Si_i^{\pm})$ is isometric to $[0,1] \times S^{1} \times \Si_i^{\pm}$.
Here $[0,1] \times S^{1} \times \Si_i^{\pm}$ is equipped with the standard product metric as in \cref{subsection: Construction of a family of Riemannian metrics,Vanishing property of the \SW moduli space on a family}.
For $i \neq j$, we may assume that the cylindrical parts $[0,1] \times S^{1} \times \Si_i^{\epsilon_i}$ and $[0,1] \times S^{1} \times \Si_j^{\epsilon_j}$ are disjoint for any $\epsilon_i, \epsilon_j \in \{+, -\}$.
(This is achieved by taking suitable $N(\cdot)$ and $a(\cdot)$.)
We therefore obtain a family of Riemannian metrics 
\[
\bar{\phi}^{N}_{\bSi} : \R^k \to \Met(N)\ ;\ (r_1, \ldots, r_k) \mapsto G(r_1, \ldots, r_k),
\]
where $G(r_1, \ldots, r_k)$ is defined by
\begin{itemize}
\item replacing $[0,1] \times S^1 \times \Si_i^{+}$ with $[0,r_i] \times S^1 \times \Si_i^{+}$ from $g^{N}$ if $r_i \geq 0$, and
\item replacing $[0,1] \times S^1 \times \Si_i^{-}$ with $[0,-r_i] \times S^1 \times \Si_i^-$ from $g^{N}$ if $r_i \leq 0$
\end{itemize}
for each $i$.
For $R \geq R^{N}_{\bSi}$,
let $D^{k}(R)$ be the cuboid in $\R^{k}$ centered at the origin and with the length of each edge $2R$.
Then the restriction
\[
\phi^{N}_{\bSi} := \bar{\phi}^{N}_{\bSi}|_{D^{k}(R)} : D^{k}(R) \to \Met(N)
\]
coincides with the metric-component of $\vp^{N}_{\bSi, R}$ under the natural identification $D^{k}_{\bSi, R} \cong D^{k}(R)$.

Denote by $V^+$ the subspace of $H^2(N;\R)$ spanned by $H_1, \ldots, H_k$.
For a Riemannian metric $g$ on $N$, let $\scrH^{+_{g}}(N)$ be the space of self-dual harmonic 2-forms on $N$.
Define a linear isomorphism
\[
\varphi_{g} : \scrH^{+_{g}}(N) \to V^{+}
\]
as the composition of the following maps $i_{g}$ and $p_{V^{+}}$.
First, the map $i_{g}$ is an injective map
\[
i_{g} : \scrH^{+_{g}}(N) \to H^{2}(N;\R)
\]
defined as the restriction of the Hodge-isomorphism
\[
\scrH^{2}_{g}(N) \to H^{2}(N;\R)
\]
from the space of harmonic $2$-forms to the second cohomology.
Note that the image of $i_{g}$ varies with $g$: each metric gives a particular maximal-dimensional positive-definite subspace of $H^{2}(N;\R)$.
Second, $p_{V^{+}}$ is the projection
\[
p_{V^{+}} : H^{2}(N;\R) = V^{+} \oplus (V^{+})^{\perp} \to V^{+}.
\]
Here $(V^{+})^{\perp}$ is the orthogonal complement with respect to the intersection form.
Set $c=c_{1}(\frak{t})$ and define
\[
\Psi : \Met(N) \to \R^{k}
\]
by
\[
\Psi(g) := (c \cdot \varphi^{-1}_{g}(H_{i}))_{i=0}^{k}.
\]
Here we regard $\varphi^{-1}_{g}(H_{i})$ as an element of $H^{2}(X;\R)$ via $i_{g}$ and take the cup product with the cohomology class $c$.
For $R>0$, define
\[
\mathcal{F} = \mathcal{F}_{R} : D^{k}(R) \to \R^{k}
\]
by $\mathcal{F} := \Psi \circ \phi^{N}_{\bSi}$.

We claim that, if $0 \notin \mathcal{F}(\del D^{k}(R))$, the mapping degree of $\mathcal{F} : \del D^{k}(R) \to \R^{k} \setminus \{0\}$ coincides with that of 
$\phi^{N}_{\bSi}|_{\del D^{k}(R)} : \del D^{k}(R) \to \Met(N) \cap \circPi.$
(Note that the image of the restriction map $\vp^{N}_{\bSi, R}|_{S^{k-1}_{\bSi, R}}$ is contained in $\Met(N) \cap \circPi$ as we mentioned, and so is the image of $\phi^{N}_{\bSi}|_{\del D^{k}(R)}$.)
First, we see that the intersection points of the image of $\phi^{N}_{\bSi}$ with the wall $\calW(N, \frakt)$ correspond to the zero set of $\mathcal{F}$.
Recall that, for a metric $g$ on $N$, the condition that there exists a reducible solution to the unperturbed Seiberg--Witten equation for $g$ and $\frakt$ is equivalent to the condition that $c$ is orthogonal to $\scrH^{+_{g}}(N)$ with respect to the intersection form.
Since $\varphi_{g}$ is an isomorphism, the latter condition is equivalent to $\Psi(g)=0$.
This shows that $(\phi^{N}_{\bSi})^{-1}(\calW(N, \frakt))=\mathcal{F}^{-1}(0)$.
Next, to complete the proof of the coincidence of the degrees of $\mathcal{F}|_{\del D^{k}(R)}$ and $\phi^{N}_{\bSi}|_{\del D^{k}(R)}$,
we see that, generically, the signed count of each intersection point of the image of $\phi^{N}_{\bSi}$ with the wall $\calW(N, \frakt)$ coincides with the signed count of the corresponding intersection point of $\mathcal{F}$ with $0$.
Considering each intersection point separately,
it suffices to prove the correspondence in the case that the family of metrics $\phi^{N}_{\bSi}$ makes $\mathcal{F} : \del D^{k}(R) \to \R^{k} \setminus \{0\}$ a degree one map.
Let $p_{g} : \Omega^{+}_{g}(N) \to \scrH^{+_{g}}(N)$ be the $L^{2}$-projection for each metric $g$.
Since we have  $c \cdot \varphi^{-1}_{g}(H_{i}) = p_{g}\left(\frac{-1}{2\pi \sqrt{-1}}F_{A_{0}}^{+_{g}}\right) \wedge \varphi^{-1}_{g}(H_{i})$ and supposed that $\mathcal{F}|_{\del D^{k}(R)}$ is a degree one map, the family $\{p_{g}(F_{A_{0}}^{+_{g}})\}_{g \in \phi^{N}_{\bSi}(D^{k}(R))}$ generically intersects with $0$ once.
In other words, the family of metrics $\phi^{N}_{\bSi}(D^{k}(R))$ generically intersects with the wall $\calW(N, \frakt)$ once. 

The mapping degree of $\phi^{N}_{\bSi}|_{\del D^{k}(R)}$, which is that of
 $\vp^{N}_{\bSi, R}|_{S^{k-1}_{\bSi, R}} : S^{k-1}_{\bSi, R} \to \circPi(N, \frakt) \simeq S^{k-1}$, is the intersection number $\vp^{N}_{\bSi, R} \cdot \calW(N, \frakt)$.
By the last paragraph, we have that $\vp^{N}_{\bSi, R} \cdot \calW(N, \frakt)$ is the mapping degree of $\mathcal{F}$, and our remaining task is to show that the mapping degree of $\mathcal{F}$ is $\pm1$.
Set $\alpha_{i}^{\pm} := [\Si_{i}^{\pm}]$ for $i=1, \ldots, k$, and define
\[
\Hyp_{\Si_{i}^{\pm}} := \Set{(a_{1}, \ldots, a_{k}) \in \R^{k} | a_{i} = \frac{c \cdot \alpha_{i}^{\pm}}{H_{i} \cdot \alpha_{i}^{\pm}}}.
\]
We shall show that, for all $i$, the image $\mathcal{F}(F(\Si_{i}^{\pm},R))$ is contained in an $\epsilon$-neighborhood of $\Hyp_{\Si_{i}^{\pm}}$ for sufficiently small $\epsilon>0$.
Then the conditions on $(c \cdot \alpha_{i}^{\pm})\cdot(H_{i} \cdot \alpha_{i}^{\pm})$ in \cref{defi: wall-crossing collection of surfaces} imply that $\mathcal{F}$ has mapping degree $\pm1$.

Let $i \in \{1, \ldots, k\}$.
For a metric $g$ on $N$, define
\[
\omega_{i,g}^{\pm} := \varphi^{-1}_{g}(p_{V^{+}}(\alpha_{i}^{\pm})) \in \mathscr{H}^{g_{+}}(N)
\]
Since we have that $\omega_{i,g}^{\pm} = \varphi^{-1}_{g}((H_{i} \cdot \alpha_{i}^{\pm})H_{i}) = (H_{i} \cdot \alpha_{i}^{\pm}) \varphi^{-1}_{g}(H_{i})$,
 we obtain that
\begin{align}
\label{eq: quotient}
\frac{c \cdot [\omega_{i,g}^{\pm}]}{H_{i} \cdot \alpha_{i}^{\pm}}
= c \cdot \varphi^{-1}_{g}(H_{i}).
\end{align}

Let $r \in F(\Si_{i}^{\pm},R)$ and set $g=\phi^{N}_{\bSi}(r)$.
By Lemma~3.2 in \cite{Konno}, we have that
\[
|\alpha_{i}^{\pm} - [\omega_{i,g}^{\pm}]| \leq \frac{C}{R^{1/4}},
\]
where $C$ is a constant depends only on the metric outside neighborhoods of $\Si_{i}^{\pm}$.
Therefore, by \eqref{eq: quotient},
$c \cdot \varphi^{-1}_{g}(H_{i})$ is close to
$(c \cdot \alpha_{i}^{\pm})/(H_{i} \cdot \alpha_{i}^{\pm})$
for sufficiently large $R$.
Namely, the $i$-th component of $\mathcal{F}(r)$ is close to $(c \cdot \alpha_{i}^{\pm})/(H_{i} \cdot \alpha_{i}^{\pm})$, and this implies that $\mathcal{F}(F(\Si_{i}^{\pm},R))$ is contained in an $\epsilon$-neighborhood of $\Hyp_{\Si_{i}^{\pm}}$ for sufficiently small $\epsilon$.
This completes the proof of the proposition.
\end{proof}

\subsection{Non-vanishing theorem}

\label{subsection: Non-vanishing theorem}

In this subsection we prove a non-vanishing theorem for the cohomological invariant $\SWcoh(X, \fraks)$.
We also derive constraints connected with the adjunction inequality on configurations of embedded surfaces from the non-vanishing theorem.

Let $(M, \fraks_0)$ be a smooth oriented closed $\spc$ $4$-manifold with $b^+(M) \geq 2$ and $d(\fraks_0) = 0$, and $k, l$ be natural numbers with $l \geq k > 0$.
We set $N = k\CP^2 \# l(-\CP^2)$ and $(X, \fraks) = (M \# N, \fraks_0\#\frakt)$, where $\frakt$ is the $\spc$ structure defined by \eqref{eq: first Chern class defining property} on $N$.
Since we have $b^{+}(X) \geq k+2$ and $d(\fraks) = -k$, the invariant $\SWcoh(X, \fraks)$ is defined and belongs to $H^{k-1}(\calK(X, \fraks);\Z)$.
Recall that, for a wall-crossing collection of surfaces $\bSi = \{\Si_i^{\epsilon}\}_{1 \leq i \leq k, \epsilon \in \{+,-\}}$ equipped with an orientation in the sense of \cref{defi: orientation of bSi}  in $(N, \frakt)$, we can consider the element $[K(\bSi)] \in H_{k-1}(\calK(X, \fraks);\Z)$ as in \cref{rem: regarding a subcomplex on connected sum}.
We now have the following non-vanishing theorem:

\begin{theo}
\label{theo: non-vanishing theorem}
Let $\bSi = \{\Si_i^{\epsilon}\}_{1 \leq i \leq k, \epsilon \in \{+,-\}}$ be a wall-crossing collection of surfaces in $(N, \frakt)$.
Assume that $\Si_i^+$ and $\Si_i^-$ intersect transversally for each $i$.
Let us equip $\bSi$ with an orientation in the sense of \cref{defi: orientation of bSi}.
Then
\[
\left< \SWcoh(X, \fraks), [K(\bSi)] \right> = \pm \SWinv(M, \fraks_{0})
\]
holds.
In particular, if $\SWinv(M, \fraks_{0}) \neq 0$, then we obtain 
\[
\SWcoh(X, \fraks) \neq 0 \ {\it in}\  H^{k-1}(\calK(X, \fraks);\Z)
\]
and
\[
[K(\bSi)] \neq 0 \ {\it in}\ H_{k-1}(\calK(X, \fraks);\Z).
\]
\end{theo}

\begin{rem}
There could be a generalization of \cref{theo: non-vanishing theorem}
which includes the case that $b^{+}(M)=1$.
However, the proof of \cref{theo: non-vanishing theorem} depends on the gluing argument in \cite{BK}, and the case that $b^{+}=1$ is not
discussed in \cite{BK}.
Therefore we do not deal with such a generalization in this paper.
%
\end{rem}

\begin{proof}[Proof of \cref{theo: non-vanishing theorem}]
To prove the theorem, we combine the setting of \cref{subsection: description of higher-dimensional wall-crossing phenomena in terms of embedded surfaces} with that of \cref{subsection: Ruberman's combination of wall-crossing and gluing arguments} as follows.
Let $B_{M}^{4} \subset M$ and $B_{N}^{4} \subset N$ be small balls such that $B_{N}^{4}$ is contained in the complement of some neighborhood of the surfaces belonging to $\bSi$, namely $\Si^{\epsilon}_{i}$'s for $i \in \{1, \ldots, k\}$ and $\epsilon \in \{+,-\}$.
Take a metric $g^{N} \in \Met(N)$ which is cylindrical near the boundary of $B^{4}_{N}$.
We also fix functions $N(\cdot)$ and $a(\cdot)$ for the surfaces belonging to $\bSi$.
Then we can define the map $\phi^{N}_{\tau}$ for each $\tau \prec \si_{\bep}$ and $\bep \in \Signk$ using $g^{N}$, $N(\cdot)$, and $a(\cdot)$ for the surfaces belonging to $\tau$.

Recall that surfaces belonging to $\bSi$ can be regarded as vertices of $\calK(X, \fraks)$.
Let us consider the metric $g = g^{M} \# g^{N} \in \Met(X)$ defined by connected sum.
Then we obtain the map $\phi_{\tau}$ defined as in \eqref{eq: map for total stretching} for each $\tau \prec \si_{\bep}$ and $\bep$ using $g$, $N(\cdot)$, and $a(\cdot)$ for the surfaces belonging to $\si_{\bep}$.
We now see that $R^{X}(\tau) := R(\tau, g, N(\cdot), a(\cdot))$ is independent of the choice of the length of the cylinder used to define the connected sum.
First note that $\lambda(\tau) = \lambda(\tau, a(\cdot))$ is only determined by $a(\cdot)$ for faces of $\tau$.
To see that $C(\tau) = C(\tau, g, N(\cdot), a(\cdot))$ is independent of the choice of the length, take any metric  $h \in \phi_{\tau}(|\tau|)$.
Since $h$ can be written as the connected sum of $g^{M}$ and a metric changed from $g^{N}$ only near the surfaces belonging to $\tau$,
the restriction of $h$ to the cylinder in the connected sum coincides with the standard cylindrical metric, and hence is a metric with positive scalar curvature.
From this we deduce that the restriction of $\kappa(h)$ to the cylinder is identically zero.
Therefore $\|\kappa(h)\|_{L^{2}(X, h)}$ is invariant under any change of the length of the cylinder.
We thus have that $R^{X}(\tau)$ is independent of the choice of the length of the cylinder, and we write it without length parameter of the cylinder.

Let us consider the collection of positive numbers $\{R^{X}(\tau)\}_{\tau \prec \si_{\bep}, \bep \in \Signk}$.
Then the maps 
\[
\{\vp_{\tau,R^{X}(\tau)}^{N} : [0, R^{X}(\tau)]^{\tau} \to \Pi(N, \fraks)\}_{\tau \prec \si_{\bep}, \bep \in \Signk}
\]
are obtained as in \cref{subsection: description of higher-dimensional wall-crossing phenomena in terms of embedded surfaces}.
We extend $N(\cdot)$ and $a(\cdot)$ to any vertices of $\calK(X, \fraks)$.
Let $(g^{M}, \mu) \in \Pi(M, \fraks_0)$ be a generic point.
If we need, we may arrange that the metric $g^{M}$ and the metric-component of each point in the image of $\vp^{N}_{\tau, R^{X}(\tau)}$ for each $\tau \prec \si_{\bep}$ and $\bep \in \Signk$ are cylindrical near the boundaries of $B_{M}^{4}$ and of $B_{N}^{4}$ respectively.
Then the map $\phi_{\tau}$ for each $\tau \prec \si_{\bep}$ and each $\bep$ can be written as
\begin{align}
\phi_{\tau} (x) = g^{M}\#\phi^{N}_{\tau}(x)
\label{eq: connected sum expression of phi}
\end{align}
for each $x \in \dom(\phi_{\tau}) = \dom(\phi^{N}_{\tau})$.
We may also assume that $\mu$ and the self-dual $2$-form-component of each point in the image of $\vp^{N}_{\tau,R^{X}(\tau)}$ for each $\tau$ are supported on the complement of $B_{M}^{4}$ and of $B_{N}^{4}$ respectively; as in \cref{subsection: Ruberman's combination of wall-crossing and gluing arguments},
we may define 
\[
\vp_{\tau} : [0, R^{X}(\tau)]^{\tau} \to \Pi(X, \fraks)
\]
by
\begin{align}
\vp_{\tau} (x) := (g^{M}, \mu)\#\vp^{N}_{\tau, R^{X}(\tau)}(x)
\label{eq: def using connected sum}
\end{align}
using long-stretched connected sum.
We also take $\vp_{\bullet}$ for any vertices of $\calK(X, \fraks)$ as an extension of $\vp_{\bullet}$ for $\tau$'s.
Then we can use the tuple $\calA = (g, N(\cdot), a(\cdot),  \vp_{\bullet})$ as an additional data to construct a cochain on $\calK(X, \fraks)$ through the procedure in \cref{subsection: Definition of the invariant}.

For each $\bep \in \Signk$,
we set $R^{N}(\si_{\bep}) := R(\si_{\bep}, g^{N}, N(\cdot), a(\cdot))$ and 
\[
R_{\bSi} := \max_{\bep \in \Signk}\{R^{X}(\si_{\bep}), R^{N}(\si_{\bep})\}.
\]
We here define $\vp_{\si_{\bep}, R_{\bSi}} : [0, R_{\bSi}]^{\si_{\bep}} \to \circPi(X, \fraks)$ as follows.
Let us consider the decomposition
\[
[0, R_{\bSi}]^{\si_{\bep}}
= [0, R^{X}(\si_{\bep})]^{\si_{\bep}} \cup \bigsqcup_{R^{X}(\si_{\bep}) \leq R \leq R_{\bSi}} \extdel(\si_{\bep}, R).
\]
The map $\vp_{\si_{\bep}} : [0, R^{X}(\si_{\bep})]^{\si_{\bep}} \to \circPi(X, \fraks)$ coincides with the composition
\begin{align}
\phi_{\si_{\bep}} \circ \left( \bigsqcup_{R^{X}(\si_{\bep}) \leq R \leq R_{\bSi}} \psi_{\si_{\bep}, R}^{-1} \right) : \bigsqcup_{R^{X}(\si_{\bep}) \leq R \leq R_{\bSi}} \extdel(\si_{\bep}, R) \to \circPi(X, \fraks)
\label{eq: parts consists disjoint union}
\end{align}
on $\extdel(\si_{\bep}, R^{X}(\si_{\bep}))$.
By gluing $\vp_{\si_{\bep}}$ and \eqref{eq: parts consists disjoint union} together, we obtain a map from $[0, R_{\bSi}]^{\si_{\bep}}$ to $\circPi(X, \fraks)$.
We define $\vp_{\si_{\bep}, R_{\bSi}}$ as this glued map.
Note that, by \cref{prop: vanishing on certain face},
there is no solution to the unperturbed \SW equations for any metric in the image of $\extdel(\si_{\bep}, R)$ by $\phi_{\si_{\bep}} \circ \psi_{\si_{\bep}, R}^{-1}$ if $R \geq R^{X}(\si_{\bep})$.
We therefore have
\begin{align}
\#\calM(\vp_{\si_{\bep}}, \fraks) = \#\calM(\vp_{\si_{\bep}, R_{\bSi}}, \fraks).
\label{eq: relation between two counting with different stretching}
\end{align}
We also note that, by the same argument to prove \cref{lem: vanishing on boundary}, there is no solution to the \SW equations for any element in $\vp_{\si_{\bep}, R_{\bSi}}(\del\left([0, R_{\bSi}]^{\si_{\bep}}\right))$.

Let us define $\vp^{N}_{\si_{\bep}, R_{\bSi}} : [0, R_{\bSi}]^{\si_{\bep}} \to \Pi(N, \frakt)$ using the same way to define $\vp_{\si_{\bep}, R_{\bSi}}$ above.
By the definition of $\vp_{\si_{\bep}}$ and \eqref{eq: connected sum expression of phi}, we have 
\begin{align}
\vp_{\si_{\bep}, R_{\bSi}}(x) = (g^{M}, \mu) \# \vp^{N}_{\si_{\bep}, R_{\bSi}}(x)
\label{eq: connected sum for vp}
\end{align}
for any $x \in [0, R_{\bSi}]^{\si_{\bep}}$.
By applying the construction of $\vp^{N}_{\bSi, R}$ in \cref{subsection: description of higher-dimensional wall-crossing phenomena in terms of embedded surfaces} to $R=R_{\bSi}$, we have the map $\vp^{N}_{\bSi, R_{\bSi}} : D^{k}_{R_{\bSi}} \to \Pi(N, \frakt)$ from $\vp^{N}_{\si_{\bep}, R_{\bSi}}$'s.
Note that $\vp^{N}_{\bSi, R_{\bSi}}$ satisfies that $\vp^{N}_{\bSi, R_{\bSi}}(S^{k}_{R_{\bSi}}) \subset \circPi(N, \frakt)$.
In a similar vein, $\vp_{\si_{\bep}, R_{\bSi}}$'s induce the map $\vp_{\bSi, R_{\bSi}} : D^{k}_{R_{\bSi}} \to \circPi(X, \fraks)$.
This map $\vp_{\bSi, R_{\bSi}}$ is also written as the connected sum of $(g^{M}, \mu)$ and $\vp^{N}_{\bSi, R_{\bSi}}$ by \eqref{eq: connected sum for vp}.
Let $\pi : \bigsqcup_{\bep \in \Signk} [0, R_{\bSi}]^{\si_{\bep}} \to D^{k}_{\bSi, R_{\bSi}}$ be the quotient map.
If we need, we replace $\vp_{\bSi, R_{\bSi}}$ with an approximation of it to be smooth on the complement of $\pi(\bigsqcup_{\bep \in \Signk} \del([0, R_{\bSi}]^{\si_{\bep}}))$.
From the first note in \cref{rem: foundation of disjoint counting} and the argument above, it follows that there is no solution on the space defined by substituting $R_{\bSi}$ for $R$ in \eqref{eq: identified area}.
We therefore have
\begin{align}
\# \calM(\vp_{\bSi, R_{\bSi}}, \fraks) = \sum_{\bep \in \Signk} \# \calM(\vp_{\si_{\bep}, R_{\bSi}}, \fraks).
\label{eq: counting equality right sum}
\end{align}
From the definition of $\SWcoh(X, \fraks)$, \eqref{eq: relation between two counting with different stretching}, and \eqref{eq: counting equality right sum}, we obtain
\begin{align*}
\left< \SWcoh(X, \fraks), [K(\bSi)] \right>
& = \SWcoch(X, \fraks, \calA)([K(\bSi)] )\\
& = \sum_{\bep \in \Signk} \# \calM(\vp_{\si_{\bep}}, \fraks)\\
& = \sum_{\bep \in \Signk} \# \calM(\vp_{\si_{\bep}, R_{\bSi}}, \fraks)\\
& = \# \calM(\vp_{\bSi, R_{\bSi}}, \fraks)\\
& = \pm(\vp^{N}_{\bSi, R_{\bSi}} \cdot \calW(N, \frakt)) \cdot \SWinv(M, \fraks_0)\\
& = \pm \SWinv(M, \fraks_0),
\end{align*}
where the fifth and the last equalities are deduced from \cref{prop : wall-crossing+gluing} and \cref{prop: description of higher-dim. wall-crossing} respectively.
(The sign might change in the last equality.)
This proves the \lcnamecref{theo: non-vanishing theorem}.
\end{proof}

\begin{rem}
\label{rem: non-vanishing compare with usual SW invariant}
Note that the usual \SW invariant on $X = M \# N$ in \cref{theo: non-vanishing theorem} vanishes for any $\spc$ structure on it since we have assumed that $b^{+}(M) \geq 2$ and $b^{+}(N) \geq 1$.
On the other hand, the invariant $\SWcoh(X, \fraks)$ does not vanish if $l \geq 4k$ holds and if we take $(M, \fraks_{0})$ so that $\SWinv(M, \fraks_{0}) \neq 0$.
Indeed, we may give a wall-crossing collection of surfaces in $(N, \frakt)$ with $l \geq 4k$. 
(See \cref{ex: wall-crossing collection of surfaces}.)
\end{rem}

To be concrete, we give the following non-vanishing of $\SWcoh(X, \fraks)$:

\begin{cor}
\label{cor: non-vanishing on psc manifold}
Let $k$ and $l$ be positive integers with $l \geq 4k$.
Then there is a $spin^{c}$ structure $\fraks$ on
\[
X = (k+3)\CP^{2} \# (l+19)(-\CP^{2})
\]
such that
\[
\SWcoh(X, \fraks) \neq 0 \ {\it in}\  H^{k-1}(\calK(X, \fraks);\Z).
\]
\end{cor}

\begin{proof}
In \cref{theo: non-vanishing theorem} let us take $(K3, \fraks_{\rm can})$ as $(M, \fraks_{0})$, where $K3$ is a K3 surface and $\fraks_{\rm can}$ is the canonical $\spc$ structure of $K3$.
Then $\SWinv(M, \fraks_{0}) \neq 0$ holds and $M \# N$ is diffeomorphic to $X$ in the statement of this \lcnamecref{cor: non-vanishing on psc manifold} \cite{MR537731,MR0491730}.
Let $\fraks$ be the $\spc$ structure on $X$ corresponding to $\fraks_{0} \# \frakt$ via $X \cong M \# N$.
A wall-crossing collection of surfaces in $(N, \frakt)$ will be given in \cref{ex: wall-crossing collection of surfaces}, and therefore we have $\SWcoh(X, \fraks) \neq 0$.
\end{proof}

\begin{rem}
The $4$-manifold $X$ in \cref{cor: non-vanishing on psc manifold} admits a metric with positive scalar curvature, though $\SWcoh(X, \fraks)$ does not vanish.
Note that, in contrast,  both of the \SW invariant and its refinement called the Bauer--Furuta invariant~\cite{MR2025298} vanish on any $4$-manifold admitting a positive scalar curvature metric.
\end{rem}

\begin{rem}
More generally, if we take sufficiently large numbers as $k$ and $l$, and a simply connected $4$-manifold as $M$ in \cref{theo: non-vanishing theorem},
by Wall's result~\cite{MR0163324}, $X$ in \cref{theo: non-vanishing theorem} is diffeomorphic to $k'\CP^{2} \# l'(-\CP^{2})$ for some $k'$ and $l'$.
\end{rem}

We here interpret \cref{theo: non-vanishing theorem} from a classical point of view:
\Cref{theo: non-vanishing theorem} yields certain adjunction-type genus constraints on configurations of embedded surfaces.
The most general formulation of the constraints can be described using the following concept, which we call {\it bounding collections of surfaces}.
Henceforth we follow the notation $(M, \fraks_{0})$, $(N, \frakt)$, and $(X, \fraks)$ given at the beginning of this subsection.
Recall that if we give an orientation of $\bSi$ in the sense of \cref{defi: orientation of bSi}, then we can consider the fundamental chain/class of $K(\bSi)$.
We also simply write the fundamental chain $K(\bSi)$.

\begin{defi}
\label{defi: bounding collection of surfaces}
Let $\bSi = \{\Si_i^{\epsilon}\}_{1 \leq i \leq k, \epsilon \in \{+,-\}}$ be a wall-crossing collection of surfaces in $(N, \frakt)$.
We call a finite subset $\bS \subset V(\bar{\calK}(X))$ a {\it bounding collection of surfaces in $X$ for $\bSi$} if there exist a finite set $\Lambda$, an integer $c_{\lambda}$, and a $k$-simplex $\si_{\lambda}$ of $\bar{\calK}(X)$ for each $\lambda \in \Lambda$ such that
\begin{enumerate}
\item $\bSi \subset \bigcup_{\lambda \in \Lambda} \si_{\lambda}$,
\item $\bS = \bigcup_{\lambda \in \Lambda} \si_{\lambda} \setminus \bSi$, and
\item $\del(\sum_{\lambda \in \Lambda} c_{\lambda} \lb \si_{\lambda} \rb) = K(\bSi)$ holds in $C_{k-1}(\bar{\calK}(X))$ for some orientations of $\si_{\lambda}$'s and some orientation of $\bSi$ in the sense of \cref{defi: orientation of bSi}.
\label{item: bounding condition for the fundamental cycle}
\end{enumerate}
Here in \eqref{item: bounding condition for the fundamental cycle} above $K(\bSi)$ denotes the fundamental chain and is regarded as an element in $C_{k-1}(\bar{\calK}(X))$ via the natural inclusion $ C_{\ast}(\calK(X, \fraks)) \inc C_{\ast}(\bar{\calK}(X))$.
\end{defi}

\begin{ex}
\label{ex: a bounding collection of surfaces consisting of a single surface}
For any wall-crossing collection of surfaces $\bSi = \{\Si_i^{\epsilon}\}_{1 \leq i \leq k, \epsilon \in \{+,-\}}$ in $(N, \frakt)$, there exists a bounding collection of surfaces in $X$.
(In particular, if we give a suitable orientation of $\bSi$,
\[
[K(\bSi)] = 0\ {\rm in}\ H_{k-1}(\bar{\calK}(X);\Z)
\]
holds.
Compare this equality in $H_{k-1}(\bar{\calK}(X);\Z)$ with the non-vanishing result in $H_{k-1}(\calK(X, \fraks);\Z)$ in \cref{theo: non-vanishing theorem}.)
Indeed, let us take $S \in V(\bar{\calK}(X))$ with
\begin{align}
S \cap \Si_i^{\epsilon} = \emptyset \quad (\forall \epsilon \in \{+,-\},\ \forall i \in \{1, \ldots, k\}).
\label{eq: surface which disjoint from all menber}
\end{align}
Then $\bS = \{S\}$ is a bounding collection of surfaces.
We will simply write the condition \eqref{eq: surface which disjoint from all menber}
\[
S \cap \bSi = \emptyset.
\]

To give $S$ with $S \cap \bSi = \emptyset$ corresponds to considering a ``cone'' of $K(\bSi)$.
Since $|K(\bSi)| \cong S^{k-1}$, the cone of $K(\bSi)$ is a $k$-dimensional disk and therefore $K(\bSi)$ is homologically trivial.
For example, in the case that $k=2$, the bounding by $S$ of $\bSi$ is described in the left picture in \cref{figure : Some bounding for KbSi}.

\begin{figure}
\begin{center}
\scalebox{1}{
\begin{tikzpicture}
[xscale = 0.8, yscale = 0.8]
\draw [thick] (4,2) -- (6,2);
\draw [thick] (4,0) -- (6,0);
\draw [thick] (4,2) -- (4,0);
\draw [thick] (6,2) -- (6,0);
\draw [thick] (4,2) -- (6,0);
\draw [thick] (4,0) -- (6,2);
\fill(4,2) circle (2.2 pt);
\fill(6,2) circle (2.2 pt);
\fill(4,0) circle (2.2 pt);
\fill(6,0) circle (2.2 pt);
\fill(5,1) circle (2.2 pt);
\draw(3.7,2.4) node {$\Si_1^{+}$};
\draw(6.3,2.4) node {$\Si_2^{+}$};
\draw(6.3,1-1.4) node {$\Si_1^{-}$};
\draw(3.7,1-1.4) node {$\Si_2^{-}$};
\draw(5,1.4) node {$S$};
\fill [gray, opacity=.3] (4,2) -- (6, 2) -- (6, 0) -- (4,0);
\draw [thick] (10,0) -- (10,3);
\draw [thick] (10,0) -- (13,0);
\draw [thick] (13,0) -- (13,3);
\draw [thick] (13,3) -- (10,3);
\draw [thick] (10,3) -- (13,0);
\draw [thick] (10,0) -- (11,2);
\draw [thick] (10,0) -- (12,1);
\draw [thick] (13,3) -- (11,2);
\draw [thick] (13,3) -- (12,1);
\fill(10,0) circle (2.2 pt);
\fill(13,0) circle (2.2 pt);
\fill(10,3) circle (2.2 pt);
\fill(13,3) circle (2.2 pt);
\fill(11,2) circle (2.2 pt);
\fill(12,1) circle (2.2 pt);
\draw(10-0.3,3.4) node {$\Si_1^{+}$};
\draw(13.3,3.4) node {$\Si_2^{+}$};
\draw(13.3,-0.4) node {$\Si_1^{-}$};
\draw(10-0.3,-0.4) node {$\Si_2^{-}$};
\draw(11+0.1,2.4) node {$S$};
\draw(12-0.1,0.6) node {$S'$};
\fill [gray, opacity=.3] (10,0) -- (13, 0) -- (13, 3) -- (10,3);
\end{tikzpicture}
}
\end{center}
\caption{Some bounds for $K(\bSi)$}
\label{figure : Some bounding for KbSi}
\end{figure}
\end{ex}

\begin{ex}
\label{ex: infinitely many bounding}
In fact, there are infinitely many distinct bounding collections of surfaces in $X$ for any wall-crossing collection of surfaces $\bSi = \{\Si_i^{\epsilon}\}_{1 \leq i \leq k, \epsilon \in \{+,-\}}$ in $(N, \frakt)$.
Indeed, there are infinitely many triangulations of a disk, which is regarded as the cone of $K(\bSi)$.
If we fix such a triangulation, roughly speaking, by assigning elements of $V(\bar{\calK}(X))$ to vertices of the triangulation, we have a bounding collection of surfaces $\bS$ for $\bSi$.
For example in the case that $k=2$, let us consider the triangulation of the $2$-dimensional disk described in the right picture in \cref{figure : Some bounding for KbSi}.
The corresponding bounding collection of surfaces is as follows:
let us consider $S, S' \in V(\bar{\calK}(X))$ such that 
\begin{align}
S \cap \Si_2^{\pm} = S \cap \Si_1^{+} = \emptyset,\ 
S' \cap \Si_2^{\pm} = S' \cap \Si_1^{-} = \emptyset,\ {\rm and}\ 
S \cap S' = \emptyset.
\label{eq: bounding condition using two surfaces}
\end{align}
Then $\bS = \{S, S'\}$ is a bounding collection of surfaces for $\bSi$.
\end{ex}

We now have the following adjunction-type genus constraints for any bounding collection of surfaces:

\begin{cor}
\label{cor: the adjunction inequality for infinitely many configurations}
Suppose that $\SWinv(M, \fraks_{0}) \neq 0$.
Let $\bSi = \{\Si_i^{\epsilon}\}_{1 \leq i \leq k, \epsilon \in \{+,-\}}$ be a wall-crossing collection of surfaces in $(N, \frakt)$.
Assume that $\Si_i^+$ and $\Si_i^-$ intersect transversally for each $i$.
Then, for any bounding collection of surfaces $\bS$ in $X$ for $\bSi$, there exists a surface $S \in \bS$ such that the inequality 
\begin{align}
 \Euler^-(S) \geq |c_1(\fraks) \cdot [S]|
 \label{eq: the adjunction inequality as an application}
 \end{align}
 holds.
 
 In particular, for $S \in V(\bar{\calK}(X))$ with $S \cap \bSi = \emptyset$, the inequality \eqref{eq: the adjunction inequality as an application} holds.
\end{cor}

\begin{proof}
Assume that all elements in $\bS$ were to violate the adjunction inequalities.
Then we would have 
\begin{align}
\del(\sum_{\lambda \in \Lambda} c_{\lambda} \lb \si_{\lambda} \rb) = K(\bSi)
\label{eq: the bounding in not only in barcalK but also in calK}
\end{align}
for some $\Lambda$, $c_{\lambda}$, $\si_{\lambda}$, and some orientations of $\si_{\lambda}$ and of $K(\bSi)$,
where $\si_{\lambda}$ is a simplex not only of $\bar{\calK}(X)$ but also of $\calK(X, \fraks)$.
Hence the equation \eqref{eq: the bounding in not only in barcalK but also in calK} would imply that $[K(\bSi)] = 0$ in $H_{k-1}(\calK(X, \fraks);\Z)$.
This contradicts \cref{theo: non-vanishing theorem}.
\end{proof}

\begin{rem}
As we mentioned in \cref{rem: non-vanishing compare with usual SW invariant}, the usual \SW invariant vanishes on $X$ for any $\spc$ structure on it.
\Cref{cor: the adjunction inequality for infinitely many configurations} therefore gives the adjunction inequality for a $\spc$ $4$-manifold whose \SW invariant vanishes under some condition on geometric mutual intersections of embedded surfaces.
This point is similar to Strle~\cite{MR2064429} and the author's result~\cite{Konno}.

On the other hand, from \cref{ex: infinitely many bounding}, \cref{cor: the adjunction inequality for infinitely many configurations} asserts that {\it one} non-trivial element in $H_{\ast}(\calK(X, \fraks);\Z)$ provides constraints on {\it infinitely many} configurations of embedded surfaces.
This is a phenomenon newly found through our use of the simplicial homology theory.
\end{rem}

\begin{rem}
In the statement of \cref{cor: the adjunction inequality for infinitely many configurations}, if an element of $\bS$ is contained in $M$ before taking the connected sum with $N$, the adjunction inequality for the element follows from an argument of Kronheimer--Mrowka~\cite{MR1306022} since we have assumed that $\SWinv(M, \fraks_{0}) \neq 0$ in \cref{cor: the adjunction inequality for infinitely many configurations}.
\end{rem}

We can generalize the constraint on configurations given in \cref{cor: the adjunction inequality for infinitely many configurations} to surfaces with positive self-intersection number.
To state it, let us introduce a variant of $\bar{\calK}$:

\begin{defi}
For a smooth oriented closed connected $4$-manifold $X$, we define an abstract simplicial complex $\bar{\calK}_{\geq0} = \bar{\calK}_{\geq0}(X)$ as follows: 
\begin{itemize}
\item The vertices of $\bar{\calK}_{\geq0}$ are defined as smoothly embedded surfaces with non-negative self-intersection number.
\item For $n \geq 1$, a collection of $(n+1)$ vertices $\Si_0, \ldots, \Si_n \in V(\bar{\calK}_{\geq0})$ spans an $n$-simplex if and only if $\Si_0, \ldots, \Si_n$ are mutually disjoint.
\end{itemize}
\end{defi}

We now extend the concept of bounding collections of surfaces to finite subsets of $V(\bar{\calK}_{\geq0})$ by the quite same way in \cref{defi: bounding collection of surfaces}.

\begin{defi}
Let $\bSi$ be a wall-crossing collection of surfaces in $(N, \frakt)$.
We call a finite subset $\bS \subset V(\bar{\calK}_{\geq0}(X))$ a {\it bounding collection of surfaces in $X$ for $\bSi$} if the condition obtained by replacing $\bar{\calK}$ with $\bar{\calK}_{\geq0}$ in the defining condition in \cref{defi: bounding collection of surfaces} is satisfied.
\end{defi}

Using this extension of the concept of bounding collections of surfaces, we can generalize \cref{cor: the adjunction inequality for infinitely many configurations} as follows:

\begin{cor}
\label{cor: adjunction type constraint for surfaces with positive intersection}
Suppose that $\SWinv(M, \fraks_{0}) \neq 0$.
Let $\bSi = \{\Si_i^{\epsilon}\}_{1 \leq i \leq k, \epsilon \in \{+,-\}}$ be a wall-crossing collection of surfaces in $(N, \frakt)$.
Assume that $\Si_i^+$ and $\Si_i^-$ intersect transversally for each $i$.
Then, for any bounding collection of surfaces  $\bS \subset V(\bar{\calK}_{\geq0}(X))$ in $X$ for $\bSi$, there exists a surface $S \in \bS$ such that the inequality 
\begin{align}
 \Euler^-(S) \geq |c_1(\fraks) \cdot [S]| + [S]^{2}
 \label{eq: the adjunction inequality for positive self intersection}
 \end{align}
 holds.
 
 In particular, for $S \in V(\bar{\calK}_{\geq0}(X))$ with $S \cap \bSi = \emptyset$, the inequality \eqref{eq: the adjunction inequality for positive self intersection} holds.
\end{cor}

\begin{proof}
Let us write $\bS = \{S_{1}, \ldots, S_{m}\}$ ($m>0$) and set $l_{i} := [S_{i}]^{2} \geq 0$ for each $i \in \{ 1, \ldots, m \}$.
We also set $l_{0} := 0$ and $l_{i}' := \sum_{a=0}^{i-1} l_{a}$ for each $i \in \{1, \ldots, m+1\}$.
Consider the blowup
\[
M' := M \# l_{1}(-\CP^{2}) \# \cdots \# l_{m}(-\CP^{2}).
\]
Let $\fraks_{0}'$ be the blowup $\spc$ structure on $M'$, whose first Chern class is given as
\[
c_{1}(\fraks_{0}') = c_{1}(\fraks_{0}) + \sum_{j=1}^{l_{m+1}'}E,
\]
where $E$ is a generator of $H^{2}(-\CP^{2}) \cong H_{2}(-\CP^{2})$.
We set $(X', \fraks') := (M' \# N, \fraks_{0}' \# \frakt)$.
Let $S_{i}'$ $(1 \leq i \leq m)$ be an embedded surface in $X'$ which is obtained from the connected sum of $S_{i}$ and exceptional curves in $l_{i}(-\CP^{2})$ and whose homology class is given as
\[
[S'_{i}] = [S_{i}] - \sum_{j=l_{i}'+1}^{l_{i+1}'}E
\]
if $c_{1}(\fraks) \cdot [S_{i}] \geq 0$, and
\[
[S'_{i}] = [S_{i}] + \sum_{j=l_{i}'+1}^{l_{i+1}'}E
\]
if $c_{1}(\fraks) \cdot [S_{i}] < 0$.
Set $\bS' := \{S_{0}', \ldots,S_{m}'\}$.
Then $\bS' \subset V(\bar{\calK}(X'))$ holds and $\bS'$ is a bounding collection of surfaces in $X'$ for $\bSi$.
Since the \SW invariant of $(M, \fraks_{0})$ coincides with that of the blowup $(M', \fraks_{0}')$ by the blowup formula, we can apply \cref{cor: the adjunction inequality for infinitely many configurations} for the new bounding collection of surfaces $\bS'$.
Then we have the adjunction inequality
\begin{align*}
\Euler^-(S'_{i}) \geq |c_1(\fraks') \cdot [S'_{i}]|
 \end{align*}
for some $i$.
If $c_{1}(\fraks) \cdot [S_{i}] \geq 0$, we have
\begin{align*}
\Euler^-(S_{i}) = \Euler^-(S'_{i})
\geq |c_1(\fraks') \cdot [S'_{i}]|
= |c_1(\fraks) \cdot [S_{i}] + [S_{i}]^{2}|
= |c_1(\fraks) \cdot [S_{i}]| + [S_{i}]^{2},
\end{align*}
and if $c_{1}(\fraks) \cdot [S_{i}] < 0$, we have
\begin{align*}
\Euler^-(S_{i}) = \Euler^-(S'_{i})
\geq |c_1(\fraks') \cdot [S'_{i}]|
= |c_1(\fraks) \cdot [S_{i}] - [S_{i}]^{2}|
= |c_1(\fraks) \cdot [S_{i}]| + [S_{i}]^{2}.
\end{align*}
Hence the inequality \eqref{eq: the adjunction inequality for positive self intersection}  holds for some $S \in \bS$.
\end{proof}

\subsection{Examples of wall-crossing collections of surfaces}
\label{subsection: Examples}

We give some examples of wall-crossing collections of surfaces to complete the proof of \cref{cor: non-vanishing on psc manifold} and to understand \cref{cor: adjunction type constraint for surfaces with positive intersection} concretely.

\begin{ex}
\label{ex: wall-crossing collection of surfaces}

Let $k$, $l_{i}$, $d_{+,i}$, and $d_{-,i}$ $(1 \leq i \leq k)$ be positive integers satisfying
\begin{align}
l_{i} \geq d_{\epsilon,i}^2 \geq 4
\label{eq: condition for d pm}
\end{align}
for each $\epsilon \in \{+, -\}$ and $i \in \{1, \ldots, k\}$.
We set $l := \sum_{i=1}^{k} l_{i}$, $l_{0} := 0$, and $l_{i}' := \sum_{a=0}^{i-1} l_{a}$ for each $i \in \{1, \ldots, k\}$.
Let us consider $N = k\CP^2 \# l(-\CP^2) = \#_{i=1}^k \CP^2_{i} \# (\#_{j=1}^l (-\CP^2))$ and let $H_{i}$, $E$, and $\frakt$ be the one in \cref{subsection: Ruberman's combination of wall-crossing and gluing arguments}.
Henceforth we identify $H^{2}(\CP^{2}_{i})$ with $H_{2}(\CP^{2}_{i})$ by the Poincar\'{e} duality and similarly $H^{2}(-\CP^{2})$ with $H_{2}(-\CP^{2})$.
Let $\Si^+_{i}, \Si^-_{i}\ (1 \leq i \leq k)$ be embedded surfaces in $N$
satisfying \eqref{eq: disjoint condition} and the following two conditions for each $i \in \{1, \ldots, k\}$:
\begin{align}
[\Si^\pm_{i}] &= d_{\pm,i}H_{i} \pm \sum_{j = l_{i}' + 1}^{l_{i}' + d_{\pm,i}^2} E,\ {\rm and} \label{eq: expression in homology level}\\
\genus(\Si^\pm_{i}) &= \frac{(d_{\pm,i}-1)(d_{\pm, i}-2)}{2}. \label{eq: genus given by algebraic curve}
\end{align}
Here double-signs $\pm$ correspond in the both sides of \eqref{eq: expression in homology level} and \eqref{eq: genus given by algebraic curve}.
(One can make such surfaces, for example, using connected sum in $N$ of a smooth algebraic curve of degree $d_{\pm, i}$ in $\CP^{2}_{i}$ $(1 \leq i \leq k)$ and the exceptional curves in $-\CP^{2}$'s.
In this section, if we say an algebraic curve it means a smooth algebraic curve.)
It is easy to check that $\Si^{\epsilon}_{i} \in V(\calK(N, \frakt))$ holds for each $\epsilon \in \{+,-\}$, $i \in \{1, \ldots, k\}$ and that the collection $\bSi = \{\Si_i^{\epsilon}\}_{1 \leq i \leq k, \epsilon \in \{+,-\}}$ is a wall-crossing collection of surfaces in $(N, \frakt)$.

From \cref{cor: adjunction type constraint for surfaces with positive intersection}, we therefore have the following constraint on configurations.
Let $(M, \fraks_0)$ be a smooth oriented closed $\spc$ $4$-manifold with $b^+(M) \geq 2$, $d(\fraks_0) = 0$, and $\SWinv(M, \fraks_0) \neq 0$.
Set $(X, \fraks) := (M \# N, \fraks_0\#\frakt)$.
Let $\bS \subset V(\bar{\calK}_{\geq0}(X))$ be a bounding collection of surfaces in $X$ for $\bSi$.
(As we mentioned in \cref{ex: infinitely many bounding}, there are infinitely many bounding collections of surfaces for $\bSi$.)
Then the inequality
\begin{align*}
 \Euler^-(S) \geq |c_1(\fraks) \cdot [S]| + [S]^{2}
 \end{align*}
holds for some surface $S \in \bS$.
\end{ex}

Similarly one can easily make other examples of wall-crossing collections of surfaces.
For example, the signs of $E$'s in \eqref{eq: expression in homology level} have some freedom for many $d_{\pm,i}$.
We also give other examples as follows.

\begin{ex}
\label{ex: wall-crossing collection of surfaces with coefficient 2}
Let $k$, $l_{i}$, $d_{+,i}$, and $d_{-,i}$ $(1 \leq i \leq k)$ be positive integers satisfying
$l_{i} \geq d_{\epsilon,i}^{2}-3$ and $d_{\epsilon,i} \geq 3$ for each $\epsilon \in \{+, -\}$ and $i \in \{1, \ldots, k\}$.
Let $l$, $l_{0}$, $l_{i}'$ and $(N, \frakt)$ be the one in \cref{ex: wall-crossing collection of surfaces}.
Let $\Si^+_{i}, \Si^-_{i}\ (1 \leq i \leq k)$ be embedded surfaces in $N$
satisfying \eqref{eq: disjoint condition}, \eqref{eq: genus given by algebraic curve}, and 
\begin{align*}
[\Si^\pm_{i}] &= d_{\pm,i}H_{i} + 2E_{l_{i}'+1} \pm \sum_{j = l_{i}'+2}^{l_{i}'+d_{\pm,i}^{2}-3} E
\end{align*}
for each $i \in \{1, \ldots, k\}$.
Since $2E_{l_{i}'+1}$ can be represented by a sphere in $-\CP^{2}_{l_{i}'+1}$, one can make such surfaces $\Si^{+}_{i}$ and $\Si^{-}_{i}$ as the connected sum of an algebraic curve of degree $d_{\pm,i}$ in $\CP^{2}_{i}$, the sphere in $-\CP^{2}_{l_{i}'+1}$, and the exceptional curves in $-\CP^{2}$ with $l_{i}'+2 \leq j \leq l_{i}'+d_{\pm,i}^{2}-3$.
It is easy to check that $\Si^{\epsilon}_{i} \in V(\calK(N, \frakt))$ holds for each $\epsilon$, $i$ and that the collection $\bSi = \{\Si_i^{\epsilon}\}_{1 \leq i \leq k, \epsilon \in \{+,-\}}$ is a wall-crossing collection of surfaces in $(N, \frakt)$.
\end{ex}

\begin{ex}
\label{ex: wall-crossing collection of surfaces times coefficient 2}
Let $k$, $l_{i}$, $d_{+,i}$, and $d_{-,i}$ $(1 \leq i \leq k)$ be positive integers satisfying
$d_{-,i} \in \{2,3\}$ for each $i \in \{1, \ldots, k\}$
and \eqref{eq: condition for d pm} for each $\epsilon \in \{+, -\}$ and $i$.
Let $l$, $l_{0}$, $l_{i}'$ and $(N, \frakt)$ be the one in \cref{ex: wall-crossing collection of surfaces}.
For each $i \in \{1, \ldots, k\}$, let $\Si^+_{i}$ be an embedded surface in $N$
satisfying the $+$-part of the conditions \eqref{eq: expression in homology level} and \eqref{eq: genus given by algebraic curve}, and $\Si^-_{i}$ be an embedded surface in $N$
satisfying
\begin{align*}
[\Si^-_{i}] &= 2d_{-,i}H_{i} - 2\sum_{j = l_{i}' + 1}^{l_{i}' + d_{-,i}^2} E,\ {\rm and}\\
\genus(\Si^-_{i}) &= \frac{(2d_{-,i}-1)(2d_{-, i}-2)}{2}.
\end{align*}
We also assume that $\Si^+_{i}$ and $\Si^-_{i}$ satisfy \eqref{eq: disjoint condition}.
One can make such surfaces $\Si^{+}_{i}$ and $\Si^{-}_{i}$ as the connected sum of an algebraic curve in $\CP^{2}_{i}$ and spheres in $-\CP^{2}$'s.
It is easy to check that $\Si^{\epsilon}_{i} \in V(\calK(N, \frakt))$ holds for each $\epsilon$, $i$ and that the collection $\bSi = \{\Si_i^{\epsilon}\}_{1 \leq i \leq k, \epsilon \in \{+,-\}}$ is a wall-crossing collection of surfaces in $(N, \frakt)$.
\end{ex}

\begin{rem}
\label{rem: comment on blowup formula}
We note that the consequences of \cref{cor: adjunction type constraint for surfaces with positive intersection} in \cref{ex: wall-crossing collection of surfaces,ex: wall-crossing collection of surfaces with coefficient 2,ex: wall-crossing collection of surfaces times coefficient 2} are generalizations of the following result obtained from  the blowup formula.
For simplicity we here focus our attention on the case of \cref{ex: wall-crossing collection of surfaces}.
Consider the wall-crossing collection of surfaces $\bSi$ in \cref{ex: wall-crossing collection of surfaces} obtained from the connected sum of algebraic curves and the exceptional curves.
Suppose that  there exists $\epsilon_{i} \in \{+,-\}$ such that $d_{\epsilon_{i},i} = 2$ for each $i \in \{1, \ldots,k\}$.
Let $\bS$ be a bounding collection of surfaces.
Assume that $S \cap \Si_{i}^{\epsilon_{i}} = \emptyset$ $(1\leq i \leq k)$ holds for some surface $S \in \bS$.
For example, this assumption is satisfied in the both case that $\bS$ consists of a single surface and that $\bS$ consists of two surfaces as in \cref{ex: infinitely many bounding} and $k=2$.
Then the adjunction inequality for $S$ follows from a surgery argument and the blowup formula for the usual \SW invariant. 
(This note on the blowup formula is due to Kouichi Yasui.)
The surgery argument, however, does not work in the case that $d_{\pm,i} > 2$ holds for any $i \in \{1, \ldots, k\}$, and therefore the constraint obtained from \cref{cor: adjunction type constraint for surfaces with positive intersection} on elements in $\bS$ cannot be shown using the blowup formula in this case.

More precisely, the surgery argument above is used to construct a negative definite closed $4$-manifold from $N$.
In the case that $d_{\pm,i} > 2$ holds for any $i \in \{1, \ldots, k\}$,
a closed $4$-manifold cannot be obtained from usual surgeries since they can be used only for spheres.
One cannot therefore use the usual blowup formula in this case.
Nevertheless, we can use the combination of wall-crossing and gluing arguments due to Ruberman~\cite{MR1671187, MR1734421, MR1874146};
information about the \SW equations on $M$ survives on the connected sum $X = M \# N$.
This suggests a possibility that one can extend the notion of ``negative definite closed $4$-manifolds" from the gauge theoretic point of view:
in our case, the complement of neighborhoods of suitable embedded surfaces in $N$ behaves as if it were a negative definite closed $4$-manifold.
\end{rem}

\subsection{Application to bounds on the complexity of configurations of surfaces}
\label{subsection: Application to bounds on the complexity of configurations of surfaces}

In this subsection we give a concrete application of \cref{cor: adjunction type constraint for surfaces with positive intersection} to give bounds on the complexity of configurations of surfaces.
The basic idea of this subsection is due to Strle~\cite{MR2064429}, in particular Section~15 of \cite{MR2064429}.

Let $(M, \fraks_0)$ be a smooth oriented closed $\spc$ $4$-manifold with $b^+(M) \geq 2$ and $d(\fraks_0) = 0$, and $k, l$ be natural numbers with $l \geq k > 0$.
We set $N = k\CP^2 \# l(-\CP^2)$ and $(X, \fraks) = (M \# N, \fraks_0\#\frakt)$, where $\frakt$ is the $\spc$ structure defined by \eqref{eq: first Chern class defining property} on $N$.

First, we give a bound on the complexity in the case that a bounding collection of surfaces consists of a single surface.
Recall that, for embedded surfaces $S, \Si$ in $X$, if $[S] \cdot [\Si]=0$ and $S$ intersects transversally with $\Si$, geometric intersection points between $S$ and $\Si$ appears as pairs of $\pm1$-intersection points.
Denote by $I(S,\Si)$ the number of the pairs of $\pm1$-intersection points between $S$ and $\Si$.
For several embedded surfaces $S_{1}, \ldots, S_{n}$,
we say that $S_{1}, \ldots, S_{n}$ are in general position if 
\begin{itemize}
\item for each $i,j \in \{1, \ldots, n\}$ with $i \neq j$, either $S_{i}$ intersects with $S_{j}$ transversally or $S_{i} \cap S_{j}=\emptyset$ holds, and
\item for each distinct three numbers $i, j, k \in \{1, \ldots, n\}$, $S_{i} \cap S_{j} \cap S_{k}=\emptyset$ holds.
\end{itemize}

\begin{theo}
\label{theo: bounds on the complexity of configurations}
Suppose that $\SWinv(M, \fraks_{0}) \neq 0$.
Let $\bSi = \{\Si_i^{\epsilon}\}_{1 \leq i \leq k, \epsilon \in \{+,-\}}$ be a wall-crossing collection of surfaces in $(N, \frakt)$.
Let $S \in V(\bar{\calK}_{\geq0}(X))$ be a surface satisfying that $[S] \cdot [\Si_i^{\epsilon}] = 0$ and that $S$ and $\Si_{i}^{\epsilon}$ $(\epsilon \in \{+,-\}, i \in \{1, \ldots, k\})$ are in general position.
Then we have
 \[
\max\{g(S) + I(S; \bSi) -1,0\} \geq \frac{|c_{1}(\fraks) \cdot [S]| + [S]^{2}}{2},
 \]
 where
 \[
 I(S; \bSi) := \sum_{1 \leq i \leq k, \epsilon \in \{+,-\}} I(S,\Si_i^{\epsilon}).
 \]
\end{theo}

\begin{proof}
First, we recall a procedure replacing a pair of $\pm1$-intersection points with the increasing the genus by $1$.
For a pair ${\mathbf p}=\{p_{+}, p_{-}\}$ of $\pm1$-intersection points between $S$ and $\Si_i^{\epsilon}$ for some $(i, \epsilon)$,
we may construct a $3$-dimensional $1$-handle $h^{1}_{\mathbf p}$ as follows.
Take a path $\gamma_{\mathbf p} : [0,1] \to \Si_i^{\epsilon}$ on $\Si_i^{\epsilon}$ from $p_{+}$ to $p_{-}$ so that $\gamma_{\mathbf p}([0,1])$ does not intersect with $\Si_{i'}^{\epsilon'}$ for all $(i', \epsilon')$ with $(i', \epsilon') \neq (i,\epsilon)$.
This condition can be achieved since the definition of general position implies that $S \cap \Si_i^{\epsilon} \cap \Si_{i'}^{\epsilon'} = \emptyset$ if $(i, \epsilon) \neq (i', \epsilon')$.
Fattening $\gamma_{\mathbf p}$ to an extra $2$-dimension in the normal direction in $X$, we obtain a $3$-dimensional $1$-handle $h^{1}_{\mathbf p} \cong D^{1} \times D^{2}$.
More precisely, $h^{1}_{\mathbf p}$ is defined to be (the total space of) the restriction of the normal disk bundle $D(\nu(\Si_i^{\epsilon}))$ of $\Si_i^{\epsilon}$ in $X$ to the arc $\gamma_{\mathbf p}$.
Eliminating neighborhoods $D_{+}^{2}, D_{-}^{2}$ of $p_{+}, p_{-}$ in $S$ and attaching the part of $\del h^{1}_{\mathbf p}$ which corresponds to $D^{1} \times \del D^{2}$, we obtain a surface $S_{\mathbf p}$ satisfying $g(S_{\mathbf p})=g(S)+1$.
In other words, $S_{\mathbf p}$ is the sum of $S\setminus(D_{+}^{2} \cup D_{-}^{2})$ with the normal sphere bundle $S(\nu(\Si_i^{\epsilon}))$ restricted on the arc $\gamma_{\mathbf p}$.
Then $S_{\mathbf p}$ does not intersect with $\Si_i^{\epsilon}$ anymore.
Taking $h^{1}_{\mathbf p}$ thin enough, we may assume that $[S_{\mathbf p}] = [S]$ and 
\begin{align}
\label{eq: a condition on one handle}
h^{1}_{\mathbf p} \cap \Si_{i'}^{\epsilon'} = \emptyset
\end{align}
for all $(i', \epsilon')$ with $(i', \epsilon') \neq (i, \epsilon)$.

Let $\hat{S}$ be the surface obtained by doing the above procedure to all pairs of $\pm1$-intersection points between $S$ and all $\Si_i^{\epsilon}$.
Here, taking the $1$-handles in the construction thin enough, we may assume that all of these $1$-handles do not intersect with each other.
Thus we may assume that $\hat{S}$ is an embedded surface, not immersed.
Note that $\hat{S}$ satisfies $[\hat{S}] = [S]$ and $g(\hat{S}) = g(S) + I(S;\bSi)$ by construction.
By condition \eqref{eq: a condition on one handle}, $\hat{S}$ does not intersect with any $\Si_i^{\epsilon}$, namely $\hat{S} \cap \bSi = \emptyset$.
Because of the general position assumption, $\Si_i^+$ and $\Si_i^-$ intersect transversally for each $i$.
Therefore, from \cref{cor: adjunction type constraint for surfaces with positive intersection}, we obtain that 
\begin{align*}
2\max\{g(S) + I(S; \bSi) -1,0\}
= \chi^{-}(\hat{S})
\geq |c_1(\fraks) \cdot [\hat{S}]| + [\hat{S}]^{2}
= |c_1(\fraks) \cdot [S]| + [S]^{2}.
\end{align*}
This proves the \lcnamecref{theo: bounds on the complexity of configurations}.
\end{proof}

\begin{ex}
\label{ex: K3 application1}
Let $M=K3$ be the underlying smooth $4$-manifold of a $K3$ surface.
Let $N = k\CP^2 \# l(-\CP^2) = \#_{i=1}^k \CP^2_{i} \# (\#_{j=1}^l (-\CP^2))$, where $k,l$ are constrained
as in \cref{ex: wall-crossing collection of surfaces}.
Let $\frakt$ be the $\spc$ structure on $N$ determined by \eqref{eq: first Chern class defining property}.
Let $\bSi = \{\Si_i^{\epsilon}\}_{1 \leq i \leq k, \epsilon \in \{+,-\}}$ be the wall-crossing collection of surfaces in $(N, \frakt)$ given in \cref{ex: wall-crossing collection of surfaces}.
Set $X=M\#N$.
Recall that the intersection form of $M$ can be written as $2(-E_8) \oplus 3 \left(\begin{smallmatrix}0&1\\1&0\end{smallmatrix}\right)$ with respect to a suitable basis of $H_{2}(M;\Z)$.
Let
\[
\xi = \left( 0_{2(-E_{8})}, \left(p_{i}, q_{i}\right)_{i=1}^{3}, 0_{N} \right)
\]
be a homology class in $H_{2}(X;\Z) = H_{2}(M;\Z) \oplus H_{2}(N;\Z)$ which is trivial except over the part of the basis corresponding to $3 \left(\begin{smallmatrix}0&1\\1&0\end{smallmatrix}\right)$.
Assume that $p_{i}, q_{i} > 0$.
Then we have $\xi^{2} = 2\sum_{i}p_{i}q_{i} > 0$.
Let $S$ be a smooth representative of $\xi$ such that 
$S$ and $\Si_{i}^{\epsilon}$ $(\epsilon \in \{+,-\}, i \in \{1, \ldots, k\})$ are in general position.
Then we have
\begin{align}
\label{eq: estimate on the complexity in the first ex}
g(S) + I(S; \bSi)
\geq \sum_{i=1}^{3} p_{i}q_{i}+1.
\end{align}
To deduce \eqref{eq: estimate on the complexity in the first ex}, let us choose the canonical $\spc$ structure on $M$ as $\fraks_{0}$.
Then we have $\SWinv(M,\fraks_{0}) \neq 0$.
It follows that $c_{1}(\fraks) \cdot [S] =0$ since $c_{1}(\fraks_{0})=0$.
Thus we obtain \eqref{eq: estimate on the complexity in the first ex} from \cref{theo: bounds on the complexity of configurations}.

Sicne $\CP^{2}\#2(-\CP^{2}) \cong S^{2} \times S^{2} \# (-\CP^{2})$,
by Theorem~3 in \cite{MR163323} by Wall, there exists a representative $S$ of $\xi$ with $g(S)=0$ provided that $\xi$ is primitive.
Note that, because of the adjunction inequality for $M=K3$, we cannot take such $S$ inside the direct summand $M$ in $X$.
The inequality \eqref{eq: estimate on the complexity in the first ex} implies that $S$ with $g(S)=0$ must have intersection points with $\Si_i^{\epsilon}$'s whose number is estimated by
\begin{align*}
I(S; \bSi)
\geq \sum_{i=1}^{3} p_{i}q_{i}+1.
\end{align*}
\end{ex}

Next, we shall give an example of a bounding collection of surfaces consisting of two surfaces.
To make the following discussion simple, let us consider $N = 2\CP^2 \# l(-\CP^2)$ and the configuration 
given in the right picture in \cref{figure : Some bounding for KbSi}.
For a wall-crossing collection of surfaces $\bSi$, a surface $\Si \in \bSi$, and another embedded surface $S$,
if $[S] \cdot [\Si']=0$ for all $\Si' \in \bSi \setminus \{\Si\}$, let us define
\[
I(S;\bSi \setminus \{\Si\}) := \sum_{\Si' \in \bSi \setminus \{\Si\}} I(S,\Si').
\]

\begin{theo}
\label{theo: bounds on the complexity of configurations consisting of two surfaces}
Suppose that $\SWinv(M, \fraks_{0}) \neq 0$.
Let $\bSi = \{\Si_i^{\epsilon}\}_{1 \leq i \leq 2, \epsilon \in \{+,-\}}$ be a wall-crossing collection of surfaces in $(N, \frakt)=(2\CP^2 \# l(-\CP^2), \frakt)$.
Let $S,S' \in V(\bar{\calK}_{\geq0}(X))$ be surfaces satisfying that
\begin{align*}
[S] \cdot [\Si_2^{\pm}] = [S] \cdot [\Si_1^{+}] = 0,\ 
[S'] \cdot [\Si_2^{\pm}] = [S'] \cdot [\Si_1^{-}] = 0,\ {\rm and}\ 
[S] \cdot [S'] = 0
\end{align*}
and that $S$, $S'$, $\Si_{1}^{\pm}$, and $\Si_{2}^{\pm}$ are in general position.
Assume that we have
\begin{align}
\label{eq: violating for S}
\chi^{-}(S) + 2I(S;\bSi \setminus {\{\Si_{1}^{-}\}}) &< |c_{1}(\fraks) \cdot [S]| + [S]^{2},\\
\label{eq: violating for S'}
\chi^{-}(S') + 2I(S;\bSi \setminus {\{\Si_{1}^{+}\}}) &< |c_{1}(\fraks) \cdot [S']| + [S']^{2}.
\end{align}
Then the inequality
\begin{align}
\label{eq: ineq for theo: bounds on the complexity of configurations consisting of two surfaces}
&\min\left\{ \max\{g(S)-1,0\} + g(S'), \max\{g(S')-1,0\}  + g(S)\right\} + I(S,S')\\
\geq& \frac{|c_{1}(\fraks) \cdot [S]|+[S]^{2} + |c_{1}(\fraks) \cdot [S']|+[S']^{2}}{2}\nonumber
 \end{align}
 holds.
\end{theo}

\begin{proof}
Let $\breve{S}$ be a surface obtained from $S$ so that all pairs of $\pm1$-intersection points with $\Si_{1}^{+}, \Si_{2}^{\pm}$ are replaced with $1$-handles via the procedure in the first paragraph in the proof of \cref{theo: bounds on the complexity of configurations}.
Similarly, let $\breve{S'}$ be a surface obtained from $S'$ so that all intersection points with $\Si_{1}^{-}, \Si_{2}^{\pm}$ are resolved.
Then we have
\[
g(\breve{S}) = g(S) + I(S;\bSi \setminus \{\Si_{1}^{-}\})
\]
and 
\[
g(\breve{S'}) = g(S') + I(S';\bSi \setminus \{\Si_{1}^{+}\}).
\]
Therefore, by the assumptions \eqref{eq: violating for S}, \eqref{eq: violating for S'}, 
both of $\breve{S}$ and $\breve{S'}$ violate the adjunction inequalities, namely
\begin{align*}
\chi^{-}(\breve{S})&< |c_{1}(\fraks) \cdot [S]| + [S]^{2},\\
\chi^{-}(\breve{S'}) &< |c_{1}(\fraks) \cdot [S']| + [S']^{2}.
\end{align*}
Now we argue precisely as in the proof of Theorem~15.1 in \cite{MR2064429}
First, note that $\breve{S}$ and $\breve{S'}$ have intersection points.
This is deduced as follows.
If $\breve{S}$ and $\breve{S'}$ are disjoint, since we have made these surfaces so that there do not exist intersection points with $\Si_{i}^{\epsilon}$'s, the configuration of $\breve{S}$, $\breve{S'}$, and $\Si_{i}^{\epsilon}$'s are as in the right picture in \cref{figure : Some bounding for KbSi}.
Then, as noted in \cref{ex: infinitely many bounding}, 
$\{\breve{S}, \breve{S'}\}$ is a bounding collection of surfaces for $\bSi$.
Since both of $\breve{S}$ and $\breve{S'}$ violate the adjunction inequalities, by \cref{cor: adjunction type constraint for surfaces with positive intersection}, $\breve{S}$ and $\breve{S'}$ must intersect.

Let us attach a $1$-handle $h^{1}_{1}$ to $S$ to eliminate one pair of $\pm1$-intersection points between $S$ and $S'$.
If the resulting surface $\breve{S} \cup h^{1}_{1}$ still violates the adjunction inequality, $\breve{S} \cup h^{1}_{1}$ and $\breve{S}'$ have intersection points by \cref{cor: adjunction type constraint for surfaces with positive intersection} again,
and therefore we may repeat this procedure and obtain a new surface $\breve{S} \cup h^{1}_{1} \cup h^{1}_{2}$.
Continuing this procedure until when the resulting surface satisfies the adjunction inequality,
we deduce that
there exists uniquely $k \in \{1, \ldots, I(S,S')\}$ such that
\begin{align}
\label{eq: h1k handle}
\chi^{-}(\breve{S} \cup h^{1}_{1} \cup \cdots \cup h^{1}_{k})
&\geq |c_{1}(\fraks) \cdot [S]| + [S]^{2},\\
\label{eq: h1k1 handle}
\chi^{-}(\breve{S} \cup h^{1}_{1} \cup \cdots \cup h^{1}_{k-1})
&< |c_{1}(\fraks) \cdot [S]| + [S]^{2}.
\end{align}
Define $\hat{S} := \breve{S} \cup h^{1}_{1} \cup \cdots \cup h^{1}_{k-1}$, and define $\hat{S}'$ as the surface obtained by attaching the all remaining handles to eliminate the all remaining intersection points:
$\hat{S}' = \breve{S}' \cup h^{1}_{k} \cup \cdots \cup h^{1}_{I(S,S')}$.
Note that $\hat{S}$ satisfies
\begin{align}
\label{eq: h1k handle2}
\chi^{-}(\hat{S})
&\geq |c_{1}(\fraks) \cdot [S]| + [S]^{2}-2
\end{align}
because of \eqref{eq: h1k handle}.
By \eqref{eq: h1k1 handle} and \cref{cor: adjunction type constraint for surfaces with positive intersection},
$\hat{S}'$ satisfies the adjunction inequality:
\[
\chi^{-}(\hat{S}') \geq |c_{1}(\fraks) \cdot [S']| + [S']^{2}.
\]
This inequality and the inequality \eqref{eq: h1k handle2}
imply that
\begin{align}
\label{eq: chichiprime adj}
\chi^{-}(\hat{S})+\chi^{-}(\hat{S}')
\geq |c_{1}(\fraks) \cdot [S]| + [S]^{2} + |c_{1}(\fraks) \cdot [S']| + [S']^{2}-2.
\end{align}
Let us calculate the left-hand side of this inequality.
First, since $g(S')-1 + I(S,S') -(k-1) \geq 0$, we have that
\begin{align}
\label{eq: chichiprime adj3}
&\frac{\chi^{-}(\hat{S})+\chi^{-}(\hat{S}')}{2}\\
=& \max\{g(S)-1 + k-1,0\}
+ g(S')-1 + I(S,S') -(k-1)\nonumber\\
=&\max\{g(S)-1, -(k-1)\}
+ g(S') -1 +I(S,S')\nonumber\\
\leq &\max\{g(S)-1, 0\}
+ g(S') -1 +I(S,S').\nonumber
\end{align}
It follows from the inequalities \eqref{eq: chichiprime adj} and \eqref{eq: chichiprime adj3} that
\begin{align}
\max\{g(S)-1,0\}
+ g(S') + I(S,S')
\geq \frac{|c_{1}(\fraks) \cdot [S]| + [S]^{2} + |c_{1}(\fraks) \cdot [S']| + [S']^{2}}{2}.
\label{eq: chichiprime adj4}
\end{align}

Repeating the argument in the last paragraph under swapping the roles of $S$ and $S'$, we obtain
\begin{align}
\max\{g(S')-1,0\}
+ g(S) + I(S,S')
\geq \frac{|c_{1}(\fraks) \cdot [S]| + [S]^{2} + |c_{1}(\fraks) \cdot [S']| + [S']^{2}}{2}.
\label{eq: chichiprime adj5}
\end{align}
Then the desired inequality \eqref{eq: ineq for theo: bounds on the complexity of configurations consisting of two surfaces} is obtained combining \eqref{eq: chichiprime adj4} and \eqref{eq: chichiprime adj5}.
\end{proof}

\begin{ex}
As well as \cref{ex: K3 application1}, let $M=K3$ be the underlying smooth $4$-manifold of a $K3$ surface.
Let $N_{1}$ be a copy of $\CP^2 \# l_{1}(-\CP^2)$,
where $l_{1}$ is constrained
as in \cref{ex: wall-crossing collection of surfaces}.
Let $\frakt_{1}$ be the $\spc$ structure on $N_{1}$ determined by \eqref{eq: first Chern class defining property} for $k=1$.
Let $\bSi_{1} = \{\Si_1^{+}, \Si_{1}^{-}\}$ be the wall-crossing collection of surfaces in $(N_{1}, \frakt_{1})$ given in \cref{ex: wall-crossing collection of surfaces}.
As well as \cref{ex: K3 application1}, let
\begin{align*}
\xi &= \left( 0_{2(-E_{8})}, \left(p_{1}, q_{1},0,0,0,0\right), 0_{N} \right),\\
\xi' &= \left( 0_{2(-E_{8})}, \left(0,0,p_{2}, q_{2},0,0\right), 0_{N} \right)
\end{align*}
be homology classes in $H_{2}(M\#N_{1};\Z)$ which are trivial except over the part of the basis corresponding to the first and second direct summands of $3\left(\begin{smallmatrix}0&1\\1&0\end{smallmatrix}\right)$ respectively.
Assume that $p_{i}, q_{i} > 0$ and that $\xi$ and $\xi'$ are primitive.
Applying Theorem~3 in \cite{MR163323} by Wall to $M\# N_{1} \cong M\# (S^{2} \times S^{2}) \#(-\CP^{2})$,
we may take representatives $S$ and $S'$ in $M\#N_{1}$ of $\xi$ and $\xi'$ 
so that $g(S) = g(S') = 0$.
Assume that $S,S',\Si_{1}^{+}, \Si_{1}^{+}$ are in general position.
We shall deduce from \cref{theo: bounds on the complexity of configurations consisting of two surfaces} that at least one of the following three inequalities holds:
\begin{align}
\label{eq: at least one of three}
\begin{cases}
I(S,\Si_{1}^{+}) \geq 2p_{1}q_{1},\\
I(S',\Si_{1}^{-}) \geq 2p_{2}q_{2},\\
I(S,S') \geq p_{1}q_{1} + p_{2}q_{2}.
\end{cases}
\end{align}

Before starting to deduce this constraint, it is interesting to compare it with the result obtained in \cref{ex: K3 application1}.
\cref{ex: K3 application1} ensures that {\it both} of the following two inequalities hold:
\begin{align}
\label{eq: both of two}
\begin{cases}
I(S; \{\Si_{1}^{+}, \Si_{1}^{-}\}) \geq p_{1}q_{1}+1,\\
I(S'; \{\Si_{1}^{+}, \Si_{1}^{-}\}) \geq p_{2}q_{2}+1.
\end{cases}
\end{align}
Obviously the condition that both of the inequalities in \eqref{eq: both of two} holds does not imply the condition that at least one of the inequalities in \eqref{eq: at least one of three} holds, and vice versa.
These two constraints are obtained from \cref{theo: bounds on the complexity of configurations} and \cref{theo: bounds on the complexity of configurations consisting of two surfaces} respectively,
and the main difference between these two theorems are choices of bounding collections of surfaces:
\cref{theo: bounds on the complexity of configurations} involves a bounding collection consisting of a single surface, and 
\cref{theo: bounds on the complexity of configurations consisting of two surfaces} involves a bounding collection consisting of two surfaces.
Therefore these two different constraints exhibit the freedom of choice of bounding collections provides several different geometric constraints.

Now let us deduce that at least one of the inequalities in \eqref{eq: at least one of three} holds.
To see this, it is enough to show that we have the last inequality in \eqref{eq: at least one of three} provided that both of
\begin{align}
\label{eq: two of three hold}
\begin{cases}
I(S,\Si_{1}^{+}) < 2p_{1}q_{1},\\
I(S',\Si_{1}^{-}) < 2p_{2}q_{2}
\end{cases}
\end{align}
hold.
Let $N_{2}$ be a copy of $\CP^2 \# l_{2}(-\CP^2)$,
where $l_{2}$ is constrained
as in \cref{ex: wall-crossing collection of surfaces}.
Set $N = N_{1} \# N_{2}$ and $X=M\#N$.
Let $\frakt$ be the $\spc$ structure on $N$ determined by \eqref{eq: first Chern class defining property} for $k=2$, and $\fraks_{0}$ be the canonical $\spc$ structure on $M$, i.e. the spin$^{c}$ structure coming from the unique spin structure on $M$.
Adding new surfaces $\{\Si_{2}^{+}, \Si_{2}^{-}\}$ in $N_{2}$ to $\bSi_{1}$, we obtain a wall-crossing collection of surfaces $\bSi=\{\Si_{1}^{+}, \Si_{1}^{-}, \Si_{2}^{+}, \Si_{2}^{-}\}$ in $(N, \frakt)$ as in \cref{ex: wall-crossing collection of surfaces}.
Since $S, S'$ are inside $M\# N_{1}$, they do not intersect with $\Si_{2}^{\pm}$.
Therefore we have
\[
I(S,\Si_{1}^{+}) = I(S; \bSi \setminus \{\Si_{1}^{-}\}),
\quad I(S',\Si_{1}^{-}) = I(S'; \bSi \setminus \{\Si_{1}^{+}\}).
\]
This and \eqref{eq: two of three hold} imply that $S$ and $S'$ satisfy the all assumptions of \cref{theo: bounds on the complexity of configurations consisting of two surfaces}.
Therefore it follows from \cref{theo: bounds on the complexity of configurations consisting of two surfaces} that the last inequality in \eqref{eq: at least one of three} holds.
\end{ex}

\section{Concluding remarks}

Finally, we remark some further potential developments of the invariant $\SWcoh(X, \fraks)$.

\begin{rem}
For a closed $\spc$ $4$-manifold $(X, \fraks)$, one can define a subcomplex $\calK_{\geq0}(X, \fraks)$ of $\bar{\calK}_{\geq0}(X)$ by considering all embedded surfaces violating the adjunction inequalities with respect to $\fraks$ and having non-negative self-intersection number.
The author expects that we can construct an invariant which is formulated as a cohomology class on $\calK_{\geq0}(X, \fraks)$ in a similar vein of the construction of $\SWcoh(X, \fraks)$.
There are two differences between $\calK_{\geq0}(X, \fraks)$ and $\calK(X, \fraks)$ from the point of view of the construction of a cohomological invariant.
The first one is on the stretching construction of families of Riemannian metrics in \cref{subsection: Construction of a family of Riemannian metrics}, and the second is on the quantitative estimate given as \cref{prop: vanishing on certain face} for the length of stretched cylinders to assure the non-existence of solutions to the \SW equations.
For the first difference, the stretching argument can be replaced with the stretching of metrics for the fiber-direction of the normal bundles of embedded surfaces in general. 
However, for the second difference, the author does not know any quantitative estimate as in \cref{prop: vanishing on certain face} for surfaces of positive self-intersection number at this stage.
\end{rem}

\begin{rem}
Bauer--Furuta~\cite{MR2025298} have given a refinement, called the Bauer--Furuta invariant, of the usual \SW invariant.
This is defined as the stable cohomotopy class of a finite-dimensional approximation of the \SW equations.
The author expects that we may define such a refinement of our invariant $\SWcoh(X, \fraks)$.
If we can define such an invariant, we can expect a generalization of \cref{theo: non-vanishing theorem}: we might replace $(M, \fraks_{0})$ in \cref{theo: non-vanishing theorem} with a $\spc$ $4$-manifold whose Bauer--Furuta invariant does not vanish.
Since there are $\spc$ $4$-manifolds such that their Bauer--Furuta invariants do not vanish but their \SW invariants do vanish,
this expected non-vanishing theorem will provide new constraints on configurations of embedded surfaces as in \cref{cor: the adjunction inequality for infinitely many configurations,cor: adjunction type constraint for surfaces with positive intersection}.
\end{rem}

\begin{rem}
We have considered only closed $4$-manifolds in this paper.
However, in general, one can expect that a gauge theoretical invariant of closed $4$-manifolds is generalized to an invariant of $4$-manifolds with boundary using the Floer homology theory.
The author conjectures that our invariant $\SWcoh(X, \fraks)$ can be also generalized for $4$-manifolds with boundary.
The first thing we have to do is to define the complex of surfaces for $4$-manifolds with boundary;
two candidates for it might be considered.
The first one is the simplicial complex obtained by considering closed surfaces embedded into the interior of the given $4$-manifold, and the second is that obtained by considering surfaces with boundaries embedded into the given $4$-manifold such that the boundary of each surface is in that of the $4$-manifold.
The author expects that $\SWcoh(X, \fraks)$ can be generalized to a cohomology class at least on the first candidate.
On the other hand, he also expects that, if $\SWcoh(X, \fraks)$ can be generalized to the second candidate, it might relate to some variant of the Kakimizu complex~\cite{MR1177053} in $3$-dimensional topology and knot theory.
\end{rem}

\appendix

\section{Simple properties of complex of surfaces}

In this appendix, we prove that the complex of surfaces $\bar{\calK}(X)$ is contractible for any 4-manifold $X$.

\begin{lem}
\label{lem : cone bouding}
Let $K = \bar{\calK}$ or $\calK$.
Let $z = \sum_{k=1}^m a_k \lb \si_k \rb \in C_{n}(K)$ be an $n$-cycle, where $n \geq 1$, $m \geq 1$, $a_k \in \Z$ and $\si_k \in S_{n}(K)$.
If there exists $\Si \in V(K)$ such that $\Si \notin \si_k$ holds for any $k$, then
\[
[z] = 0\ {\it in}\ H_{n}(K;\Z)
\]
holds.
\end{lem}

\begin{proof}
Let us define an oriented $(n+1)$-simplex $\lb \tau_k \rb$ of $K$ by
$\lb \tau_k \rb := \lb \Si, \si_k \rb$.
More precisely, when we write $\lb \si_k \rb = \lb \Si_{0,k}, \ldots, \Si_{n,k} \rb$,
we define
\[
\lb \tau_k \rb := \lb \Si, \Si_{0,k}, \ldots, \Si_{n,k}  \rb.
\]
Then we obtain
\begin{align*}
\del(\sum_{k=1}^m a_k \lb \tau_k \rb )
= z + \sum_{k=1}^m a_k \sum_{l=0}^n (-1)^{l+1} \lb \Si, \Si_{0,k}, \ldots, \hat{\Si}_{l,k} ,\ldots,\Si_{n,k} \rb.
\end{align*}
We note an isomorphism between free modules
\begin{align*}
&\Z \Set{ \lb \Si_{0,k}, \ldots, \hat{\Si}_{l,k} ,\ldots,\Si_{n,k} \rb | 1 \leq k \leq m,\ 0 \leq l \leq n }\\
&\to \Z \Set{ \lb \Si, \Si_{0,k}, \ldots, \hat{\Si}_{l,k} ,\ldots,\Si_{n,k} \rb | 1 \leq k \leq m,\ 0 \leq l \leq n }
\end{align*}
defined by
\begin{align*}
\lb \Si_{0,k}, \ldots, \hat{\Si}_{l,k} ,\ldots,\Si_{n,k} \rb
\mapsto \lb \Si, \Si_{0,k}, \ldots, \hat{\Si}_{l,k} ,\ldots,\Si_{n,k} \rb.
\end{align*}
The equality $\del z=0$ in the first module hence implies that $\del(\sum_{k=1}^m a_k \lb \tau_k \rb ) = z$.
\end{proof}

\begin{prop}
\label{prop : any complex of surface is contractible}
The simplicial complex $\bar{\calK} = \bar{\calK}(X)$ is contractible for any $X$.
\end{prop}

\begin{proof}
We show that $\pi_{n}(\bar{\calK}) = 0$ for any $n \geq 0$.
For any continuous map $f : S^n \to |\bar{\calK}|$, there exist finite simplices $\sigma_1, \ldots, \sigma_m$ such that the image of $f$ is contained in the union of the geometric realizations of these simplices.
By the argument in the proof of \cref{lem : cone bouding}, the union $\bigcup_{k=1}^m |\sigma_k|$ is contained in its cone $C(\bigcup_{k=1}^m |\sigma_k|) \subset |\bar{\calK}|$, and the map $f$ can be deformed into the constant map to the vertex point of the cone in $\bar{\calK}$.
Hence we have $\pi_{n}(\bar{\calK}) = 0$, and $\bar{\calK}$ is contractible by the J. H. C. Whitehead theorem.
\end{proof}

\begin{rem}
For an embedded surface $\Si \subset X$ with self-intersection number zero, by pushing $\Si$ into the fiber direction of its normal bundle, we obtain an embedded surface whose homology class coincides with $\Si$'s and which has no geometric intersections with $\Si$.
We call such a surface a {\it parallel copy} of $\Si$.
The complex $\calK$ is a very huge space, however, homology group absorbs the hugeness in the following sense.
For $n \geq 0$, let us take a cycle $z \in C_{n}(\calK)$.
Let $z' \in C_{n}(\calK)$ be a cycle obtained by replacing a surface which is a part of $z$ with a parallel copy of the surface.
Then $[z] = [z']$ in $H_{n}(\calK;\Z)$ holds.

To see this, we write $z = \sum_{k=1}^m a_k \lb \si_k \rb$, where $m \geq 1$, $a_k \in \Z$ and $\si_{k} \in S_{n}(\calK)$.
We also write 
$\lb \si_k \rb
= \lb \Si_{0,k}, \ldots, \Si_{n, k} \rb$.
Let $\Si_{0,1}$ be the surface which is replaced with its parallel copy in the definition of $z'$, and $\Si_{0,1}'$ be the parallel copy of $\Si_{0,1}$.
We define an oriented simplex  $\lb \si_k' \rb$ as follows.
If $\Si_{0,1} \in \si_{k}$ holds, we define $\lb \si_k' \rb$ as the simplex obtained by replacing $\Si_{0,1}$ with $\Si_{0,1}'$ from $\lb \si_k \rb$.
If $\Si_{0,1} \notin \si_{k}$ holds, we define $\lb \si'_k \rb = \lb \si_k \rb$.
Then $z' = \sum_{k=1}^m a_k \lb \si_k' \rb$ holds.
By changing the order of indices,  we may assume that
\begin{align*}
1 \leq k \leq m' &\Rightarrow \lb \si'_{k} \rb \neq \lb \si_{k} \rb,\ {\rm and}\\
m'+1 \leq k \leq m &\Rightarrow \lb \si'_{k} \rb = \lb \si_{k} \rb
\end{align*}
hold for some $m'$.
In addition, by changing the sign of $a_k$, we may assume that $\Si_{0, k} = \Si_{0,1}$ for each $k \in \{1, \ldots, m'\}$.
We set
\[
\lb \tau_k \rb
:= \lb \Si_{0,1}', \Si_{0,k}, \Si_{1,k}, \ldots, \Si_{n,k} \rb
= \lb \Si_{0,1}', \Si_{0,1}, \Si_{1,k}\ldots, \Si_{n,k} \rb
\]
for each $k \in \{1, \ldots, m'\}$.
Then one can check that
$\del(\sum_{k=1}^{m'} a_k \lb \tau_k \rb) = z-z'$.

\end{rem}

\end{document}